\documentclass[mathpazo]{cicp}
\usepackage[margin=1in]{geometry}
\usepackage{amsmath,epsf,cite}
\usepackage{amssymb,amsthm,tocvsec2}
\usepackage{graphicx}
\usepackage{epsfig,epsf,latexsym,subfigure}
\usepackage{float}
\usepackage{color}
\usepackage{mathrsfs}
\usepackage{indentfirst}

\usepackage[width=.75\textwidth]{caption}

\date{ }

\numberwithin{equation}{section}
\numberwithin{figure}{section}
\numberwithin{table}{section}

\theoremstyle{plain}
\newtheorem{thm}{Theorem}[section]

\theoremstyle{remark}
\newtheorem{rem}{Remark}[section]

\newcommand{\bu}{{\mathbf u}}

\newcommand{\be}{\begin{equation}}
\newcommand{\ee}{\end{equation}}

\newcommand{\bse}{\begin{subequations}}
\newcommand{\ese}{\end{subequations}}

\def\bh{\mathbf{h}}

\def\bu{\mathbf{u}}
\def\bW{\mathbf{W}}
\def\bn{\mathbf{n}}

\def\be{\mathbf{e}}

\def\bx{\mathbf{x}}
\def\ba{\mathbf{a}}

\def\bI{\mathbf{I}}

\def\bD{\mathbf{D}}

\def\half{\frac{1}{2}}

\def\bp{\mathbf{p}}
\def\bG{\mathbf{G}}

\def\bu{\mathbf{u}}
\def\bv{\mathbf{v}}
\def\bw{\mathbf{w}}

\def\bH{\mathbf{H}}

\def\cL{\mathcal{L}}

\def\cG{\mathcal{G}}
\def\cB{\mathcal{B}}
\def\cA{\mathcal{A}}
\def\cN{\mathcal{N}}
\def\cE{\mathcal{E}}

\def\half{\frac{1}{2}}

\newcommand{\ben}{\begin{eqnarray}}
\newcommand{\een}{\end{eqnarray}}
\newcommand{\beq}{\begin{equation}}
\newcommand{\eeq}{\end{equation}}
\newcommand{\bea}{\begin{array}}
\newcommand{\eea}{\end{array}}
\newcommand{\bef}{\begin{figure}}
\newcommand{\eef}{\end{figure}}

\newtheorem{scheme}{Scheme}[section]

\begin{document}
\title{A General Framework to Derive Linear, Decoupled and Energy-stable Schemes for Reversible-Irreversible Thermodynamically Consistent Models: Part I Incompressible Hydrodynamic Models}
\author[J. Zhao]{
Jia Zhao\affil{1}\comma \corrauth}
\address{\affilnum{1}\ Department of Mathematics \& Statistics, Utah State University, Logan, UT, USA }
\email{ {\tt jia.zhao@usu.edu.} (J.~Zhao)}

\begin{abstract}
In this paper, we present a general numerical platform for designing accurate, efficient, and stable numerical algorithms for incompressible hydrodynamic models that obeys the thermodynamical laws. The obtained numerical schemes are automatically linear in time. It decouples the hydrodynamic variable and other state variables such that only small-size linear problems need to be solved at each time marching step. Furthermore, if the classical velocity projection method is utilized, the velocity field and pressure field can be decoupled. In the end, only a few elliptic-type equations shall be solved in each time step. This strategy is made possible through a sequence of model reformulations by fully exploring the models' thermodynamic structures. The generalized Onsager principle directly guides these reformulation procedures. In the reformulated but equivalent models, the reversible and irreversible components can be identified, guiding the numerical platform to decouple the reversible and irreversible dynamics. This eventually leads to decoupled numerical algorithms, given that the coupling terms only involve irreversible dynamics. To further demonstrate the numerical platform's power, we apply it to several specific incompressible hydrodynamic models. The energy stability of the proposed numerical schemes is shown in detail. The second-order accuracy in time is verified numerically through time step refinement tests. Several benchmark numerical examples are presented to further illustrate the proposed numerical framework's accuracy, stability, and efficiency.
\end{abstract}

\ams{}
\keywords{Phase Field; Decoupled Scheme; Energy Stable; Cahn-Hilliard-Navier-Stokes; Hydrodynamics; Liquid Crystal}
\maketitle

\section{Introduction}

Non-equilibrium phenomena are ubiquitous, which require well-developed models to describe their time-dependent dynamics. In contrast to the classical thermodynamic theories for equilibrium systems, paradigms for developing theories for non-equilibrium phenomena have not gained widespread recognition. There have yet been any universally accepted physical laws, analogous to the three fundamental laws in equilibrium thermodynamics,  discovered nor formulated for non-equilibrium thermodynamics. Nevertheless, a plethora of theoretical frameworks has been developed and used to derive new theories or validate the existing theories for non-equilibrium phenomena, which are consistent with the classical thermodynamic theory at equilibrium. 

In the search for a systematic approach rooted in a solid mathematical foundation, two new formalisms that emerged in the last century culminated in a series of seminal papers and a monograph by Beris and Edwards on the Poisson bracket formalism \cite{Beris&EdwardsJR1990A, Beris&EdwardsJR1990B, BerisEdwards1994}, and by Ottinger and Grmela on the GENERIC formalism \cite{OttingerPRE1997A, OttingerPRE1997B, GENERIC-2DNS}, where GENERIC is an acronym for "General Equation for Non-Equilibrium Reversible-Irreversible Coupling." Both approaches established mathematical equations and physical structures for the non-equilibrium models to follow. On another front, Onsager pioneered his linear response theory and reciprocal relation for dissipative thermodynamic systems and developed the variational method using the Onsager-Machlup potential \cite{OnsagerL1, OnsagerL2}.  This method has recently been amplified by the energetic variational approach to non-equilibrium models \cite{LiuChunEV1,LiuChunEV2,LiuChunEV3} and the generalized Onsager principle for more general situations \cite{OnsagerL1,OnsagerL2,Yang&Li&Forest&Wang2016,JoannyNJP2007}.
The matrix formulation applied to viscoelastic fluid models is a simplified version of the generalized Onsager principle \cite{JongschaapJR1994, GENERIC-Matrix-relation}.

All the formulations mentioned above produce thermodynamically consistent models. Given the broader range of applicability and simplicity, we adopt the Onsager formalism for deriving non-equilibrium models in this paper. In formulating non-equilibrium thermodynamic models, the Onsager formalism provides a clear mathematical description for the reversible and irreversible process involved through the mobility operator and the free energy in the isothermal case or the entropy in the non-isothermal case.   The energy (or entropy) and mobility pair delineate the coupling among all thermodynamic variables and dictates that the total entropy production rate is nonnegative. This setting is especially suitable for developing structure-preserving numerical approximations for such non-equilibrium models. 

For a given thermodynamically-consistent PDE system describing non-equilibrium phenomena, a high order, accurate, computationally efficient, and property and structure-preserving discretization are always desirable. In particular, for thermodynamically consistent models, a measure for good or better numerical approximations should always be if the discrete scheme would preserve the physical laws and as much as possible the physical properties at the discrete level. A numerical scheme preserving the continuous model's original mathematical structures and physical properties is called a structure-preserving or geometric-preserving scheme. The latter normally refers to the numerical scheme for a Hamiltonian system. Structure-preserving schemes have had enormous success in solving conservative dynamical systems, most notably the Hamiltonian systems, during the past decades \cite{HairerBook, Structure-Preserving-1, Structure-Preserving-2}. Such structure-preserving approximation not only has a tremendous theoretical value but also has a practical implication. For instance, in one of our early studies \cite{Wang-JCP2011}, we showed that an energy-dissipation-preserving scheme could resolve more details in the flow structure than a non-preserving scheme at the same level of numerical resolution. So, structure-preservation would be one of the attributes that we would like to attain in developing numerical approximations to the thermodynamically consistent models.

Despite the success of Hamiltonian dynamical systems, structure-preserving numerical approximations have not been well developed for thermodynamic and hydrodynamic systems.  This is partly because of the increased complexity in the thermodynamic and hydrodynamic models. More pertinent is perhaps because of the lack of understanding of the mathematical structure of the non-equilibrium models.
Another issue challenging the computational science community is how to deal with nonlinearity in numerical algorithms.
Most of the available structure-preserving schemes in the literature today \cite{thermodynamic-consistent-scheme1,RomeroCMAME2010A,RomeroCMAME2010B,RomeroNME2009,GENERIC-2DNS} are nonlinear and hard to implement. Issues on the solvability of the discrete system, the uniqueness of the discrete solution, error estimates, and time step constraint are difficult to address.  So, a linear scheme would be desirable since it enables a rigorous proof of solution existence and uniqueness using, for example, the Babuska-Lax-Milgram theorem for the discrete system resulted from the linear scheme. Another property of the non-equilibrium model that the numerical scheme should respect is total energy conservation and positive entropy production. The latter corresponds to energy dissipation in the system. Currently, an energy dissipation rate preserving numerical scheme independent of the time-step size is called energy stable in the literature unconditionally. Here, we name the structure-preserving scheme unconditional energy stable if it respects the total energy conservation and positive entropy production regardless of the time-step size.

Inspired by many seminal works \cite{Li&ShenDecoupled2020, YangDecoupled2021,YangDecoupled-1,YangDecoupled-2,Wang&Wang&WiseDCDS2010,Shen&Yang2014SISC, Zhao&Yang&Shen&WangJCP2016,Zhao2018EQreview} ,in this paper, we propose a general framework for designing structure-preserving numerical schemes for thermodynamically consistent models in non-equilibrium dynamics. The resulted numerical schemes from our general framework are linear, high-order-in-time, structure-preserving, energy stable, and easy-to-implement. We particularly focus on the generic formulation of the numerical framework and its applications in the incompressible hydrodynamic models in this paper. The applications of our general numerical platform on other thermodynamically consistent models, such as non-isothermal thermodynamical and hydrodynamical models for complex fluid flows, will be elaborated consequently in our late papers.  Meanwhile, we emphasize that one advantage of our general numerical framework is its systematical formulation. Under its guidance, a computational toolkit with modular code constructs to simplify the numerical implementation of solving the thermodynamically consistent models. 

The rest of this paper is organized as follows. In Section \ref{sec:GOP}, we provide a generic formulation of thermodynamically consistent models using the generalized Onsager principle. Some examples of casting incompressible hydrodynamic models into the generalized form are discussed. In Section \ref{sec:generic-numerical-method}, the generic model is then transformed into an equivalent form by first applying the energy quadratization (EQ) method and then the reversible-irreversible decoupling (RID) method. Some generic numerical schemes for solving the general model are provided. Afterward, in Section \ref{sec:applications}, we apply the EQ-RID method to some widely-used incompressible hydrodynamic models. Specific numerical schemes for these models are elaborated. In Section \ref{sec:results}, we benchmark the proposed framework with numerical examples. In the end, we draw a brief conclusion.

\section{Thermodynamically consistent reversible-irreversible PDE Models based on the generalized Onsager principle} \label{sec:GOP}

\subsection{Generalized Onsager principle}
The generalized Onsager formalism provides a theoretical framework for developing thermodynamically consistent (TC) models describing non-equilibrium phenomena. Many well-known thermodynamically consistent PDE systems are, in fact, derivable from the generalized Onsager principle, including the Navier-Stokes equation, the Fokker Planck (or Smoluchowski) equation, the gradient flow models, the thermodynamically consistent viscoelastic fluid models, non-isothermal hydrodynamic models, etc.
By casting TCPDE models into the generalized Onsager form,  the underlying physical mechanism for reversible and irreversible processes is put on full display \cite{OnsagerL1, OnsagerL2, Yang&Li&Forest&Wang2016, JoannyNJP2007}.
We believe that any physically meaningful dynamical model describing non-equilibrium phenomena must be derived following thermodynamical principles and obey necessary conservation laws. The physical laws include the first law of thermodynamics (energy conservation), the second law of thermodynamics (positive entropy production) or, more generally, the generalized Onsager principle in the linear response regime \cite{OnsagerL1, OnsagerL2}, and additional conservation laws: mass, linear momentum, angular momentum and so on. These laws can also be viewed as constraints imposed on the thermodynamical (and hydrodynamical) variables. In this paper, we refer that models derived from the Onsager principle and subject to the necessary conservation laws thermodynamically consistent.

Consider the domain $\Omega$ and time $t \in (0, T]$, and denote $\Omega_T = \Omega \times (0, T]$, and the state variables as $\Phi=(\phi_1,\cdots,\phi_n)^T$.  We recall the generalized Onsager principle using a generic model as an example \cite{OnsagerL1, OnsagerL2}. It consists of three key ingredients: the state/thermodynamic variables $\Phi$, the free energy $\cE$, and a mobility matrix (or operator) $\cG$, all of which will dictate the kinetic equation, namely
\beq
\mbox{ the Onsager triple: } (\Phi, \cG, \cE).
\eeq 
The kinetic equation, stemming from the Onsager linear response theory,   is given by
\begin{subequations} \label{eq:evolution-general}
\begin{align}
&  \partial_t \Phi (\bx,t) = - \cG \frac{\delta \cE}{\delta \Phi} \mbox{ in } \Omega, \\
&\cB(\Phi(\bx,t)) = g(\bx,t), \mbox{ on } \partial \Omega,
\end{align}
\end{subequations}
where  $\cB$ is a trace operator,  and $\cG$ is the mobility operator that contains two parts:
\beq
\cG=\cG_a+\cG_s.
\eeq
Here $\cG_s$  is symmetric and positive semi-definite that controls the irreversible dynamics, and $\cG_a$  is skew-symmetric that controls reversible dynamics.  $\frac{\delta F}{\delta \Phi}$ is the variational derivative of $\cE$, known as the chemical potential. Then, the Onsager triple $(\Phi,\cG,\cE)$ uniquely defines a thermodynamically consistent model. 

One intrinsic property  of \eqref{eq:evolution-general} owing to the thermodynamical consistency  is the energy dissipation law
\begin{subequations} \label{EDL}
\begin{align}
& \frac{d \cE}{d t} = \Big( \frac{\delta \cE}{\delta \Phi},   \frac{\partial\Phi}{\partial t} \Big)+ \dot{\cE}_{surf} =\dot{\cE}_{bulk}+\dot{\cE}_{surf},\\
&\dot{\cE}_{bulk}=-\Big( \frac{\delta \cE}{\delta \Phi}, \cG_s \frac{\delta F}{\delta \Phi} \Big) \leq 0, \\
& \Big( \frac{\delta \cE}{\delta \Phi}, \cG_a \frac{\delta \cE}{\delta \Phi} \Big) = 0, \quad  \dot{\cE}_{surf}=\int_{\partial \Omega} g_b ds, 
\end{align}
\end{subequations}
where the inner product is defined by
$$
\Big({\bf f}, {\bf g}\Big) = \sum\limits_{i=1}^d \int_\Omega f_i g_i d\bx, \quad \forall f, g \in [L^2(\Omega)]^d, 
$$
and $\dot{\cE}_{surf}$ is due to the boundary contribution, and $g_b$ is the boundary integrand.
When $\cG_a=0$, \eqref{eq:evolution-general} is a purely dissipative system; while $\cG_s=0$, it is a purely dispersive system. $\dot{\cE}_{surf}$ vanishes only for suitable boundary conditions, which include periodic and certain physical boundary conditions.
When the mass, momentum, and total energy conservation are present in hydrodynamic models, these conservation laws are viewed as constraints imposed on the hydrodynamic variables. Then, the energy dissipation rate will have to be calculated subject to the constraints.

\subsection{Thermodynamically Consistent Incompressible Hydrodynamic Models as Constraint Gradient Flow Models} 

This section elaborates that many existing thermodynamically consistent incompressible hydrodynamic models can be written in the form of \eqref{eq:evolution-general}. Though we mainly focus on incompressible hydrodynamic models, we shall emphasize that the proposed numerical framework works on other quasi-incompressible or compressible hydrodynamics models so long that they are thermodynamically consistent.

First of all,  we introduce a few important notations that will help to explain the reformulation procedure. 
Recall the incompressible Navier-Stokes equations
\begin{subequations} \label{eq:NS-equations}
\begin{align}
&\rho (\partial_t \bu + \bu \cdot \nabla \bu) = - \nabla p + \eta \Delta \bu + \mathbf{f}, \\
& \nabla \cdot \bu = 0,
\end{align}
\end{subequations}
with $\bu$ the velocity field, $\rho$ the density, $\mathbf{f}$ is the external force, and $p$ is the pressure.
We emphasize that the pressure $p$ in \eqref{eq:NS-equations} is a Lagrangian multiplier to enforce the incompressibility of the velocity field $\bu$. With this in mind, we can reformulate the incompressible Navier-Stokes equations into a constraint gradient flow form. This will guide us in designing decoupled numerical algorithms.  

Following the notations in \cite{OttingerPRE1997A,OttingerPRE1997B}, let $P_\bu$ be a functional space defined by
\beq
P_\bu = \left\{ \bu(x, t): \bu \in [L^2(\Omega)]^d; \quad \nabla \cdot \bu =0  \mbox{ in } \Omega, \quad \bu = 0 \mbox{ on } \partial \Omega   \right\},
\eeq 
$\Pi_\bu$ denotes a projection operator, defined as
\beq \label{eq:projection-operator}
\Pi_\bu(\ba) =
\left\{
\bea{l}
\ba - \nabla p, \quad \mbox{in } \Omega / \partial \Omega, \\
0 \quad \mbox{ on } \partial \Omega,
\eea 
\right.
\eeq
where $p$ satisfies a Poisson condition with a Neumann-type boundary condition, i.e. 
$$
\left\{
\bea{l}
\Delta p = \nabla \cdot \ba, \quad  \mbox{ in } \Omega, \\
\frac{\partial p}{\partial \bn} = \ba \cdot \bn, \quad  \mbox{ on } \partial \Omega.
\eea 
\right.
$$
Definite the kinetic energy
\beq \label{eq:kinetic-energy}
E(\bu) =\int_\Omega \frac{\rho}{2}|\nabla \bu|^2 d\bx.
\eeq 
With the projection operator in \eqref{eq:projection-operator}, we denote the constraint variational derivative of the kinetic energy in \eqref{eq:energy} with respect to the velocity field as
\beq \label{eq:projection-u}
\frac{\delta E}{\delta \bu} = \Pi_\bu \frac{\partial E}{\partial \bu}.
\eeq

We rewrite the nonlinear convection term into a skew-symmetric form
\beq \label{eq:reformulation-convection}
B(\bv, \bu) = \frac{1}{2} \Big[  \bv \cdot \nabla \bu + \nabla \cdot (\bv \bu) \Big].
\eeq 
In addition, we introduce induce a trilinear form $b$ defined as \cite{Han&WangJCP2015}
\beq
b(\bv,\bu,\bw) = \Big( B(\bv,\bu), \bw \Big) = \frac{1}{2} \left[ \Big( \bv \cdot \nabla \bu, \bw \Big) - \Big( \bv \cdot \nabla \bw, \bu \Big)  \right],  \forall \bu, \bv, \bw \in \bH_0^1(\Omega).
\eeq 
It follows immediately that 
\beq
b(\bv, \bu, \bu) = 0, \quad \forall \bu, \bv \in \bH_0^1(\Omega).
\eeq

\subsection{Casting thermodynamically-consistent incompressible hydrodynamic models into the generalized Onsager form}
We emphasize that many existing thermodynamically consistent incompressible hydrodynamic models can be cast as special cases of the generalized model in \eqref{eq:evolution-general}. In this sub-section, we illustrate it by examples.

\subsubsection{A hydrodynamic model for two phase incompressible fluids}
In this model, we use $\phi(\bx,t) \in [-1, 1]$ as the phase variable, with $\phi(\bx,t)=1$ to label one phase, $\phi(\bx,t)=-1$ to label the other phase, and $\phi \in (-1, 1)$ representing the interface. The Cahn-Hilliard-Navier-Stokes (CHNS) equations are proposed as
\begin{subequations} \label{eq:CHNS}
\begin{align}
&\rho \Big( \partial_t \bu + \bu \cdot \nabla \bu \Big) = - \nabla p + \eta \nabla^2 \bu -   \phi \nabla \mu,  \quad   (\bx, t) \in \Omega_T, \\
&\nabla \cdot \bu  =  0,  \quad  (\bx, t) \in \Omega_T,  \\
&\partial_t \phi +\nabla \cdot (\bu \phi) =  M \Delta   \mu, \quad  (\bx, t) \in \Omega_T, \\
&\mu = -  \varepsilon^2 \Delta \phi +  f'(\phi), \quad (\bx, t) \in \Omega_T, 
\end{align}
\end{subequations}
where $\eta$ is the viscosity parameter, $M$ is the mobility parameter, $p$ is the hydrodynamic pressure, and $\mu = \frac{\delta F}{\delta \phi}$ is the chemical potential. 
The boundary conditions could be
\begin{subequations}
\begin{align}
& \mbox{(i) periodic on } \partial \Omega; \mbox{ or}  \\
& \label{eq:CHNS-boundary}  \bu(\bx, t) = 0, \quad \nabla \mu(\bx, t) \cdot \bn = 0 , \quad \nabla \phi(\bx, t) \cdot \bn =0, \quad (\bx, t) \in \partial \Omega \times (0, T],
\end{align}
\end{subequations}
with $\bn$ the outward normal vector at the boundary.
The total energy of the two phase fluid-mixture system $\mathcal{E}$ include the Helmholtz free energy $F$ and the kinetic energy $E$, i.e.
\beq \label{eq:energy}
\cE(\bu, \phi)  = F(\phi) + E(\bu), \quad F(\phi)= \int_\Omega  \Big( \frac{\varepsilon^2}{2} | \nabla \phi |^2 +  f(\phi)  \Big) d\bx, 
\eeq
where $E(\bu)$ is the kinetic energy defined in \eqref{eq:kinetic-energy}, and
$\varepsilon$ is an artificial parameter controlling the interfacial thickness. $f(\phi)$ is the bulk free energy for the two phase material. $\bu$ the volume-averaged velocity, and $\rho$ is the volume-averaged density. 

The CHNS system  in \eqref{eq:CHNS}-\eqref{eq:CHNS-boundary} is known to satisfy the second law of thermodynamics, with the energy dissipation rate calculated as
\beq \label{eq:energy-law-continous}
\frac{d\cE}{dt} = - \int_\Omega \Big( M|\nabla \mu|^2 + \eta |\nabla \bu|^2  \Big) d\bx.
\eeq 
To be specific, for the Cahn-Hilliard-Navier-Stokes equation in \eqref{eq:CHNS}, it can be written in the form of \eqref{eq:evolution-general}, if we denote 
$\Phi =
\begin{bmatrix}
\bu \\ \phi
\end{bmatrix}
$, and definite the operators
\beq
\cG_a = 
\begin{pmatrix}
-\frac{1}{\rho} B(\bu, \bullet )  & - \frac{1}{\rho} \phi \nabla \bullet  \\
-\frac{1}{\rho} \nabla \cdot (\bullet \phi) & 0 \\
\end{pmatrix}, 
\quad 
\cG_s =
\begin{pmatrix}
\frac{1}{\rho^2} \eta \Delta \bullet & 0 \\
0 & M \Delta  \bullet
\end{pmatrix}.
\eeq

\subsubsection{Hydrodynamic Ericksen-Leslie model for nematic liquid crystals}
The widely acceptable hydrodynamic theory for small molecular weight,  nematic liquid crystal flows is the Ericksen-Leslie model \cite{Leslie1979}.  In this theory, a vector $\bf p$ is used to describe the average molecular orientation and $\bu$ is the mass average velocity. The Ericksen-Leslie hydrodynamic model reads as
\begin{subequations} \label{eq:LC-Ericksen-Lesile}
\begin{align}
& \rho(\partial_t \bu + \bu \cdot \nabla \bu) = -\nabla p + \eta \Delta \bu + \nabla \cdot ( \frac{1}{2}(\bp \bh -\bh \bp) - \frac{a}{2}(\bp \bh +\bh \bp) ) - \bh \nabla \bp, \\
& \nabla \cdot \bu = 0, \\
& \partial_t \bp + (\bu \cdot \nabla) \bp -\bW\cdot \bp - a\bD \cdot \bp = M \bh, \\
& \bh = K \Delta \bp - \frac{1}{\varepsilon^2} (|\bp|^2 -1) \bp,
\end{align}
\end{subequations}
where $\rho$ is the mass density of the liquid crystal solution, $a$ is a parameter associated to the molecular geometry, $\varepsilon$ controls the defect length scale, $K$ is the elastic constant, and $M$ is the mobility coefficient. The Ericksen-Leslie model is usually associated with the homogeneous boundary conditions
\begin{subequations}  \label{eq:LC-Ericksen-Lesile-boundary}
\begin{align}
& \mbox{i.} \quad \bu(\bx, t) = 0, \quad  \bp(x, t) = 0, \quad (\bx, t) \in \partial \Omega \times (0, T]; \mbox{ or }\\
& \mbox{ii.} \quad \bu(\bx, t) = 0, \quad  \frac{\partial \bp(x, t)}{\partial \bn} = 0, \quad (\bx, t) \in \partial \Omega \times (0, T],
\end{align}
\end{subequations}
with $\bn$ the outward normal vector at the boundary.

The nematic liquid crystal model in \eqref{eq:LC-Ericksen-Lesile} can be derived in a similar manner. The total energy is
\beq
\cE(\bu, \bp) = E(\bu) + F(\bp), 
\eeq 
with $E(\bu)$ the kinetic energy defined in \eqref{eq:kinetic-energy}, and $F(\bp)$ the Ossen-Frank free energy, given by 
\beq
F(\bp) = \int_\Omega  d\bx \Big[ \frac{K}{2}|\nabla \bp|^2 + \frac{1}{4\varepsilon^2}(|\bp|^2 -1)^2 \Big],
\eeq
in the form of the one-constant approximation,  where $K$ is the Frank elastic constant and $\varepsilon$ is a small parameter for the width of the diffuse interface. And the molecular field can be derived as
$\bh  : = -\frac{\delta \cE}{\delta \bp}$. 
Hence, if we denote $\Phi = \begin{bmatrix}\bu \\ \bp  \end{bmatrix}$, the Ericksen-Lesile model can be writen in the general form \eqref{eq:evolution-general}, with the mobility operators given as
\begin{subequations}
\begin{align}
&
\cG_s=
\left(
\begin{array}{cc}
-\frac{1}{\rho} B(\bu, \bullet )   & \frac{1}{\rho} G_{s12}   \\  
\frac{1}{\rho} G_{s21} & 0 \\
\end{array}
\right)
, \quad
\cG_a=
\left(
\begin{array}{cc}
-\frac{1}{\rho^2} \eta \Delta \bI & 0  \\
0 & M  \\
\end{array}
\right),\\
& G_{s12}=\nabla \cdot ( \frac{1}{2}(\bp \bullet + \bullet \bp) -\frac{a}{2}(\bp \bullet + \bullet \bp)) -\bullet \nabla \bp ), \\
&G_{s21} = - \frac{1}{2}(\nabla \bullet - \nabla \bullet^T) \cdot \bp -\frac{a}{2}(\nabla \bullet + \nabla \bullet^T) \cdot \bp + \bullet \cdot \nabla \bp.
\end{align}
\end{subequations}

\section{Decoupled numerical algorithms based on the EQ-RID method}\label{sec:generic-numerical-method}
Through the examples in the previous section, we are clear that the model \eqref{eq:evolution-general} is rather general, that many widely used incompressible hydrodynamic models, including \eqref{eq:CHNS} and \eqref{eq:LC-Ericksen-Lesile}, can be cast into its form. In this section, we propose numerical algorithms for the general model \eqref{eq:evolution-general}, which in turn will guide us to develop numerical algorithms for specific models that can be cast in \eqref{eq:evolution-general}.

\subsection{Model reformulation with the energy quadratization (EQ) method}
In the fist step, we transform the general model \eqref{eq:evolution-general} into the energy-quadratized form, using the the idea of energy quadratization (EQ). Denote the total energy as
\beq \label{eq:general-E}
\cE(\Phi) = \int_\Omega e d\bx,
\eeq 
with $e$ the energy density function. We denote $L_0$ as a linear operator that can be separated from $e$. For instance, for the CHNS system in \eqref{eq:CHNS}, we may denote
\beq
\cL_0 = \begin{bmatrix}
\rho & 0 \\
0 & - \varepsilon^2 \Delta + \gamma_0 
\end{bmatrix}.
\eeq 
Introduce the auxiliary variable
\beq \label{eq:EQ-intermediate-variable}
q = \sqrt{2\Big(e -\frac{1}{2}|\cL_0^{\frac{1}{2}} \Phi|^2+ \frac{A_0}{|\Omega|}\Big)},
\eeq
where  $A_0$ is  such that $q$ is a well defined real variable. Then we rewrite the energy in \eqref{eq:general-E} as
\beq
\cE(\Phi, q) =  \frac{1}{2}\Big(\Phi, \cL_0\Phi \Big) + \frac{1}{2}\Big( q, q \Big) - A_0.
\eeq 
With the EQ approach above,  we transform the free energy density into a quadratic one by introducing an auxiliary variable to "remove" the quadratic gradient term from the energy density.
Assuming $q =q(\Phi,\nabla \Phi)$ and denoting
\beq   \label{eq:partial_derivative}
g(\Phi)= \frac{\partial q}{\partial \Phi},  \quad  \bG(\Phi) = \frac{\partial q}{\partial \nabla \Phi},
\eeq
we reformulate  \eqref{eq:evolution-general} into an equivalent form
\begin{subequations} \label{eq:evolution-general-EQ}
\begin{align}
&\partial_t \Phi = -(\cG_a + \cG_s) \Big[ \cL_0 \Phi + q g(\Phi) - \nabla \cdot ( q \bG(\Phi))\Big] , \\
&\partial_t q = g(\Phi): \partial_t \Phi + \bG(\Phi): \nabla \partial_t \Phi, 
\end{align}
\end{subequations}
with the consistent initial condition $q|_{t=0} =\left. \sqrt{2\Big(e -\frac{1}{2}|\cL_0^{\frac{1}{2}} \Phi|^2+ \frac{A_0}{|\Omega|}\Big)}\right|_{t=0}$.
Now, instead of dealing with \eqref{eq:evolution-general} directly, we develop   structure-preserving schemes for    \eqref{eq:evolution-general-EQ}.

The advantage of using  model \eqref{eq:evolution-general-EQ} over   model \eqref{eq:evolution-general} is that the   energy density is transformed into a quadratic one  in  \eqref{eq:evolution-general-EQ}.
Denoting $\Psi =\begin{bmatrix} \Phi \\ q \end{bmatrix}$, we rewrite \eqref{eq:evolution-general-EQ} into a compact from
\beq \label{eq:evolution-general-EQ-vector}
 \partial_t \Psi  = -\cN(\Psi) \cL \Psi,  \quad \mbox{ with } \cN(\Psi) = \cN_s(\Psi) + \cN_a(\Psi),
\eeq 
where  $\cL = \begin{bmatrix}
\cL_0 & \\
& 1
\end{bmatrix}_{n+1,n+1},$ is a linear operator, and 
\begin{subequations}
\begin{align}
& \cN(\Psi) = \cN_s(\Psi) + \cN_a(\Psi), \\
& \cN_a(\Psi) = \cN_0^\ast \cG_a  \cN_0 ,   \cN_s(\Psi) = \cN_0^\ast \cG_s \cN_0, \\
&\cN_0= (\bI_n, \,\,\, g(\Phi) + \bG(\Phi)\colon\nabla)_{n,n+1} , 
\end{align}
\end{subequations}
and $\cN_0^\ast$ is the adjoint operator of $\cN_0$. We name it the Onsager-Q model, where the energy is
\beq
\cE(\Psi) = \frac{1}{2}(\Psi, \cL \Psi) - A_0, 
\eeq 
with the energy law given as
\beq  \label{eq:energy-law-EQ}
\frac{d \cE(\Psi)}{d t}
= \Big( \frac{\delta \cE}{\delta \Psi} \frac{d \Psi}{d t}, 1 \Big)
= -\Big( \ \cL \Psi,  \cN(\Psi) \cL \Psi  \Big) = -\Big( \cN_0 \cL \Psi,  \cG_s \cN_0 \cL \Psi  \Big)   \leq  0,
\eeq
when $\dot{\cE}_{surf}=0$. This  is called the energy quadratization (EQ) reformulation (or method).
Note that the Onsager-Q model's energy is quadratized so that we can develop a paradigm to derive linear, energy-stable numerical schemes for the model.

\subsection{Model reformulation to decouple the reversible and irreversible dynamics}
We introduce an auxiliary scalar variable $s(t)$, for instance $s(t) = e^{-\frac{t}{T}}$, such that $s(t)e^{\frac{t}{T}}=1$, where $T$ is the final time. And we reformulate the Onsager-Q model in \eqref{eq:evolution-general-EQ} by multiplying the constant $s(t)e^{\frac{t}{T}}=1$ on the reversible terms. This brings us the equivalent system as
\begin{subequations} \label{eq:gradient-flow-decoupled}
\begin{align}
&  \label{eq:gradient-flow-decoupled-1} \partial_t \Psi =- (\cN_s + s e^{\frac{t}{T}}\cN_a)  \cL \Psi, \\
&\label{eq:gradient-flow-decoupled-2} \partial_t s =  -\frac{1}{T} s + e^{\frac{t}{T}} \Big( \cL \Psi, \cN_a \cL \Psi\Big),  \quad s(t=0) = 1. 
\end{align}
\end{subequations}

\begin{rem}
We emphasis that the reformulated model \eqref{eq:gradient-flow-decoupled} is equivalent to the original generalized model \eqref{eq:evolution-general}. Notice that in the continuous level, $\Big(\cL \Psi, \cN_a \cL\Psi\Big)=0$ so that $s(t)$ can be solved in \eqref{eq:gradient-flow-decoupled-2} as $s(t) = e^{\frac{t}{T}}$. Plugging this into \eqref{eq:gradient-flow-decoupled-1}, we have
$\partial_t \Psi = (\cN_s + \cN_a) \cL \Psi$, which reduces to \eqref{eq:evolution-general-EQ-vector} that is equivalent to \eqref{eq:evolution-general}.
\end{rem}

In the rest of this paper, we focus on developing numerical algorithms for the reformulated equations in \eqref{eq:gradient-flow-decoupled}. It will be clear in the later section that the reformulated system in \eqref{eq:gradient-flow-decoupled} provides guidance on designing accurate and efficient numerical algorithms. In particular, it decouples the reversible and irreversible dynamics, which automatically decouples the equations when the coupling terms are only in irreversible dynamics.

\subsection{Generic numerical algorithms for the generalized Onsager model} 

Consider the time domain $t \in [0, T]$. We discretize it into uniform meshes 
$0=t_0 < t_1< t_2< \cdots < t_{N-1}< t_N = T,$
with $N$ a positive integer. In other words, we have $t_i = i\delta t$ with $\delta t = \frac{T}{N}$. We also introduce the notations:
\begin{subequations}
\begin{align}
 &\overline{(\bullet)}^{n+1} = 2 (\bullet)^n - (\bullet)^{n-1},  \quad \overline{(\bullet)}^{n+\frac{1}{2}} = \frac{3}{2} (\bullet)^n - \frac{1}{2} (\bullet)^{n-1}, \\ & (\bullet)^{n+\frac{1}{2}} = \frac{1}{2} (\bullet)^n +  \frac{1}{2} (\bullet)^{n+1}, \quad (\hat{\bullet})^{n+\frac{1}{2}} = \frac{1}{2} (\bullet)^n +  \frac{1}{2} (\hat{\bullet})^{n+1}.
\end{align}
\end{subequations}

\subsubsection{Generic BDF2 numerical schemes}
For the reformulated system in \eqref{eq:gradient-flow-decoupled}, we propose the following semi-implicit numerical algorithm based on the second-order backward differentiation formula (BDF2).
 
\begin{scheme}[Semi-implicit BDF time-integration scheme]  \label{scheme:GENERIC-BDF2-decoupled}
With $(\Psi^{n-1},s^{n-1})$ and $(\Psi^n,s^n)$, we compute $(\Psi^{n+1},s^{n+1})$ via
\begin{subequations}\label{eq:GENERIC-BDF2-decoupled}
\begin{align}
& \label{eq:BDF-1} \frac{3\Psi^{n+1} - 4\Psi^n+\Psi^{n-1}}{2\delta t}  = -\overline{\cN}_s^{n+1}  \cL \Psi^{n+1}- s^{n+1} e^{\frac{ t_{n+1}}{T}} \overline{\cN}_a^{n+1}  \cL \overline{\Psi}^{n+1}, \\
& \label{eq:BDF-2}  \frac{3s^{n+1} -4 s^n+s^{n-1}}{2\delta  t} = -\frac{1}{T} s^{n+1}  + e^{\frac{t_{n+1}}{T}}\Big( \cL \Psi^{n+1}, \overline{\cN}_a^{n+1}  \cL \overline{\Psi}^{n+1}\Big).
\end{align}
\end{subequations}
\end{scheme}

\begin{thm}
The scheme \ref{scheme:GENERIC-BDF2-decoupled} is unconditionally energy stable, in the sense that
\beq
\cE(\Psi^{n+1}, \Psi^n, s^{n+1}, s^n)  - \cE(\Psi^n, \Psi^{n-1}, s^n, s^{n-1})  \leq - \delta t (\cL\Psi^{n+1}, \,\,\,   \overline{\cN}_s^{n+1}  \cL\Psi^{n+1}) -  \frac{\delta t}{T} (s^{n+1})^2,
\eeq 
where the modified free energy is defined as
\begin{multline}
\cE(\Psi_1, \Psi_2,  s_1, s_2) = \frac{1}{4} \Big[ (\Psi_1, \cL\Psi_1) + \Big(2\Psi_1 - \Psi_2, \cL(2\Psi_1 -\Psi_2) \Big)\Big]  + \frac{1}{4} \Big[  (s_1)^2 + (2s_1-s_2)^2 \Big]- A_0 .
\end{multline}
\end{thm}

\begin{proof}
First of all, notice the equality
$$
a\frac{3 a - 4b + c}{2} =  \frac{1}{4} (a^2 + (2a-b)^2) - \frac{1}{4}( b^2 + (2b-c)^2) + \frac{1}{4}(a-2b+c)^2.
$$
If we take inner product of \eqref{eq:BDF-1} with $\delta t \cL \Psi^{n+1}$, we will have
\begin{multline} \label{eq:bdf-proof-1}
\frac{1}{4}\Big[ \Big( \Psi^{n+1}, \cL \Psi^{n+1} \Big)   + \Big(2 \Psi^{n+1} - \Psi^{n}, \cL(2 \Psi^{n+1} - \Psi^{n}) \Big) \Big] -
\frac{1}{4}\Big[ \Big( \Psi^{n}, \cL \Psi^{n} \Big)   
+ \Big(2 \Psi^{n} - \Psi^{n-1}, \\
\cL(2 \Psi^{n} - \Psi^{n-1}) \Big) \Big] -
 \leq 
-\delta t  \Big( \cL \Psi^{n+1},  \overline{\cN}_s^{n+1}  \cL \Psi^{n+1} \Big) - \delta t s^{n+1} e^{\frac{ t_{n+1}}{T}} \Big( \cL\Psi^{n+1}, \overline{\cN}_a^{n+1}  \cL \overline{\Psi}^{n+1} \Big).
\end{multline}
Then, if we take the inner product of \eqref{eq:BDF-2} with $\delta t s^{n+1}$, we will obtain
\begin{multline} \label{eq:bdf-proof-2}
\frac{1}{4} \Big[ (s^{n+1})^2 + (2s^{n+1}-s^n)^2 \Big] - \frac{1}{4} \Big[ (s^n)^2 + (2s^n-s^{n-1})^2 \Big] \leq  \\
-\frac{\delta t}{T} (s^{n+1})^2 + \delta t s^{n+1} e^{\frac{ t_{n+1}}{T}}  \Big( \cL\Psi^{n+1}, \overline{\cN}_a^{n+1}  \cL \overline{\Psi}^{n+1} \Big).
\end{multline}
Adding up the two equations \eqref{eq:bdf-proof-1} and \eqref{eq:bdf-proof-2} above will lead us to
\beq
\cE(\Psi^{n+1}, \Psi^n, s^{n+1}, s^n)  - \cE(\Psi^n, \Psi^{n-1}, s^n, s^{n-1})  \leq - \delta t  \Big( \cL\Psi^{n+1},   \overline{\cN}_s^{n+1}  \cL \Psi^{n+1} \Big) -  \frac{\delta t}{T} (s^{n+1})^2.
\eeq 
Thus, this completes the proof.

\end{proof}

Next, we further elaborate how the scheme \ref{scheme:GENERIC-BDF2-decoupled} can be implemented effectively.
Notice the equations \eqref{eq:GENERIC-BDF2-decoupled} can be rewritten in the form of
\begin{subequations}
\begin{align}
&\cA \Psi^{n+1} = F_1 + s^{n+1} F_2, \\
&s^{n+1} = c^n + \Big( \cL\Psi^{n+1}, F_3\Big),
\end{align}
\end{subequations}
where the operators are given as
\begin{subequations}
\begin{align}
& \cA= (\frac{3}{2 \delta t} + \overline{\cN}_s^{n+1} \cL), \quad    F_1 = \frac{4 \Psi^n - \Psi^{n-1}}{2\delta t}, \quad F_2 = -  e^{\frac{ t_{n+1}}{T}} \overline{\cN}_a^{n+1}  \cL \overline{\Psi}^{n+1}, \\
& c_n = \frac{ \frac{ 4s^n- s^{n-1}}{2 \delta t}}{ \frac{3}{2 \delta t} + \frac{1}{T}}, \quad F_3 = \frac{1}{ \frac{3}{2 \delta t} + \frac{1}{T}} e^{\frac{t_{n+1}}{T}} \overline{\cN}_a^{n+1} \cL \overline{\Psi}^{n+1}.
\end{align}
\end{subequations}
We denote the solution as
$
\Psi^{n+1} = \Psi_1^{n+1} + s^{n+1} \Psi_2^{n+1},
$
where $\Psi_1^{n+1}$ and $\Psi_2^{n+1}$ are the solutions for
\beq \label{eq:BDF2-phi}
\cA \Psi_1^{n+1}  = F_1, \quad  \cA \Psi_2^{n+1} = F_2.
\eeq 
Meanwhile, we can solve the scalar variable $s^{n+\frac{1}{2}}$ from
$
s^{n+1} = c_n + \Big( B\Psi_1^{n+1}, F_3\Big) + s^{n+1} \Big( B \Psi_2^{n+1}, F_3\Big),
$
from which we can easily obtain the formula
\beq
s^{n+1} =  \frac{c_n + \Big( B\Psi_1^{n+1}, F_3\Big) }{1 - \Big( B \Psi_2^{n+1}, F_3\Big)}.
\eeq 
The Solution existence and uniqueness will depend on the existence of $s^{n+1}$. Roughly, $F_3 = O(\delta t)$, such that $s^{n+1}$ is well-defined when $\delta t$ is not too large.
Therefore, we get the final formula for the solution
\beq
\Psi^{n+1} = 
\Psi_1^{n+1} + \frac{c_n + \Big( B\Psi_1^{n+1}, F_3\Big) }{1 - \Big( B \Psi_2^{n+1}, F_3\Big)} \Psi_2^{n+1},
\eeq 
where $\Psi_1^{n+1}$ and $\Psi_2^{n+1}$ are the solutions of \eqref{eq:BDF2-phi}.

\subsubsection{Generic CN numerical schemes}
We can also propose an alternative second-order numerical scheme based on the idea of the CN (i.e., Crank-Nicolson) finite difference method. The scheme is given below.

\begin{scheme}[Semi-implicit CN time-integration scheme] \label{scheme:GENERIC-CN-decoupled}
With $(\Psi^{n-1},s^{n-1})$ and $(\Psi^n,s^n)$, we compute $(\Psi^{n+1},s^{n+1})$ via
\begin{subequations} \label{eq:GENERIC-CN-decoupled}
\begin{align}
& \label{eq:CN-1} \frac{\Psi^{n+1} - \Psi^n}{\delta t}  = -\overline{\cN}_s^{n+\frac{1}{2}}  \cL\Psi^{n+\frac{1}{2}} - s^{n+\frac{1}{2}} e^{\frac{ t_{n+1/2}}{T}} \overline{\cN}_a^{n+\frac{1}{2}}  \cL \overline{\Psi}^{n+ \frac{1}{2}} ,\\
& \label{eq:CN-2} \displaystyle \frac{s^{n+1} - s^n}{\delta  t} = -\frac{1}{T} s^{n+\frac{1}{2}} +  e^{\frac{t_{n+1/2}}{T}}\Big( \cL \Psi^{n+\frac{1}{2}}, \overline{\cN}_a^{n+\frac{1}{2}}   \cL \overline{\Psi}^{n+ \frac{1}{2}}\Big).
\end{align}
\end{subequations} 
\end{scheme}

\begin{thm} \label{thm:GENERIC-CN-decoupled}
The scheme \ref{scheme:GENERIC-CN-decoupled} is unconditionally energy stable, in the sense that
\beq
\cE(\Psi^{n+1}, s^{n+1})  - \cE(\Psi^n, s^n) = - \delta t \Big(\cL\Psi^{n+\frac{1}{2}},   \overline{\cN}_s^{n+\frac{1}{2}}  \cL\Psi^{n+\frac{1}{2}}\Big) -  \frac{\delta t}{T} (s^{n+\frac{1}{2}})^2,
\eeq 
where the modified free energy is defined as
\beq
\cE(\Psi, s) = \frac{1}{2}\Big( \Psi, \cL\Psi\Big) - A_0 + \frac{1}{2} s^2.
\eeq 
\end{thm}

\begin{proof}
As a matter of fact, we take inner product of \eqref{eq:CN-1} with $\delta t  \cL  \Psi^{n+\frac{1}{2}}$ to get
\begin{multline}\label{eq:cn-proof-1}
\frac{1}{2} \Big( \Psi^{n+1}, \cL \Psi \Big) - \frac{1}{2}  \Big(  \Psi^n, \cL \Psi^n \Big) = \\
-\delta t \Big(\cL\Psi^{n+\frac{1}{2}},   \overline{\cN}_s^{n+\frac{1}{2}}  \cL\Psi^{n+\frac{1}{2}}\Big) - \delta t s^{n+\frac{1}{2}} e^{\frac{ t_{n+1/2}}{T}}\Big( \cL\Psi^{n+\frac{1}{2}},  \overline{\cN}_a^{n+\frac{1}{2}}  \cL \overline{\Psi}^{n+ \frac{1}{2}}\Big).
\end{multline}
Meanwhile, if we take inner product of \eqref{eq:CN-2} with $s^{n+\frac{1}{2}}$, we have
\beq \label{eq:cn-proof-2}
\frac{1}{2} (s^{n+1})^2 - \frac{1}{2}(s^n)^2 = -\frac{\delta t}{T} (s^{n+\frac{1}{2}})^2 +\delta t  s^{n+\frac{1}{2}}  e^{\frac{ t_{n+1/2}}{T}}\Big( \cL\Psi^{n+\frac{1}{2}},  \overline{\cN}_a^{n+\frac{1}{2}}  \cL \overline{\Psi}^{n+ \frac{1}{2}}\Big).
\eeq 
Adding the two equations \eqref{eq:cn-proof-1} and \eqref{eq:cn-proof-2}, we end up with
\beq
\cE(\Phi^{n+1}, s^{n+1})  - \cE(\Phi^n, s^n) = - \delta t \Big(\cL\Psi^{n+\frac{1}{2}},   \overline{\cN}_s^{n+\frac{1}{2}}  \cL\Psi^{n+\frac{1}{2}}\Big) -  \frac{\delta t}{T} (s^{n+\frac{1}{2}})^2.
\eeq 
This completes the proof.

\end{proof}

Next, we explain some implementation tricks for the Scheme \ref{scheme:GENERIC-CN-decoupled}.
First of all, we solve for $(\Psi^{n+\frac{1}{2}}, s^{n+\frac{1}{2}})$ by re-write the equation of \eqref{eq:GENERIC-CN-decoupled} as
\begin{subequations} \label{eq:GENERIC-CN-decoupled-reformulation}
\begin{align}
& \frac{ 2}{\delta t} \Big( \Psi^{n+\frac{1}{2}} - \Psi^n \Big)   =- \overline{\cN}_s^{n+\frac{1}{2}}  \cL\Psi^{n+\frac{1}{2}} - s^{n+\frac{1}{2}} e^{\frac{ t_{n+1/2}}{T}} \overline{\cN}_a^{n+\frac{1}{2}}  \cL \overline{\Psi}^{n+ \frac{1}{2}} ,\\
& \frac{2}{\delta t} \Big( s^{n+\frac{1}{2}} - s^n \Big) = -\frac{1}{T} s^{n+\frac{1}{2}} +  e^{\frac{t_{n+1/2}}{T}}\Big( \cL \Psi^{n+\frac{1}{2}}, \overline{\cN}_a^{n+\frac{1}{2}}   \cL \overline{\Psi}^{n+ \frac{1}{2}}\Big).
\end{align}
\end{subequations}

Notice that after proper arrangement, we need to solve the following  linear system for each time marching step, as
\begin{subequations} \label{eq:CN-linear-system}
\begin{align}
&\cA \Psi^{n+\frac{1}{2}}  = F_1 + s^{n+\frac{1}{2}} F_2, \\
&s^{n+\frac{1}{2}} = c_n + \Big(\cL \Psi^{n+\frac{1}{2}}, F_3\Big),
\end{align}
\end{subequations}
where $c_n$ is a constant scalar, and $F_1$, $F_2$, $F_3$ are vectors, and $A, B$ are coefficient matrices that can be calculated as
\begin{subequations}
\begin{align}
&\cA=  \Big( \frac{2}{\delta t} + \overline{\cN}_s^{n+\frac{1}{2}}  \cL \Big), \quad F_1 = \frac{2}{\delta t} \Psi^n, \quad F_2 = - e^{\frac{ t_{n+1/2}}{T}} \overline{\cN}_a^{n+\frac{1}{2}}  \cL \overline{\Psi}^{n+ \frac{1}{2}}. \\
&c_n = \frac{ \frac{2}{\delta t}}{ \frac{2}{\delta t} + \frac{1}{T}} s^n, \quad  F_3 = \frac{e^{\frac{t_{n+1/2}}{T}}}{\frac{2}{\delta t} + \frac{1}{T}}\overline{\cN}_a^{n+\frac{1}{2}}   \cL \overline{\Psi}^{n+ \frac{1}{2}}.
\end{align}
\end{subequations}

The goal is to find the solution $\Psi^{n+\frac{1}{2}}$ in \eqref{eq:CN-linear-system} such that we can obtain $\Psi^{n+1} = 2\Psi^{n+\frac{1}{2}} - \Psi^n$. Given the fact that \eqref{eq:CN-linear-system} is a linear system, we can write the solution in the form
\beq
\Psi^{n+\frac{1}{2}} = \Psi_1^{n+\frac{1}{2}} + s^{n+\frac{1}{2}} \Psi_2^{n+\frac{1}{2}},
\eeq
where $\Psi_1^{n+\frac{1}{2}}$ and $\Psi_2^{n+\frac{1}{2}}$ can be derived as 
\beq \label{eq:CN_intermedidate_solution}
\Psi^{n+\frac{1}{2}}_1 = \cA^{-1} F_1, \quad \Psi^{n+\frac{1}{2}}_2 = \cA^{-1} F_2,
\eeq 
i.e., the  are the solutions for the two systems
$\cA\Psi_1^{n+\frac{1}{2}} = F_1$ and $\cA \Psi_2^{n+\frac{1}{2}} = F_2$. 
Meanwhile, we can solve the scalar variable $s^{n+\frac{1}{2}}$ from
$$
s^{n+\frac{1}{2}} = c_n + \Big( B\Psi_1^{n+\frac{1}{2}}, F_3\Big) + s^{n+\frac{1}{2}} \Big( B \Psi_2^{n+\frac{1}{2}}, F_3\Big),
$$
from which we can easily obtain the formula
\beq
s^{n+\frac{1}{2}} =  \frac{c_n + \Big( \cL \Psi_1^{n+\frac{1}{2}}, F_3\Big) }{1 - \Big( \cL \Psi_2^{n+\frac{1}{2}}, F_3\Big)}.
\eeq 
The Solution existence and uniqueness will depend on the existence of $s^{n+\frac{1}{2}}$. Roughly, $F_3 = O(\delta t)$, such that $s^{n+\frac{1}{2}}$ is well-defined when $\delta t$ is not too large.
Therefore, we get the final formula for the solution
\beq
\Psi^{n+\frac{1}{2}} = 
\Psi_1^{n+\frac{1}{2}} + \frac{c_n + \Big( \cL \Psi_1^{n+\frac{1}{2}}, F_3\Big) }{1 - \Big( \cL \Psi_2^{n+\frac{1}{2}}, F_3\Big)} \Psi_2^{n+\frac{1}{2}}, \quad \Psi^{n+1} = 2\Psi^{n+\frac{1}{2}} - \Psi^n.
\eeq 
where  the two intermediate solutions are obtained  in \eqref{eq:CN_intermedidate_solution}.

\section{Applications of the general numerical framework for specific incompressible hydrodynamic models} \label{sec:applications}

In this section, we apply the general numerical framework to some specific incompressible hydrodynamic models. Due to space limitation, we only apply the CN Scheme \ref{scheme:GENERIC-CN-decoupled}. The application of the BDF Scheme \ref{scheme:GENERIC-BDF2-decoupled} is similar. Thus we omit the details.

\subsection{Numerical algorithms for the Cahn-Hilliard-Navier-Stokes equations}
For the Cahn-Hilliard-Navier-Stokes equation in \eqref{eq:CHNS}, we introduce the energy quadratization (EQ) notations
\beq
\cL_0 = \begin{bmatrix}
\rho & 0 \\0 & - \varepsilon^2 \Delta  + \gamma_0
\end{bmatrix}, \quad 
q = \frac{\sqrt{2}}{2} (\phi^2 - 1- \gamma_0), \quad g(\phi): =\frac{\partial q}{\partial \phi} = \sqrt{2} \phi,
\eeq 
where $\gamma_0 \geq 0$ is a regularization parameter \cite{Chen&Zhao&Yang2018}.
The reformulated equations in the Onsager-Q form read as
\begin{subequations} \label{eq:CHNS-EQ}
\begin{align}
&\rho \Big( \partial_t \bu + B(\bu, \bu) \Big) = - \nabla p + \eta \nabla^2 \bu -   \phi \nabla \mu,  \quad   (\bx, t) \in \Omega_T, \\
&\nabla \cdot \bu  =  0,  \quad  (\bx, t) \in \Omega_T,  \\
&\partial_t \phi +\nabla \cdot (\bu \phi) =  M\Delta   \mu, \quad  (\bx, t) \in \Omega_T, \\
&\mu = - \varepsilon^2 \Delta \phi + \gamma_0 \phi + g(\phi) q, \quad (\bx, t) \in \Omega_T,  \\
& \partial_t q = g(\phi) \partial_t \phi, \quad (\bx, t) \in \Omega_T,
\end{align}
\end{subequations}
with consistent initial conditions. Then, we utilize the  reversible-irreversible dynamics (RID) idea in  \eqref{eq:gradient-flow-decoupled} to reformulate the equation of \eqref{eq:CHNS-EQ} into
\begin{subequations} \label{eq:CHNS-decoupled}
\begin{align}
&\rho \Big( \partial_t \bu + s e^{\frac{t}{T}}B(\bu, \bu) \Big) = - \nabla p + \eta \nabla^2 \bu -  s e^{\frac{t}{T}} \phi \nabla \mu,  \quad   (\bx, t) \in \Omega_T, \\
&\nabla \cdot \bu  =  0,  \quad  (\bx, t) \in \Omega_T,  \\
&\partial_t \phi +  s e^{\frac{t}{T}} \nabla \cdot (\bu \phi) =  M\Delta   \mu, \quad  (\bx, t) \in \Omega_T, \\
&\mu = - \varepsilon^2 \Delta \phi + \gamma_0 \phi + g(\phi) q, \quad (\bx, t) \in \Omega_T,  \\
& \partial_t q = g(\phi) \partial_t \phi, \quad (\bx, t) \in \Omega_T, \\
& \partial_t s = -\frac{1}{T} s +  e^{\frac{t}{T}} \Big[ \Big( \bu, B(\bu, \bu)\Big)+ \Big( \bu, \phi \nabla \mu\Big) +\Big(\mu, \nabla \cdot (\bu \phi)\Big) \Big].
\end{align}
\end{subequations}
with consistent initial conditions. Afterwards, by plugging the proposed generic numerical schemes to \eqref{eq:CHNS-decoupled}, we have the following specific numerical scheme.
\begin{scheme} \label{scheme:CHNS}
Given $(\bu^n, \phi^n, q^n, s^n)$ and $(\bu^{n-1}, \phi^{n-1}, q^{n-1}, s^{n-1})$, we can update $(\bu^{n+1}, \phi^{n+1}, q^{n+1}, s^{n+1})$ via the following time-marching scheme
\begin{subequations} \label{eq:scheme-CHNS}
\begin{align}
&\rho \frac{ \bu^{n+1} - \bu^n}{\delta t} + s^{n+\frac{1}{2}} e^{\frac{t_{n+1/2}}{T}}  \rho B(\overline{\bu}^{n+\frac{1}{2}}, \overline{\bu}^{n+\frac{1}{2}}) = -\nabla p + \eta \Delta \bu^{n+\frac{1}{2}} -  s^{n+\frac{1}{2}} e^{\frac{t_{n+1/2}}{T}} \overline{\phi}^{n+\frac{1}{2}} \nabla \overline{\mu}^{n+\frac{1}{2}}, \\
&\nabla \cdot \bu^{n+1} = 0, \\
&\frac{\phi^{n+1} - \phi^n}{\delta t} + s^{n+\frac{1}{2}} e^{\frac{t_{n+1/2}}{T}}  \nabla \cdot (\overline{\bu}^{n+\frac{1}{2}} \overline{\phi}^{n+\frac{1}{2}}) = M \Delta \mu^{n+\frac{1}{2}}, \\
&\mu^{n+\frac{1}{2}} = -\varepsilon^2 \Delta \phi^{n+\frac{1}{2}} + \gamma_0 \phi^{n+\frac{1}{2}} +   q^{n+\frac{1}{2}} g(\overline{\phi}^{n+\frac{1}{2}}),  \\
&\frac{q^{n+1} - q^n}{\delta t} = g(\overline{\phi}^{n+\frac{1}{2}}) \frac{\phi^{n+1} - \phi^n}{\delta t}, \\
&\nonumber \frac{s^{n+1} - s^n}{\delta t}  = - \frac{1}{T} s^{n+\frac{1}{2}} + e^{ \frac{ t_{n+1/2}}{T} } \Big[ \Big( \bu^{n+\frac{1}{2}},   \rho B(\overline{\bu}^{n+\frac{1}{2}}, \overline{\bu}^{n+\frac{1}{2}})  + \overline{\phi}^{n+\frac{1}{2}} \nabla \overline{\mu}^{n+\frac{1}{2}}\Big)  \\
& \qquad +\Big( \mu^{n+\frac{1}{2}},  \nabla \cdot (\overline{\bu}^{n+\frac{1}{2}} \overline{\phi}^{n+\frac{1}{2}}) \Big) \Big], \\
& \nabla \mu^{n+1} \cdot \bn = 0, \quad \nabla \phi^{n+1} \cdot \bn = 0, \quad   \bu^{n+1} =0, \quad \mbox{ on } \partial  \Omega.
\end{align}
\end{subequations}
\end{scheme}

Then, we can easily show the following theorem, as an analogy to  Theorem \ref{thm:GENERIC-CN-decoupled}.
\begin{thm}
The scheme \ref{scheme:CHNS} is unconditionally energy stable, in the sense that
\beq
\cE(\bu^{n+1}, \phi^{n+1}, q^{n+1}, s^{n+1}) - \cE(\bu^{n}, \phi^{n}, q^{n}, s^{n}) = -\delta t  \Big[ M \|\nabla  \mu^{n+\frac{1}{2}} \|^2  + \eta \| \nabla \bu^{n+\frac{1}{2}} \|^2  +\frac{1}{T} (s^{n+\frac{1}{2}})^2 \Big],
\eeq 
where the modified free energy is defined as
\beq
\cE(\bu, \phi, q, s) = \frac{\rho}{2} \| \bu\|^2 + \frac{\varepsilon^2}{2} \| \nabla \phi \|^2 + (\frac{\gamma_0}{2} \phi^2  +\frac{1}{2}q^2, 1) + \frac{1}{2}s^2 - A_0, \quad A_0 = \frac{1}{4} r_0^2 + \frac{r_0}{2}.
\eeq
\end{thm}

\begin{proof}
Notice the Scheme \ref{scheme:CHNS} is a direct application of Scheme \ref{scheme:GENERIC-CN-decoupled}. The proof is similar to the one for Theorem  \ref{thm:GENERIC-CN-decoupled}. We omit the details.
\end{proof}

From this specific example, we can observe that the EQ-RID idea can decouple the hydrodynamic variables $(\bu, p)$ from the state variable $\phi$.
We emphasize that the scheme \ref{scheme:CHNS} can be efficiently solved since the solution procedure has decoupled the velocity field $\bu$ and the phase-field variable $\phi$. It can be easily observed that the scheme \ref{scheme:CHNS} only includes solving smaller linear problems in each step.
Specifically, it is equivalent to the following decoupled scheme.

\begin{scheme}[Practice Implementation of Scheme \ref{scheme:CHNS}] \label{scheme:CHNS-steps}
After we calculate the previous solutions $(\bu^n, \phi^n, q^n, s^n)$ and $(\bu^{n-1}, \phi^{n-1}, q^{n-1}, s^{n-1})$, we can update $(\bu^{n+1}, \phi^{n+1}, q^{n+1}, s^{n+1})$ via the following time-marching scheme
\begin{itemize}
\item Step 1, denote $\bu^{n+\frac{1}{2}} = \bu_1^{n+\frac{1}{2}} + s^{n+\frac{1}{2}} \bu_2^{n+\frac{1}{2}}$, and $\phi^{n+\frac{1}{2}} = \phi_1^{n+\frac{1}{2}} + s^{n+\frac{1}{2}} \phi_2^{n+\frac{1}{2}}$.
\item Step 2.1, get $\phi_1^{n+\frac{1}{2}}$ by solving  the equation below
\begin{subequations} \label{eq:step-CHNS-phi1}
\begin{align}
& \frac{2}{\delta t}( \phi_1^{n+\frac{1}{2}} - \phi^n) = M \Delta \mu_1^{n+\frac{1}{2}}, \\
& \mu_1^{n+\frac{1}{2}} =  - \varepsilon^2 \phi_1^{n+\frac{1}{2}} + q^n g(\overline{\phi}^{n+\frac{1}{2}}) + g^2(\overline{\phi}^{n+\frac{1}{2}}) (\phi_1^{n+\frac{1}{2}} - \phi^n), \\
& \nabla \mu_1^{n+\frac{1}{2}} \cdot \bn = 0, \quad \nabla \phi_1^{n+\frac{1}{2}} \cdot \bn =0, \mbox{ on } \partial \Omega.
\end{align}
\end{subequations}  
\item Step 2.2, get $\phi_2^{n+\frac{1}{2}}$ by solving  the equation below
\begin{subequations} \label{eq:step-CHNS-phi2}
\begin{align}
& \frac{2}{\delta t}\phi_2^{n+\frac{1}{2}}  + e^{\frac{t_{n+1/2}}{T}} \nabla \cdot (\overline{\bu}^{n+\half} \overline{\phi}^{n+\half}) = M \Delta \mu_2^{n+\frac{1}{2}} , \\
& \mu_2^{n+\frac{1}{2}} =  - \varepsilon^2\Delta  \phi_2^{n+\frac{1}{2}} + g^2(\overline{\phi}^{n+\frac{1}{2}})\phi_2^{n+\frac{1}{2}}, \\
& \nabla \mu_2^{n+\frac{1}{2}} \cdot \bn = 0, \quad \nabla \phi_2^{n+\frac{1}{2}} \cdot \bn =0, \mbox{ on } \partial \Omega.
\end{align}
\end{subequations}
\item Step 2.3, get $\bu_1^{n+\frac{1}{2}}$ by solving
\begin{subequations}
\begin{align}
& \frac{\rho}{\delta t}( \bu_1^{n+\frac{1}{2}} - \bu^n)  = -\nabla p + \eta \Delta \bu_1^{n+\half}, \\
& \nabla \cdot \bu_1^{n+\half} =0, \\
& \bu_1^{n+\half}  = 0, \mbox{ on } \partial \Omega.
\end{align}
\end{subequations}
\item Step 2.4, get $\bu_2^{n+\frac{1}{2}}$ by solving
\begin{subequations}
\begin{align}
& \frac{\rho}{\delta t} \bu_2^{n+\frac{1}{2}}  + e^{\frac{t_{n+1/2}}{T}}  B(\overline{\bu}^{n+\half} , \overline{\bu}^{n+\half})   = -\nabla p + \eta \Delta \bu_2^{n+\half} + e^{\frac{t_{n+1/2}}{T}}  \overline{\phi}^{n+\half} \nabla \overline{\mu}^{n+\half}, \\
& \nabla \cdot  \bu_2^{n+\half} =0, \\
& \bu_2^{n+\half} = 0, \mbox{ on } \partial \Omega.
\end{align}
\end{subequations}

\item Step 2.5, get $s^{n+\frac{1}{2}}$ by solving the following linear algebra equation
\begin{multline}
\frac{s^{n+\half} - s^n}{2\delta t}  = - \frac{1}{T} s^{n+\frac{1}{2}} + e^{ \frac{ t_{n+1/2}}{T} } \Big[ \Big( \bu_1^{n+\frac{1}{2}} +s^{n+\frac{1}{2}} \bu_2^{n+\frac{1}{2}},   \rho B(\overline{\bu}^{n+\frac{1}{2}}, \overline{\bu}^{n+\frac{1}{2}})  + \overline{\phi}^{n+\frac{1}{2}} \nabla \overline{\mu}^{n+\frac{1}{2}}\Big)  \\
+\Big( \mu_1^{n+\frac{1}{2}} + s^{n+\half} \mu_2^{n+\half},  \nabla \cdot (\overline{\bu}^{n+\frac{1}{2}} \overline{\phi}^{n+\frac{1}{2}}) \Big) \Big].
\end{multline}

\item Step 3, With the information in Step 2, we can obtain the solution through 
\begin{subequations}
\begin{align}
& \bu^{n+1} = 2 (\bu_1^{n+\frac{1}{2}}+ s^{n+\frac{1}{2}} \bu_2^{n+\frac{1}{2}} ) - \bu^n, \\
& \phi^{n+1} = 2 (\phi_1^{n+\frac{1}{2}}+ s^{n+\frac{1}{2}} \phi_2^{n+\frac{1}{2}} ) - \phi^n,\\
& q^{n+1} = q^n + g(\overline{\phi}^{n+\frac{1}{2}}) (\phi^{n+1}  - \phi^n), \\
& s^{n+1} = 2s^{n+\frac{1}{2}} - s^n.
\end{align}
\end{subequations}

\end{itemize}
\end{scheme}

 Furthermore, we can decouple the velocity field $\bu$ and pressure $p$ by embracing the velocity-projection technique in \cite{Han&WangJCP2015}. This leads to the fully decoupled scheme as below.

\begin{scheme} \label{scheme:CHNS-fully}
Given $(\bu^n, p^n, \phi^n, q^n, s^n)$ and $(\bu^{n-1}, p^{n-1}, \phi^{n-1}, q^{n-1}, s^{n-1})$, we can update the solution at current time $(\bu^{n+1}, p^{n+1}, \phi^{n+1}, q^{n+1}, s^{n+1})$ via the following time-marching scheme
\begin{subequations} \label{eq:scheme-CHNS-fully}
\begin{align}
&\rho \frac{ \hat{\bu}^{n+1} - \bu^n}{\delta t} + s^{n+\frac{1}{2}} e^{\frac{t_{n+1/2}}{T}}  \rho B(\overline{\bu}^{n+\frac{1}{2}}, \overline{\bu}^{n+\frac{1}{2}})) = -\nabla p^n + \eta \Delta \hat{\bu}^{n+\frac{1}{2}} -  s^{n+\frac{1}{2}} e^{\frac{t_{n+1/2}}{T}} \overline{\phi}^{n+\frac{1}{2}} \nabla \overline{\mu}^{n+\frac{1}{2}}, \\
& \frac{\bu^{n+1} - \hat{\bu}^{n+1}}{\delta t} = - \frac{1}{2} \nabla (p^{n+1} - p^n), \\
&\nabla \cdot \bu^{n+1} = 0, \\
&\frac{\phi^{n+1} - \phi^n}{\delta t} + s^{n+\frac{1}{2}} e^{\frac{t_{n+1/2}}{T}}  \nabla \cdot (\overline{\bu}^{n+\frac{1}{2}} \overline{\phi}^{n+\frac{1}{2}}) = M \Delta \mu^{n+\frac{1}{2}}, \\
&\mu^{n+\frac{1}{2}} = -\varepsilon^2 \Delta \phi^{n+\frac{1}{2}} + \gamma_0 \phi^{n+\frac{1}{2}} + q^{n+\frac{1}{2}} g(\overline{\phi}^{n+\frac{1}{2}}),  \\
&\frac{q^{n+1} - q^n}{\delta t} = g(\overline{\phi}^{n+\frac{1}{2}}) \frac{\phi^{n+1} - \phi^n}{\delta t}, \\
&\nonumber \frac{s^{n+1} - s^n}{\delta t}  = - \frac{1}{T} s^{n+\frac{1}{2}} + e^{ \frac{ t_{n+1/2}}{T} } \Big[ \Big( \hat{\bu}^{n+\frac{1}{2}},   \rho B(\overline{\bu}^{n+\frac{1}{2}}, \overline{\bu}^{n+\frac{1}{2}})  + \overline{\phi}^{n+\frac{1}{2}} \nabla \overline{\mu}^{n+\frac{1}{2}}\Big)  \\
& \qquad +\Big( \mu^{n+\frac{1}{2}},  \nabla \cdot (\overline{\bu}^{n+\frac{1}{2}} \overline{\phi}^{n+\frac{1}{2}}) \Big) \Big], \\
& \nabla \mu^{n+1} \cdot \bn = 0, \quad \nabla \phi^{n+1} \cdot \bn = 0, \quad   \hat{\bu}^{n+1} =0, \quad \nabla p^{n+1} \cdot \bn =0, \quad \mbox{ on } \partial  \Omega.
\end{align}
\end{subequations}
\end{scheme}

Similarly, we can have the following energy stability theorem.
\begin{thm}
The scheme \ref{scheme:CHNS-fully} is unconditionally energy stable, in the sense that
\begin{multline}
\cE(\bu^{n+1}, p^{n+1}, \phi^{n+1}, q^{n+1}, s^{n+1}) - \cE(\bu^{n}, p^n, \phi^{n}, q^{n}, s^{n}) \\
 = -\delta t  \Big[ M \|\nabla  \mu^{n+\frac{1}{2}} \|^2  + \eta \| \nabla \hat{\bu}^{n+\frac{1}{2}} \|^2  +\frac{1}{T} (s^{n+\frac{1}{2}})^2 \Big],
\end{multline}
where the modified free energy is defined as
\beq
\cE(\bu, \phi, q, s) = \frac{\rho}{2} \| \bu\|^2 + \frac{\delta t^2}{8\rho} \| \nabla p\|^2 + \frac{\varepsilon^2}{2} \| \nabla \phi \|^2 + (\frac{\gamma_0}{2} \phi^2  +\frac{1}{2}q^2, 1) + \frac{1}{2}s^2 - A_0, \quad A_0 = \frac{1}{4} r_0^2 + \frac{r_0}{2}.
\eeq
\end{thm}

\begin{proof}
The proof is similar to the one in Theorem \ref{scheme:GENERIC-CN-decoupled}. The extra modification of the pressure term in the free energy can be found in \cite{Han&WangJCP2015}. We omit the details.
\end{proof}

Similarly, the Scheme \ref{scheme:CHNS-fully} can be implemented as below.

\begin{scheme}[Practice Implementation of Scheme \ref{scheme:CHNS-fully}] \label{scheme:CHNS-steps-fullly}
Given the solutions in previous time steps $(\bu^n, p^n, \phi^n, q^n, s^n)$ and $(\bu^{n-1}, p^{n-1}, \phi^{n-1}, q^{n-1}, s^{n-1})$, we update the solution at current time $(\bu^{n+1}, p^{n+1}, \phi^{n+1}, q^{n+1}, s^{n+1})$ via the following time-marching scheme
\begin{itemize}
\item Step 1, denote $\hat{\bu}^{n+\frac{1}{2}} = \bu_1^{n+\frac{1}{2}} + s^{n+\frac{1}{2}} \bu_2^{n+\frac{1}{2}}$, and $\phi^{n+\frac{1}{2}} = \phi_1^{n+\frac{1}{2}} + s^{n+\frac{1}{2}} \phi_2^{n+\frac{1}{2}}$.
\item Step 2.1, get $\phi_1^{n+\frac{1}{2}}$ by solving  the equation \eqref{eq:step-CHNS-phi1}.
\item Step 2.2, get $\phi_2^{n+\frac{1}{2}}$ by solving  the equation  \eqref{eq:step-CHNS-phi2}.
\item Step 2.3, get $\bu_1^{n+\frac{1}{2}}$ by solving
\begin{subequations}
\begin{align}
& \frac{\rho}{\delta t}( \bu_1^{n+\frac{1}{2}} - \bu^n)  = -\nabla p^n + \eta \Delta \bu_1^{n+\half}, \\
& \bu_1^{n+\half} = 0, \mbox{ on } \partial \Omega.
\end{align}
\end{subequations}
\item Step 2.4, get $\bu_2^{n+\frac{1}{2}}$ by solving
\begin{subequations}
\begin{align}
& \frac{\rho}{\delta t} \bu_2^{n+\frac{1}{2}}  + e^{\frac{t_{n+1/2}}{T}}  B(\overline{\bu}^{n+\half} , \overline{\bu}^{n+\half})   = -\nabla p^n + \eta \Delta \bu_2^{n+\half} + e^{\frac{t_{n+1/2}}{T}}  \overline{\phi}^{n+\half} \nabla \overline{\mu}^{n+\half}, \\
& \bu_2^{n+\half} = 0, \mbox{ on } \partial \Omega.
\end{align}
\end{subequations}

\item Step 2.5, get $s^{n+\frac{1}{2}}$ by solving the following linear algebra equation
\begin{multline}
\frac{s^{n+\half} - s^n}{2\delta t}  = - \frac{1}{T} s^{n+\frac{1}{2}} + e^{ \frac{ t_{n+1/2}}{T} } \Big[ \Big( \bu_1^{n+\frac{1}{2}} +s^{n+\frac{1}{2}} \bu_2^{n+\frac{1}{2}},   \rho B(\overline{\bu}^{n+\frac{1}{2}}, \overline{\bu}^{n+\frac{1}{2}})  + \overline{\phi}^{n+\frac{1}{2}} \nabla \overline{\mu}^{n+\frac{1}{2}}\Big)  \\
+\Big( \mu_1^{n+\frac{1}{2}} + s^{n+\half} \mu_2^{n+\half},  \nabla \cdot (\overline{\bu}^{n+\frac{1}{2}} \overline{\phi}^{n+\frac{1}{2}}) \Big) \Big].
\end{multline}

\item Step 3, with the information in Step 2, we update the solution through
\begin{subequations}
\begin{align}
& \phi^{n+1} = 2 (\phi_1^{n+\frac{1}{2}}+ s^{n+\frac{1}{2}} \phi_2^{n+\frac{1}{2}} ) - \phi^n,\\
& q^{n+1} = q^n + g(\overline{\phi}^{n+\frac{1}{2}}) (\phi^{n+1}  - \phi^n), \\
& s^{n+1} = 2s^{n+\frac{1}{2}} - s^n.
\end{align}
\end{subequations}
With $\hat{\bu}^{n+1} = 2 (\bu_1^{n+\frac{1}{2}}+ s^{n+\frac{1}{2}} \bu_2^{n+\frac{1}{2}} ) - \bu^n$,
we get $(\bu^{n+1}, p^{n+1})$ via
\begin{subequations}
\begin{align}
& \frac{1}{\delta t}( \bu^{n+1} - \hat{\bu}^{n+1}) = -\frac{1}{2} \nabla (\bp^{n+1} - \bp^n), \\
& \nabla \cdot \bu^{n+1} = 0.
\end{align}
\end{subequations}
\end{itemize}
\end{scheme}

\subsection{Numerical algorithms for the hydrodynamic Ericksen-Leslie model}
Next, we apply the general numerical framework on the  the nematic liquid crystal model in \eqref{eq:LC-Ericksen-Lesile}. Specifically, we introduce the energy quadratization (EQ) notations:
\beq
\cL_0 = \begin{bmatrix}
\rho & 0 \\0 & - K \Delta  + \gamma_0
\end{bmatrix}, \quad 
q = \frac{1}{\sqrt{2} \varepsilon} (|\bp|^2 - 1- \varepsilon^2 \gamma_0), \quad g(\bp): =\frac{\partial q}{\partial \bp} = \frac{\sqrt{2}}{\varepsilon} \bp,
\eeq 
where $\gamma_0$ is a regularization parameter \cite{Chen&Zhao&Yang2018}.
According to \eqref{eq:evolution-general-EQ}, we have the Onsager-Q form
\begin{subequations}\label{eq:LC-Ericksen-Lesile-EQ} 
\begin{align}
& \frac{\rho}{\delta t} (\partial_t \bu +B(\bu, \bu)) = -\nabla p + \eta \Delta \bu + \nabla \cdot ( \frac{1}{2}(\bp \bh -\bh \bp) - \frac{a}{2}(\bp \bh +\bh \bp) ) - \bh \nabla \bp, \\
& \nabla \cdot \bu = 0, \\
& \partial_t \bp + \bu \cdot \nabla \bp -\bW\cdot \bp - a\bD \cdot \bp = M \bh, \\
& \bh = K \Delta \bp - \gamma_0 \bp - q g(\bp), \\
& \partial_t q = g(\bp) \cdot \partial_t \bp,
\end{align}
\end{subequations}
with consistent initial condition for $q$. Then, based on the generic form of reversible-irreversible dynamics (RID) idea  in \eqref{eq:gradient-flow-decoupled}, we can reformulate \eqref{eq:LC-Ericksen-Lesile-EQ} into
\begin{subequations}  \label{eq:LC-Ericksen-Lesile-decoupled} 
\begin{align}
& \rho (\partial_t \bu + s e^{\frac{t}{T}} B(\bu, \bu)) = -\nabla p + \eta \Delta \bu + s e^{\frac{t}{T}} \Big[ \nabla \cdot ( \frac{1}{2}(\bp \bh -\bh \bp) - \frac{a}{2}(\bp \bh +\bh \bp) ) - \bh \nabla \bp \Big], \\
& \nabla \cdot \bu = 0, \\
& \partial_t \bp + s e^{\frac{t}{T}}  \Big[ \bu \cdot \nabla \bp - \bW\cdot \bp - a\bD \cdot \bp \Big]= M \bh, \\
& \bh = K \Delta \bp - \gamma_0 \bp - q g(\bp), \\
& \partial_t q = g(\bp) \cdot \partial_t \bp, \\
& \nonumber \partial_t s = -\frac{1}{T} s +  e^{\frac{t}{T}} \Big[  \Big(\bu, \rho B(\bu, \bu)- \nabla \cdot ( \frac{1}{2}(\bp \bh -\bh \bp) - \frac{a}{2}(\bp \bh +\bh \bp) ) + \bh \nabla \bp\Big)  \\
&\qquad + \Big( -\bh, \bu \cdot \nabla \bp - \bW\cdot \bp - a \bD\cdot \bp\Big)  \Big].
\end{align}
\end{subequations}

With the reformulation in \eqref{eq:LC-Ericksen-Lesile-decoupled}, we can plug it into the general decoupled Scheme \ref{scheme:GENERIC-CN-decoupled}. Hence, the scheme for the Ericksen-Leslie model of nematic liquid crystal flow is given as
\begin{scheme} \label{scheme:LC-Ericksen-Lesile}
After we calculate the solutions $(\bu^n, \bp^n, q^n, s^n)$ and $(\bu^{n-1}, \bp^{n-1}, q^{n-1}, s^{n-1})$, we can update $(\bu^{n+1}, \bp^{n+1}, q^{n+1}, s^{n+1})$ via the following scheme
\begin{subequations} \label{eq:scheme-LC-Ericksen-Lesile}
\begin{align}
& \nonumber  \frac{\rho}{\delta t} (\bu^{n+1} - \bu^n) + s^{n+\frac{1}{2}} e^{ \frac{t_{n+1/2}}{T}} B(\overline{\bu}^{n+1/2}, \overline{\bu}^{n+1/2}) = -\nabla p + \eta \Delta \bu^{n+\frac{1}{2}} + s^{n+\frac{1}{2}} e^{ \frac{t_{n+1/2}}{T}}   \\
& \qquad  \Big[  \nabla \cdot ( \frac{1}{2}(\overline{\bp}^{n+\frac{1}{2}}  \overline{\bh}^{n+\frac{1}{2}} -\overline{\bh}^{n+\frac{1}{2}} \overline{\bp}^{n+\frac{1}{2}} ) - \frac{a}{2}(\overline{\bp}^{n+\frac{1}{2}}  \overline{\bh}^{n+\frac{1}{2}} +\overline{\bh}^{n+\frac{1}{2}} \overline{\bp}^{n+\frac{1}{2}} ) - \overline{\bh}^{n+\frac{1}{2}}  \nabla \overline{\bp}^{n+\frac{1}{2}}  \Big], \\
& \nabla \cdot \bu^{n+\frac{1}{2}} = 0, \\
& \frac{1}{\delta t} ( \bp^{n+1} - \bp^n) + s^{n+\frac{1}{2}} e^{ \frac{t_{n+1/2}}{T}} \Big[ \overline{\bu}^{n+\frac{1}{2}} \cdot \nabla \overline{\bp}^{n+\frac{1}{2}} - \overline{\bW}^{n+\frac{1}{2}} \cdot \overline{\bp}^{n+\frac{1}{2}} - a\overline{\bD}^{n+\frac{1}{2}} \cdot \overline{\bp}^{n+\frac{1}{2}} \Big] = M \bh^{n+\frac{1}{2}}, \\
& \bh^{n+\frac{1}{2}} =  K \Delta \bp^{n+\frac{1}{2}} - \gamma_0 \bp^{n+\frac{1}{2}}- q^{n+\frac{1}{2}} g(\overline{\bp}^{n+\frac{1}{2}}), \\
& \frac{1}{\delta t}( q^{n+1} - q^n) = g(\overline{\bp}^{n+\frac{1}{2}}) \cdot \frac{1}{\delta t}(\bp^{n+1} - \bp^n), \\
& \nonumber \frac{1}{\delta t} (s^{n+1} - s^n) = - \frac{1}{T} s^{n+\frac{1}{2}} + e^{ \frac{t_{n+1/2}}{T}} \Big[    \Big( \bu^{n+\frac{1}{2}}, \rho B(\overline{\bu}^{n+\frac{1}{2}}, \overline{\bu}^{n+\frac{1}{2}}) \Big) \\
& \nonumber  - \Big( \bh^{n+\frac{1}{2}}, \overline{\bu}^{n+\frac{1}{2}} \cdot \nabla \overline{\bp}^{n+\frac{1}{2}} -(\overline{\bW}^{n+\half} + a \overline{\bD}^{n+\half}) \cdot \overline{\bp}^{n+\half}\Big) \\
&+ \Big( \bu^{n+\frac{1}{2}}, -\nabla \cdot ( \frac{1}{2}(\overline{\bp}^{n+\frac{1}{2}}  \overline{\bh}^{n+\frac{1}{2}} -\overline{\bh}^{n+\frac{1}{2}} \overline{\bp}^{n+\frac{1}{2}} ) - \frac{a}{2}(\overline{\bp}^{n+\frac{1}{2}}  \overline{\bh}^{n+\frac{1}{2}} +\overline{\bh}^{n+\frac{1}{2}} \overline{\bp}^{n+\frac{1}{2}} )) + \overline{\bh}^{n+\frac{1}{2}}  \nabla \overline{\bp}^{n+\frac{1}{2}} \Big)  \Big], \\
& \bu^{n+1} =0, \quad  \nabla \bp^{n+1} \cdot \bn =\mathbf{0}, \quad \mbox{ on } \partial \Omega.
\end{align}
\end{subequations}
\end{scheme}

With the proposed Scheme \ref{scheme:LC-Ericksen-Lesile} for the Ericksen-Leslie model, we have the following energy stable property.

\begin{thm}
The scheme \ref{scheme:LC-Ericksen-Lesile} is unconditionally energy stable, in the sense that
\beq
\cE(\bu^{n+1}, \bp^{n+1}, q^{n+1}, s^{n+1}) - \cE(\bu^{n}, \bp^{n}, q^{n}, s^{n}) = -\delta t  \Big[ M \| \bh^{n+\frac{1}{2}} \|^2  + \eta \| \nabla \bu^{n+\frac{1}{2}} \|^2 +\frac{1}{T} (s^{n+\frac{1}{2}})^2 \Big],
\eeq 
where the modified free energy is defined as
\beq
\cE(\bu, \bp, q, s) = \frac{\rho}{2} \| \bu\|^2 + \frac{K}{2} \| \nabla \bp\|^2 + (\frac{\gamma_0}{2} \bp^2  +\frac{1}{2}q^2, 1) + \frac{1}{2}s^2 - A_0, \quad A_0 = \frac{\varepsilon^2 \gamma_0^2}{4} + \frac{\gamma_0}{2}.
\eeq
\end{thm}

\begin{proof}
Given this is a direct application of the general scheme \ref{scheme:GENERIC-CN-decoupled}, the proof is similar to the proof for Theorem \ref{thm:GENERIC-CN-decoupled}. We thus omit the details.
\end{proof}

Furthermore, by utilizing the velocity projection idea,  we can decouple the velocity field and pressure. Eventually, we come up with the fully decoupled scheme as follows.

\begin{scheme} \label{scheme:LC-Ericksen-Lesile-fully}
Given $(\bu^n, \bp^n, q^n, s^n)$ and $(\bu^{n-1}, \bp^{n-1}, q^{n-1}, s^{n-1})$, we can update $(\bu^{n+1}, \bp^{n+1}, q^{n+1}, s^{n+1})$ via the following scheme
\begin{subequations} \label{eq:scheme-LC-Ericksen-Lesile-fully}
\begin{align}
& \nonumber  \frac{\rho}{\delta t} (\hat{\bu}^{n+1} - \bu^n) + s^{n+\frac{1}{2}} e^{ \frac{t_{n+1/2}}{T}} B(\overline{\bu}^{n+1/2}, \overline{\bu}^{n+1/2}) = -\nabla p + \eta \Delta \hat{\bu}^{n+\frac{1}{2}} + s^{n+\frac{1}{2}} e^{ \frac{t_{n+1/2}}{T}}   \\
& \qquad  \Big[  \nabla \cdot ( \frac{1}{2}(\overline{\bp}^{n+\frac{1}{2}}  \overline{\bh}^{n+\frac{1}{2}} -\overline{\bh}^{n+\frac{1}{2}} \overline{\bp}^{n+\frac{1}{2}} ) - \frac{a}{2}(\overline{\bp}^{n+\frac{1}{2}}  \overline{\bh}^{n+\frac{1}{2}} +\overline{\bh}^{n+\frac{1}{2}} \overline{\bp}^{n+\frac{1}{2}} ) - \overline{\bh}^{n+\frac{1}{2}}  \nabla \overline{\bp}^{n+\frac{1}{2}}  \Big], \\
& \rho \frac{\bu^{n+1} - \hat{\bu}^{n+1}}{\delta t} = - \frac{1}{2} \nabla (p^{n+1} - p^n), \\
& \nabla \cdot \bu^{n+\frac{1}{2}} = 0, \\
& \frac{1}{\delta t} ( \bp^{n+1} - \bp^n) + s^{n+\frac{1}{2}} e^{ \frac{t_{n+1/2}}{T}} \Big[ \overline{\bu}^{n+\frac{1}{2}} \cdot \nabla \overline{\bp}^{n+\frac{1}{2}} - \overline{\bW}^{n+\frac{1}{2}} \cdot \overline{\bp}^{n+\frac{1}{2}} - a\overline{\bD}^{n+\frac{1}{2}} \cdot \overline{\bp}^{n+\frac{1}{2}} \Big] = M \bh^{n+\frac{1}{2}}, \\
& \bh^{n+\frac{1}{2}} =  K \Delta \bp^{n+\frac{1}{2}} - \gamma_0 \bp^{n+\frac{1}{2}} - q^{n+\frac{1}{2}} g(\overline{\bp}^{n+\frac{1}{2}}), \\
& \frac{1}{\delta t}( q^{n+1} - q^n) = g(\overline{\bp}^{n+\frac{1}{2}}) \cdot \frac{1}{\delta t}(\bp^{n+1} - \bp^n), \\
& \nonumber \frac{1}{\delta t} (s^{n+1} - s^n) = - \frac{1}{T} s^{n+\frac{1}{2}} + e^{ \frac{t_{n+1/2}}{T}} \Big[    \Big( \bu^{n+\frac{1}{2}}, \rho B(\overline{\bu}^{n+\frac{1}{2}}, \overline{\bu}^{n+\frac{1}{2}}) \Big) \\
& \nonumber  - \Big( \bh^{n+\frac{1}{2}}, \overline{\bu}^{n+\frac{1}{2}} \cdot \nabla \overline{\bp}^{n+\frac{1}{2}} -(\overline{\bW}^{n+\half} + a \overline{\bD}^{n+\half}) \cdot \overline{\bp}^{n+\half}\Big) + \Big(\hat{\bu}^{n+\frac{1}{2}},\\
&  -\nabla \cdot ( \frac{1}{2}(\overline{\bp}^{n+\frac{1}{2}}  \overline{\bh}^{n+\frac{1}{2}} -\overline{\bh}^{n+\frac{1}{2}} \overline{\bp}^{n+\frac{1}{2}} ) - \frac{a}{2}(\overline{\bp}^{n+\frac{1}{2}}  \overline{\bh}^{n+\frac{1}{2}} +\overline{\bh}^{n+\frac{1}{2}} \overline{\bp}^{n+\frac{1}{2}} )) + \overline{\bh}^{n+\frac{1}{2}}  \nabla \overline{\bp}^{n+\frac{1}{2}} \Big)  \Big], \\
&\hat{\bu}^{n+1} =0, \quad  \nabla p^{n+1} \cdot \bn =0, \quad \nabla \bp^{n+1} \cdot \bn =\mathbf{0}, \quad \mbox{ on } \partial \Omega.
\end{align}
\end{subequations}
\end{scheme}

\begin{thm}
The scheme \ref{scheme:LC-Ericksen-Lesile-fully} is unconditionally energy stable, in the sense that
\begin{multline}
\cE(\bu^{n+1}, p^{n+1}, \bp^{n+1}, q^{n+1}, s^{n+1}) - \cE(\bu^{n}, p^{n+1}, \bp^{n}, q^{n}, s^{n}) \\
=  -\delta t  \Big[ M \| \bh^{n+\frac{1}{2}} \|^2  + \eta \| \nabla \bu^{n+\frac{1}{2}} \|^2 +\frac{1}{T} (s^{n+\frac{1}{2}})^2 \Big],
\end{multline}
where the modified free energy is defined as
\beq
\cE(\bu, p, \bp, q, s) = \frac{\rho}{2} \| \bu\|^2 + \frac{\delta t^2}{8\rho}\| \nabla p\|^2+ \frac{K}{2} \| \nabla \bp\|^2 + (\frac{\gamma_0}{2} \bp^2  +\frac{1}{2}q^2, 1) + \frac{1}{2}s^2 - A_0, \quad A_0 = \frac{\varepsilon^2 \gamma_0^2}{4} + \frac{\gamma_0}{2}.
\eeq
\end{thm}

\begin{proof}
 We thus omit the details of the proof since it is similar to before.
\end{proof}

We emphasis that the Schemes \ref{scheme:LC-Ericksen-Lesile} and \ref{scheme:LC-Ericksen-Lesile-fully} can be practically solved following the idea in previous sub-section. To save the space, we only briefly explain how Scheme \ref{scheme:LC-Ericksen-Lesile-fully} can be effectively implemented. Notice the fact
\beq
q^{n+\frac{1}{2}} = q^n  + g(\overline{\bp}^{n+\frac{1}{2}}) \cdot (\bp^{n+\frac{1}{2}} - \bp^n).
\eeq 
Plugging it into the expression for $\bh^{n+\frac{1}{2}}$ provides us
\beq
\bh^{n+\frac{1}{2}} = K \Delta \bp^{n+\frac{1}{2}} - \gamma_0 \bp^{n+\frac{1}{2}} -\Big[q^n  + g(\overline{\bp}^{n+\frac{1}{2}}) \cdot (\bp^{n+\frac{1}{2}} - \bp^n)\Big] g(\overline{\bp}^{n+\frac{1}{2}}).
\eeq 
With this in mind, we can implement Scheme \ref{scheme:LC-Ericksen-Lesile-fully} as follows.

\begin{scheme}[Practical Implementation for Scheme \ref{scheme:LC-Ericksen-Lesile-fully}]   \label{scheme:LC-Ericksen-Lesile-steps-fully}
After calculating the solutions in previous times $(\bu^n, p^n, \bp^n, q^n, s^n)$ and $(\bu^{n-1}, p^{n-1}, \bp^{n-1}, q^{n-1}, s^{n-1})$, we can update the next time step $(\bu^{n+1}, p^{n+1}, \bp^{n+1}, q^{n+1}, s^{n+1})$ via the following time-marching scheme
\begin{itemize}
\item Step 1, denote $\hat{\bu}^{n+\frac{1}{2}} = \bu_1^{n+\frac{1}{2}} + s^{n+\frac{1}{2}} \bu_2^{n+\frac{1}{2}}$, and $\bp^{n+\frac{1}{2}} = \bp_1^{n+\frac{1}{2}} + s^{n+\frac{1}{2}} \bp_2^{n+\frac{1}{2}}$.
\item Step 2.1, get $\bp_1^{n+\frac{1}{2}}$ by solving  the equation below
\begin{subequations}
\begin{align}
& \frac{2}{\delta t}( \bp_1^{n+\frac{1}{2}} - \bp^n) = M \bh_1^{n+\frac{1}{2}}, \\
& \bh_1^{n+\frac{1}{2}} =  K \Delta \bp_1^{n+\frac{1}{2}} - \gamma_0\bp_1^{n+\frac{1}{2}} -  \Big[ q^n +  g(\overline{\bp}^{n+\frac{1}{2}}) \cdot ( \bp_1^{n+\half} - \bp^n)]  g(\overline{\bp}^{n+\frac{1}{2}}), \\
& \nabla \bp_1^{n+\frac{1}{2}} \cdot \bn = \mathbf{0}, \mbox{ on } \partial \Omega.
\end{align}
\end{subequations}
\item Step 2.2, get $\bp_2^{n+\frac{1}{2}}$ by solving  the equation below
\begin{subequations}
\begin{align}
& \frac{2}{\delta t}\bp_2^{n+\frac{1}{2}}  + e^{\frac{t_{n+1/2}}{T}}  \overline{\bu}^{n+\half} \cdot \nabla \overline{\bp}^{n+\half} - e^{\frac{t_{n+1/2}}{T}} (\overline{\bW}^{n+\half} + a\overline{\bD}^{n+\half}) \cdot \overline{\bp}^{n+\half} = M \bh_2^{n+\frac{1}{2}} , \\
& \bh_2^{n+\frac{1}{2}} =  K \Delta \bp_2^{n+\frac{1}{2}} - \gamma_0\bp_2^{n+\frac{1}{2}} - \Big[ g(\overline{\bp}^{n+\frac{1}{2}}) \cdot \bp_2^{n+\frac{1}{2}} \Big] g(\overline{\bp}^{n+\frac{1}{2}}), \\
& \nabla \bp_2^{n+\frac{1}{2}} \cdot \bn = \mathbf{0}, \mbox{ on } \partial \Omega.
\end{align}
\end{subequations}
\item Step 2.3, get $\bu_1^{n+\frac{1}{2}}$ by solving
\begin{subequations}
\begin{align}
& \frac{\rho}{\delta t}( \bu_1^{n+\frac{1}{2}} - \bu^n)  = -\nabla p^n + \eta \Delta \bu_1^{n+\half}, \\
& \bu_1^{n+\half} = 0, \mbox{ on } \partial \Omega.
\end{align}
\end{subequations}
\item Step 2.4, get $\bu_2^{n+\frac{1}{2}}$ by solving
\begin{subequations}
\begin{align}
& \frac{\rho}{\delta t} \bu_2^{n+\frac{1}{2}}  + e^{\frac{t_{n+1/2}}{T}}  B(\overline{\bu}^{n+\half} , \overline{\bu}^{n+\half})   = -\nabla p^n + \eta \Delta \bu_2^{n+\half}  - e^{\frac{t_{n+1/2}}{T}}  \overline{\bh}^{n+\half} \nabla \overline{\bp}^{n+\half}\\
& \nonumber \qquad  + e^{\frac{t_{n+1/2}}{T}} \nabla \cdot \Big[   \frac{1}{2}(\overline{\bp}^{n+\frac{1}{2}}  \overline{\bh}^{n+\frac{1}{2}} -\overline{\bh}^{n+\frac{1}{2}} \overline{\bp}^{n+\frac{1}{2}} ) - \frac{a}{2}(\overline{\bp}^{n+\frac{1}{2}}  \overline{\bh}^{n+\frac{1}{2}} +\overline{\bh}^{n+\frac{1}{2}} \overline{\bp}^{n+\frac{1}{2}} )  \Big], \\
& \bu_2^{n+\half} = 0, \mbox{ on } \partial \Omega.
\end{align}
\end{subequations}

\item Step 2.5, get $s^{n+\frac{1}{2}}$ by solving the following linear algebra equation
\begin{multline}
\frac{s^{n+\half} - s^n}{2\delta t}  = - \frac{1}{T} s^{n+\frac{1}{2}} + e^{ \frac{ t_{n+1/2}}{T} } \Big[ \Big( \bu_1^{n+\frac{1}{2}} +s^{n+\frac{1}{2}} \bu_2^{n+\half}, \quad   \rho B(\overline{\bu}^{n+\frac{1}{2}}, \overline{\bu}^{n+\frac{1}{2}})  + \overline{\bh}^{n+\frac{1}{2}} \nabla \overline{\bp}^{n+\frac{1}{2}}  \\
- \nabla \cdot \Big(  \frac{1}{2}(\overline{\bp}^{n+\frac{1}{2}}  \overline{\bh}^{n+\frac{1}{2}} -\overline{\bh}^{n+\frac{1}{2}} \overline{\bp}^{n+\frac{1}{2}} ) - \frac{a}{2}(\overline{\bp}^{n+\frac{1}{2}}  \overline{\bh}^{n+\frac{1}{2}} +\overline{\bh}^{n+\frac{1}{2}} \overline{\bp}^{n+\frac{1}{2}} )  \Big)  \\
-\Big( \bh_1^{n+\frac{1}{2}} + s^{n+\half} \bh_2^{n+\half}, \quad  \overline{\bu}^{n+\frac{1}{2}} \cdot \nabla \overline{\bp}^{n+\frac{1}{2}} -(\overline{\bW}^{n+\half}+a\overline{\bD}^{n+\half}) \cdot \overline{\bp}^{n+\half} \Big) \Big].
\end{multline}

\item Step 3, With the information in Step 2, we can obtain the solution $(\bp^{n+1}, q^{n+1}, s^{n+1})$ through the following update 
\begin{subequations}
\begin{align}
& \bp^{n+1} = 2 (\bp_1^{n+\frac{1}{2}}+ s^{n+\frac{1}{2}} \bp_2^{n+\frac{1}{2}} ) - \bp^n,\\
& q^{n+1} = q^n + g(\overline{\bp}^{n+\frac{1}{2}}) \cdot (\bp^{n+1}  - \bp^n), \\
& s^{n+1} = 2s^{n+\frac{1}{2}} - s^n.
\end{align}
\end{subequations}
And, bu noticing $\hat{\bu}^{n+1} = 2 (\bu_1^{n+\frac{1}{2}}+ s^{n+\frac{1}{2}} \bu_2^{n+\frac{1}{2}} ) - \bu^n$, we get $(\bu^{n+1}, p^{n+1})$ via 
\begin{subequations}
\begin{align}
& \frac{1}{\delta t}( \bu^{n+1} - \hat{\bu}^{n+1}) = -\frac{1}{2} \nabla (\bp^{n+1} - \bp^n), \\
& \nabla \cdot \bu^{n+1} = 0.
\end{align}
\end{subequations}

\end{itemize}
\end{scheme}

We emphasize that the scheme \ref{scheme:LC-Ericksen-Lesile-steps-fully} is equivalent to the scheme \ref{scheme:LC-Ericksen-Lesile-fully}. And it is a practical implementation. From scheme \ref{scheme:LC-Ericksen-Lesile-steps-fully}, we see that the hydrodynamic variable $\bu$ and the state variable $\bp$ are decoupled so that only a few smaller linear systems shall be solved at each time marching step. This significantly reduces computational costs. In addition, the scheme is easy to implement.

\section{Numerical results} \label{sec:results}
In the rest of this section, we implement the full decoupled schemes, i.e., Scheme \ref{scheme:CHNS-steps-fullly} and Scheme \ref{scheme:LC-Ericksen-Lesile-steps-fully}, since they are most computationally efficient. Then we calculate several benchmark problems to test the accuracy and effectiveness of the proposed numerical schemes. We point out that the velocity-projection-inspired preconditioner \cite{GriffithJCP2009} can be used to effectively solve the coupled system for velocity and pressure field, such as \eqref{eq:step-CHNS-phi1} and \eqref{eq:step-CHNS-phi2}. Thus, the Scheme \ref{scheme:CHNS-steps} and Scheme \ref{scheme:LC-Ericksen-Lesile} can also be effectively solved. 

Notice the proposed general numerical framework focuses on temporal discretization. Structure-preserving spatial discretization, such as the Galerkin-type finite element method that preserves integration by parts and finite difference method that preserves summation by parts, can be utilized to obtain the full discrete schemes. In this paper, we use the structure preserving finite difference method, following our previous work \cite{Gong&Zhao&WangSISC2}. We point out that there is no particular restriction on the convection terms' discretization, given that they are treated explicitly and have no contribution to energy dissipation.  We thus use the WENO-type spatial discretization \cite{ShuC1} for the convection terms.

\subsection{Numerical examples for the Cahn-Hilliard-Navier-Stokes equations}
In this section, we use  Scheme \ref{scheme:CHNS-steps-fullly} to calculate several numerical examples for the Cahn-Hilliard-Navier-Stokes system in \eqref{eq:CHNS}. In particular, to be general, we consider periodic boundary conditions in the $x$-direction and physical boundary conditions in $y$-direction.

First of all, we verify the Scheme \ref{scheme:CHNS-steps-fullly} (or Scheme \ref{scheme:CHNS-fully}) is second-order accuracy in time. We pick the domain $\Omega=[0, 1]^2$, and parameters $\rho=1$, $\eta =1$, $\varepsilon = 10^{-2}$, and $\gamma_0 =0$. Then we fix spatial meshes as $128^2$,, $T=1$, and use various time steps $\delta t_n = 10^{-2} \times \frac{1}{2^k}$, $k=0,1,2,\cdots$. Given the real solution is unknown, we calculate the errors as the difference between the numerical solution and the adjacent numerical solution with finer time step. The $l^2$ errors and $l_\infty$ errors are summarized in Figure \ref{fig:CHNS-mesh-refinement}. We observe that  the proposed numerical scheme has second-order temporal accuracy in deed. 

\begin{figure}
\center
\subfigure[$l_2$ error]{\includegraphics[width=0.475\textwidth]{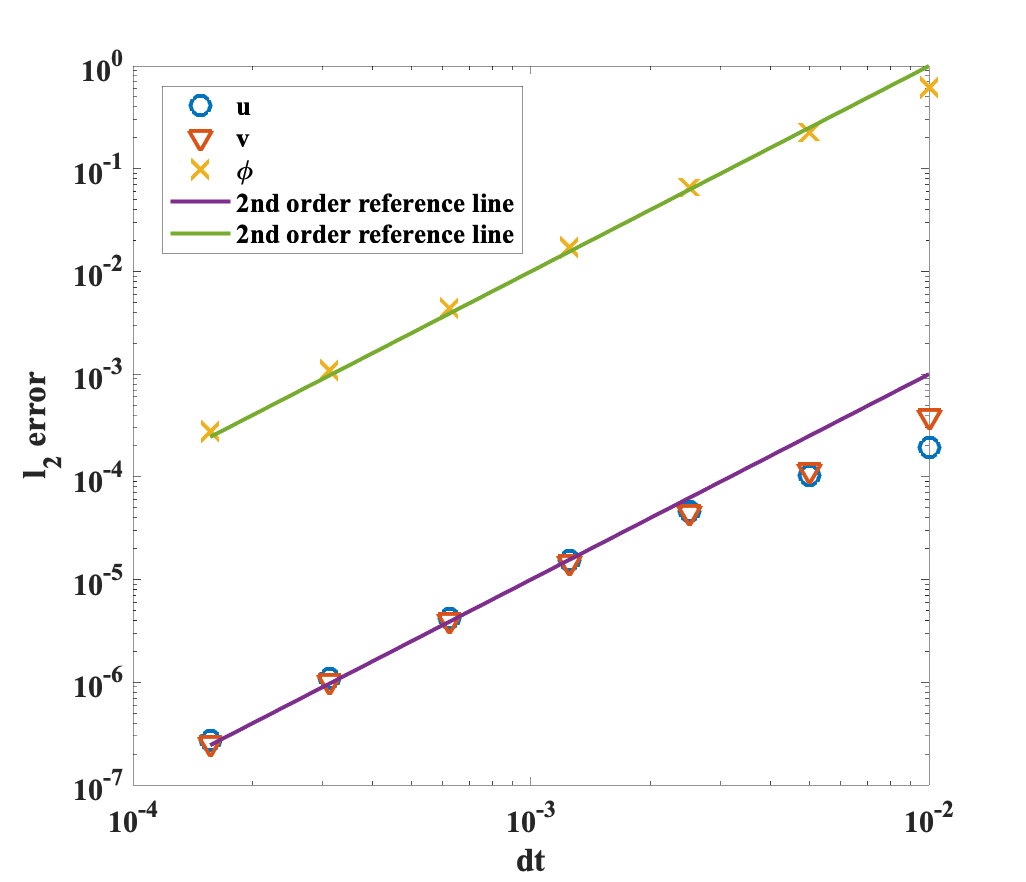}}
\subfigure[$l_\infty$ error]{\includegraphics[width=0.475\textwidth]{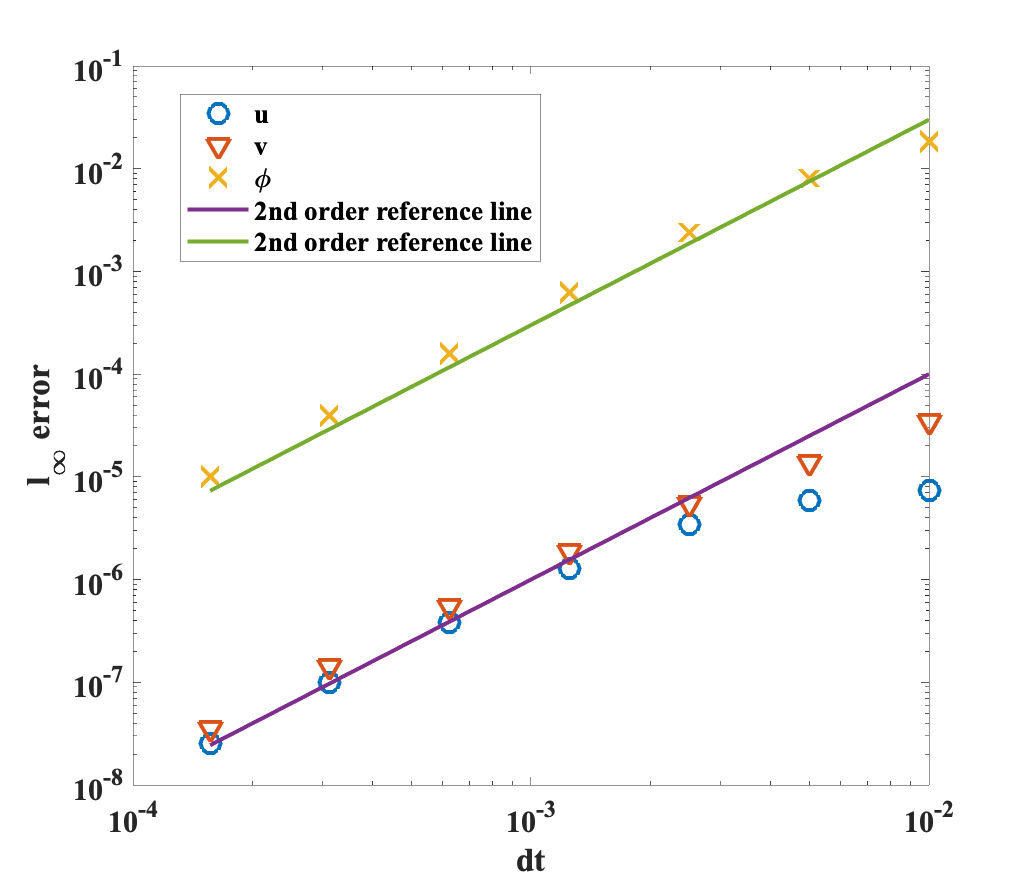}}
\caption{Time step mesh-refinement for Scheme \ref{scheme:CHNS-steps-fullly}. This figure shows that Scheme \ref{scheme:CHNS-steps-fullly} is second-order accurate in both the $l_2$ norm and $l_\infty$ norm.}
\label{fig:CHNS-mesh-refinement}
\end{figure}

Next, we use Scheme \ref{scheme:CHNS-steps-fullly} to conduct several benchmark simulations. In the first simulation, we consider the domain $\Omega=[0, L_x] \times[0, L_y]$ with $L_x=L_y=2$, parameters are chosen as $\rho=1$, $\eta=1$, $\varepsilon=0.01$, $\gamma_0=0$, $T=100$, and choose a random initial condition as
$$
\phi(x, y, t=0) = 0.9 (\frac{y}{L_y} - 0.5)  + 10^{-3} rand(-1,1), \quad (x, y) \in \Omega.
$$
In the implementation we use meshes $256^2$ and time step $\delta t= 0.005$. The evolution of the phase-field variable $\phi$ is summarized in Figure \ref{fig:CHNS-Coarsening}. We observe that spinodal decomposition takes more effect when the volume fraction of two phases is similar, saying in the middle of the domain. Meanwhile,  the nucleation takes more effect when the volume fractions of each phase differ dramatically,. This agrees well with the results in the literature.

\begin{figure}
\center
\subfigure[profiles of $\phi$ at $t=0,0.6,1,2$]{
\includegraphics[width=0.22\textwidth]{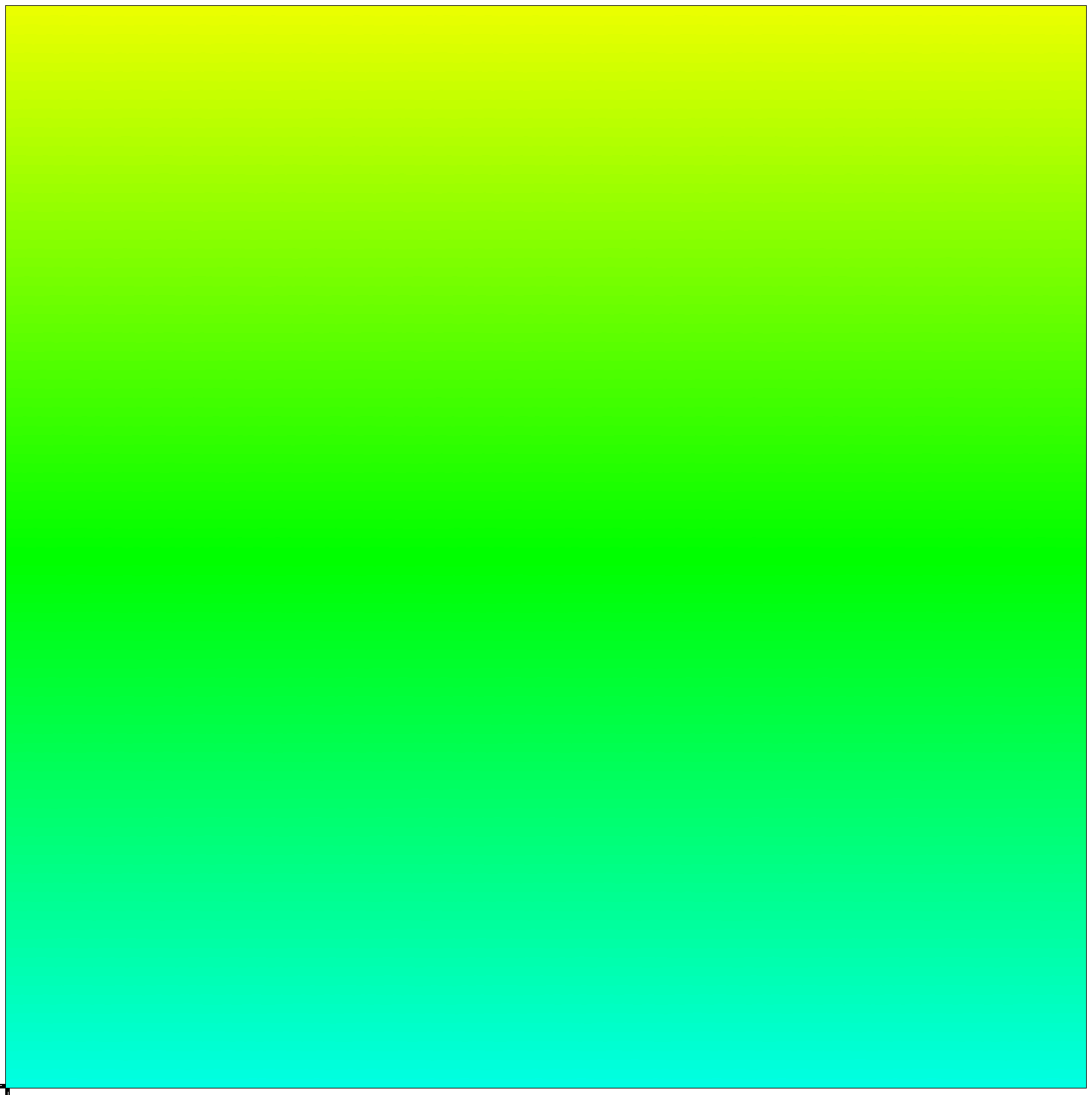}
\includegraphics[width=0.22\textwidth]{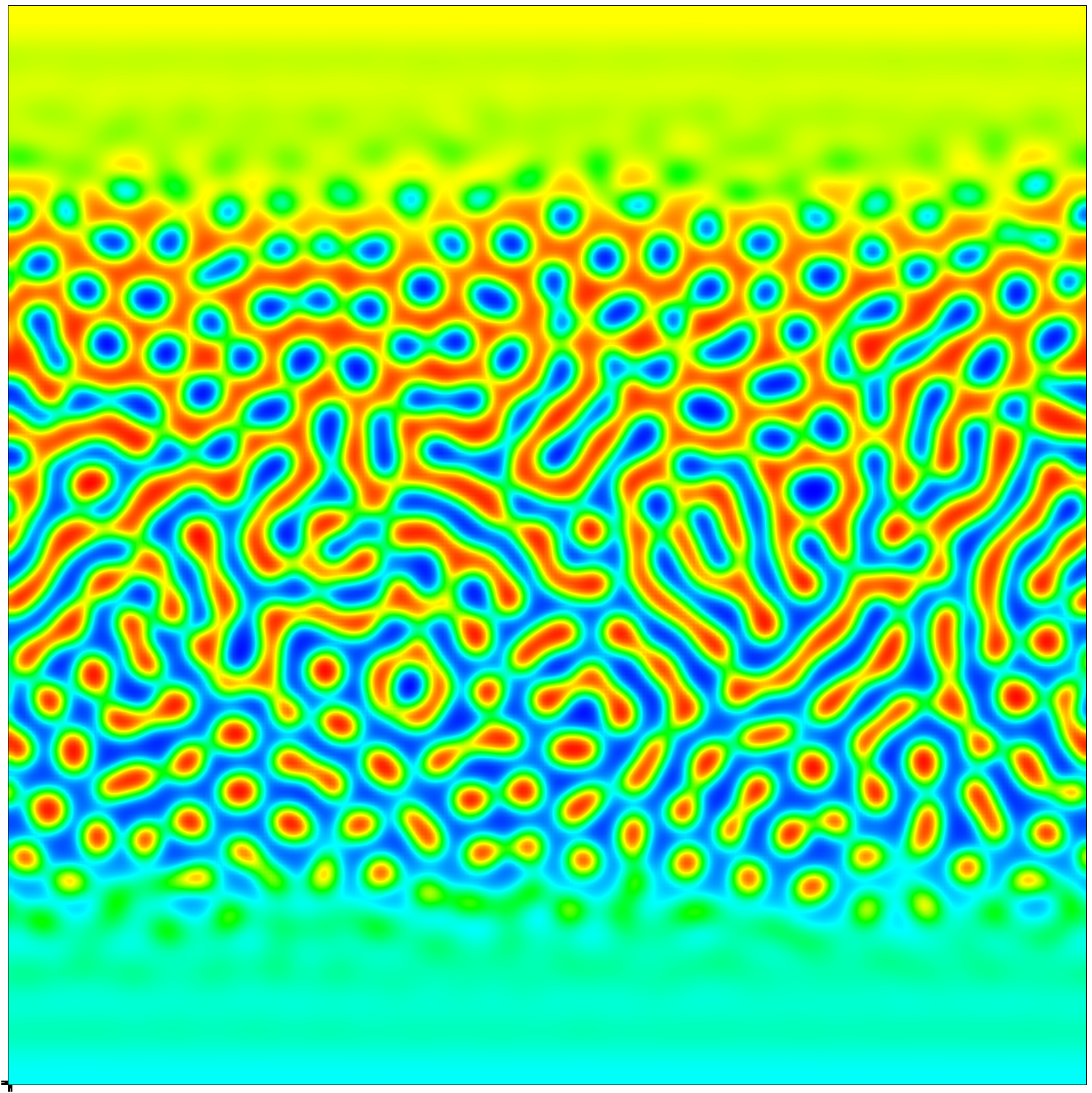}
\includegraphics[width=0.22\textwidth]{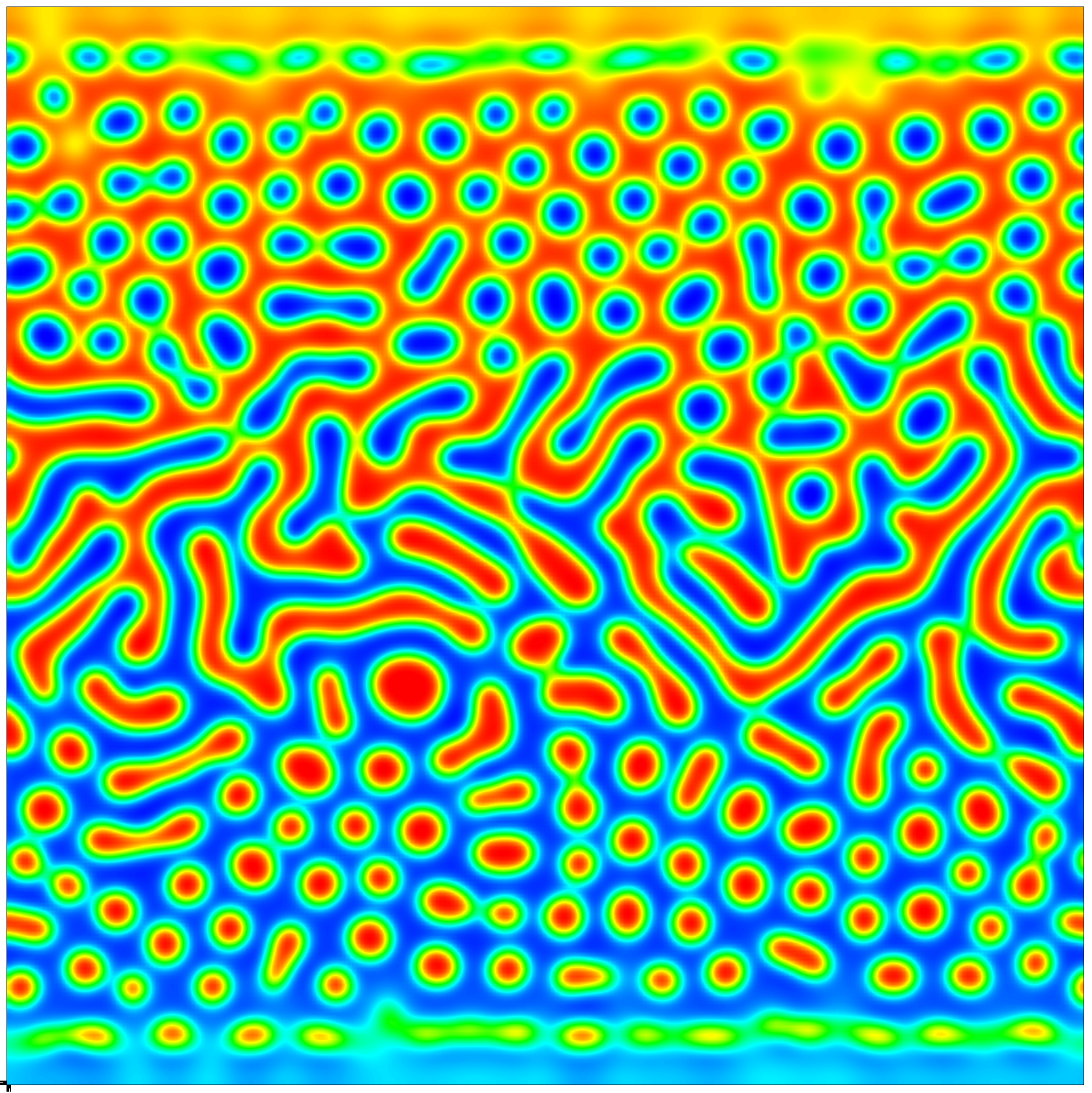}
\includegraphics[width=0.22\textwidth]{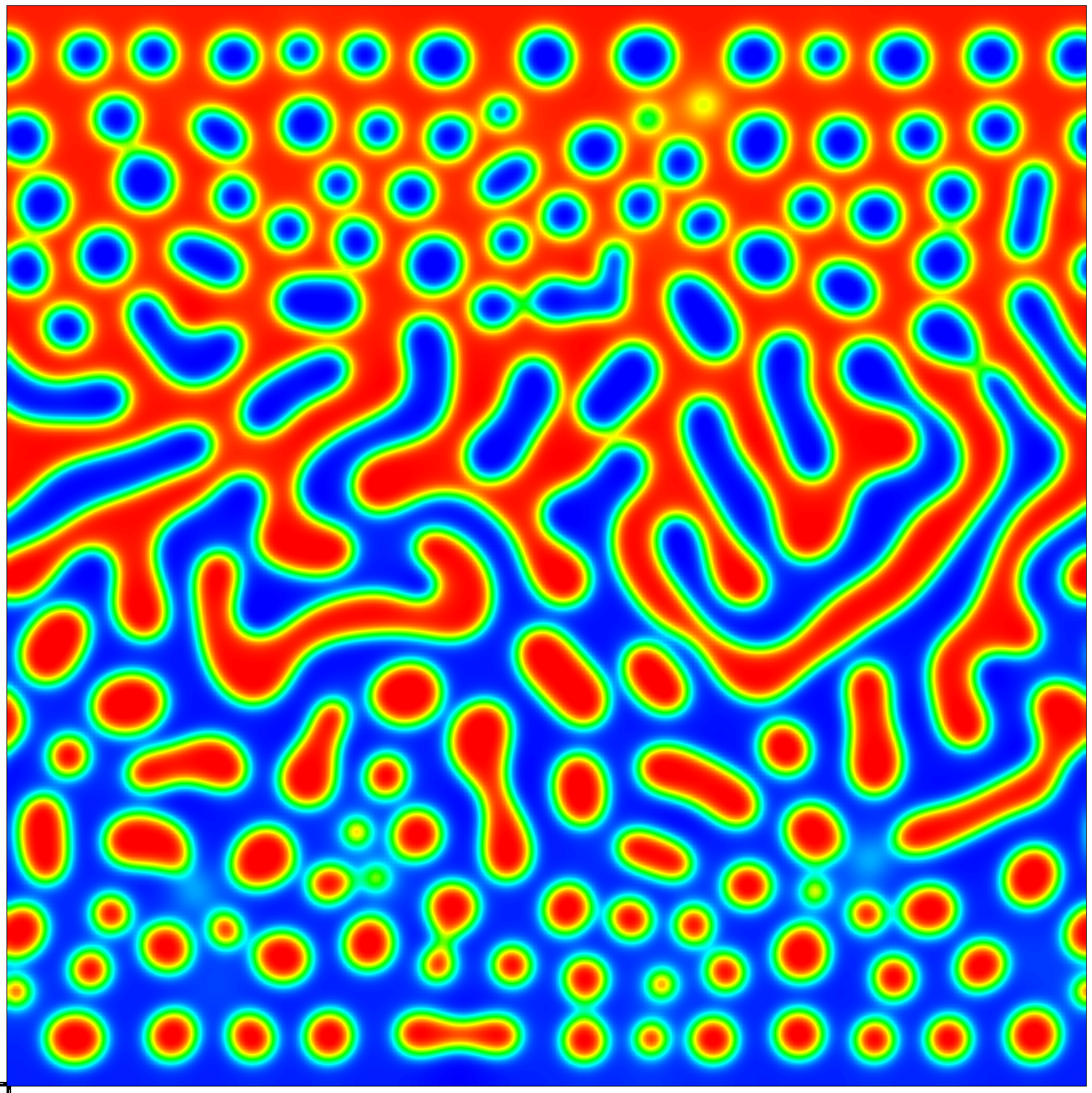}
}

\subfigure[profiles of $\phi$ at $t=5,10,20,100$]{
\includegraphics[width=0.22\textwidth]{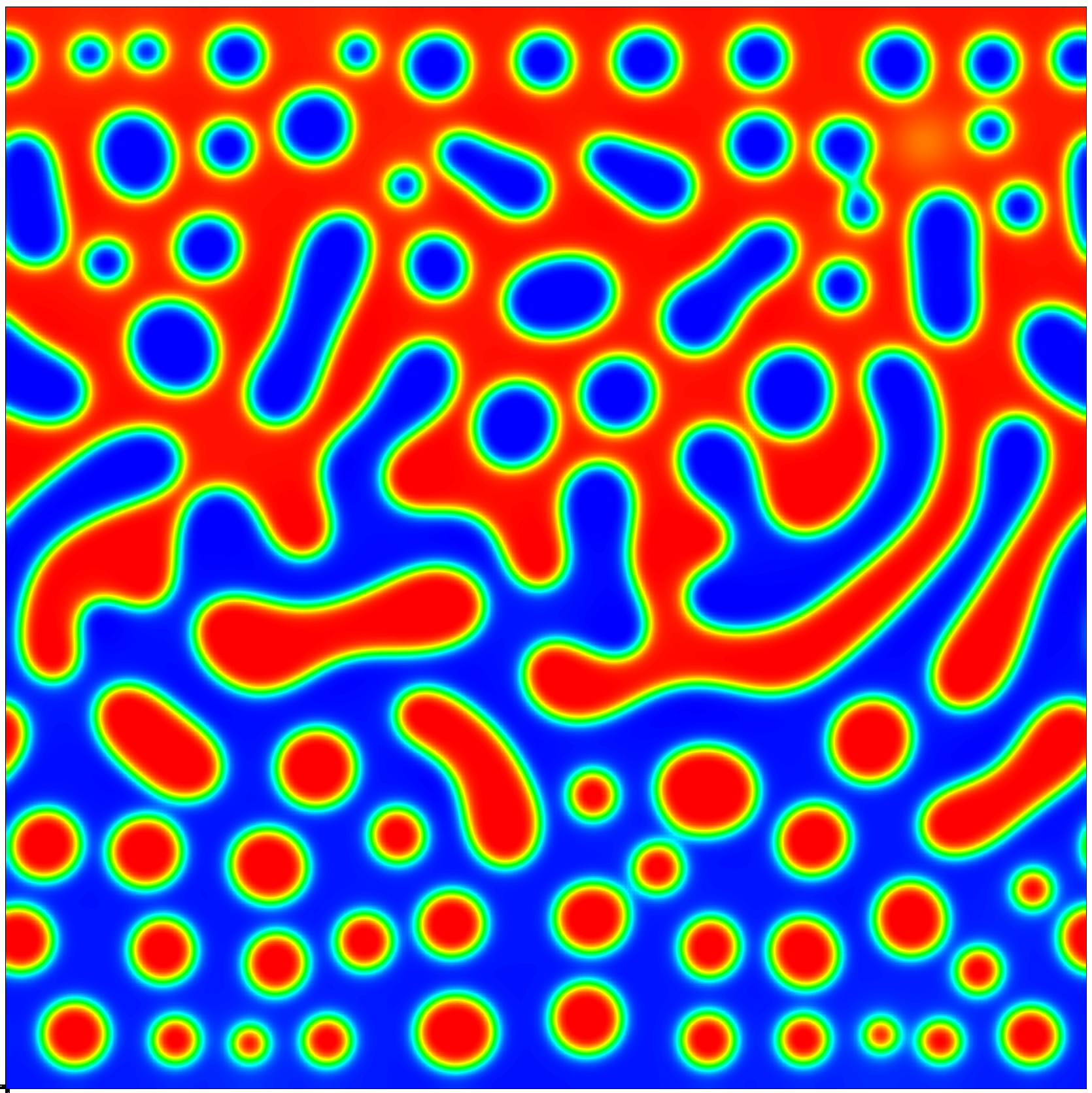}
\includegraphics[width=0.22\textwidth]{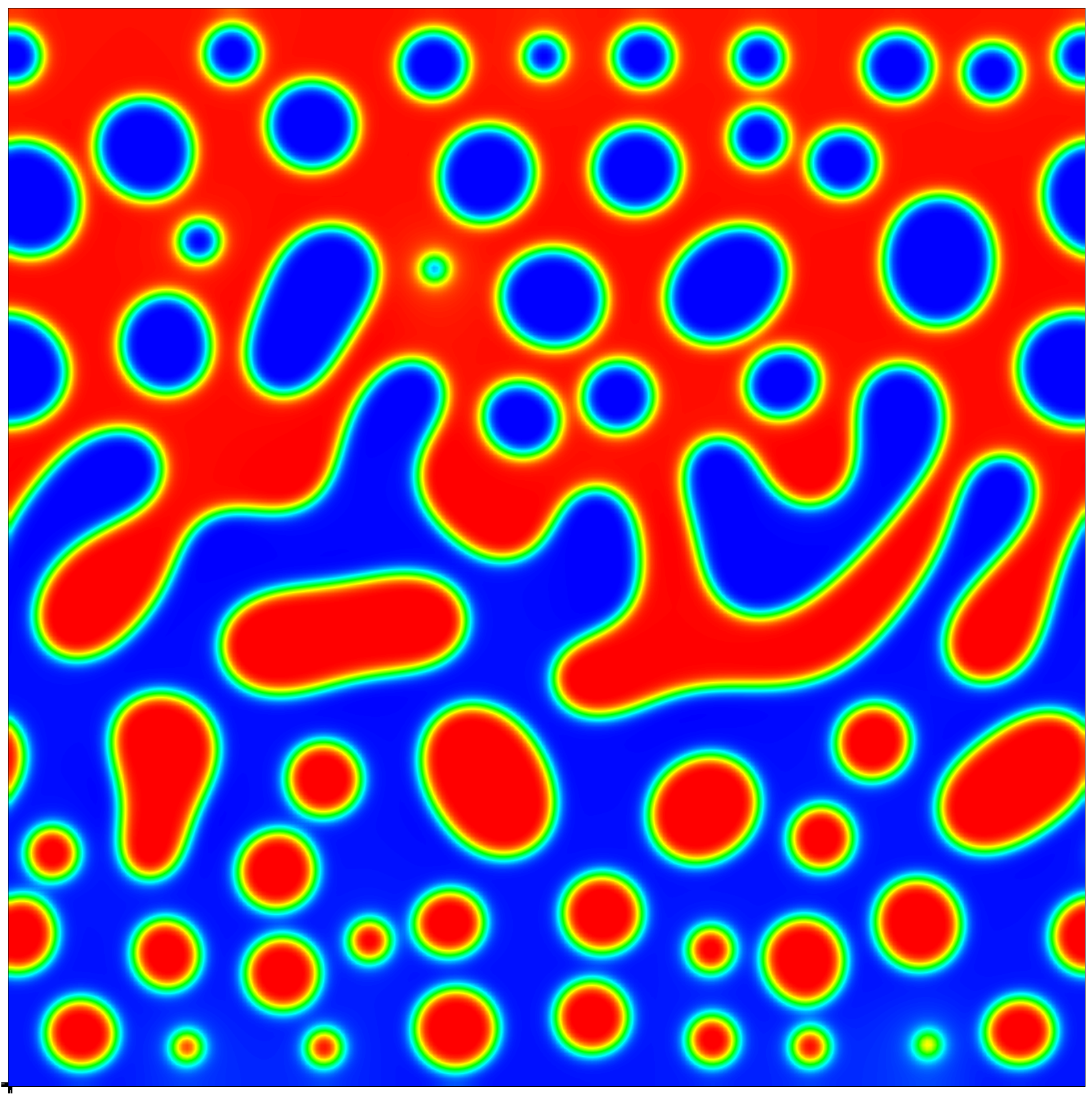}
\includegraphics[width=0.22\textwidth]{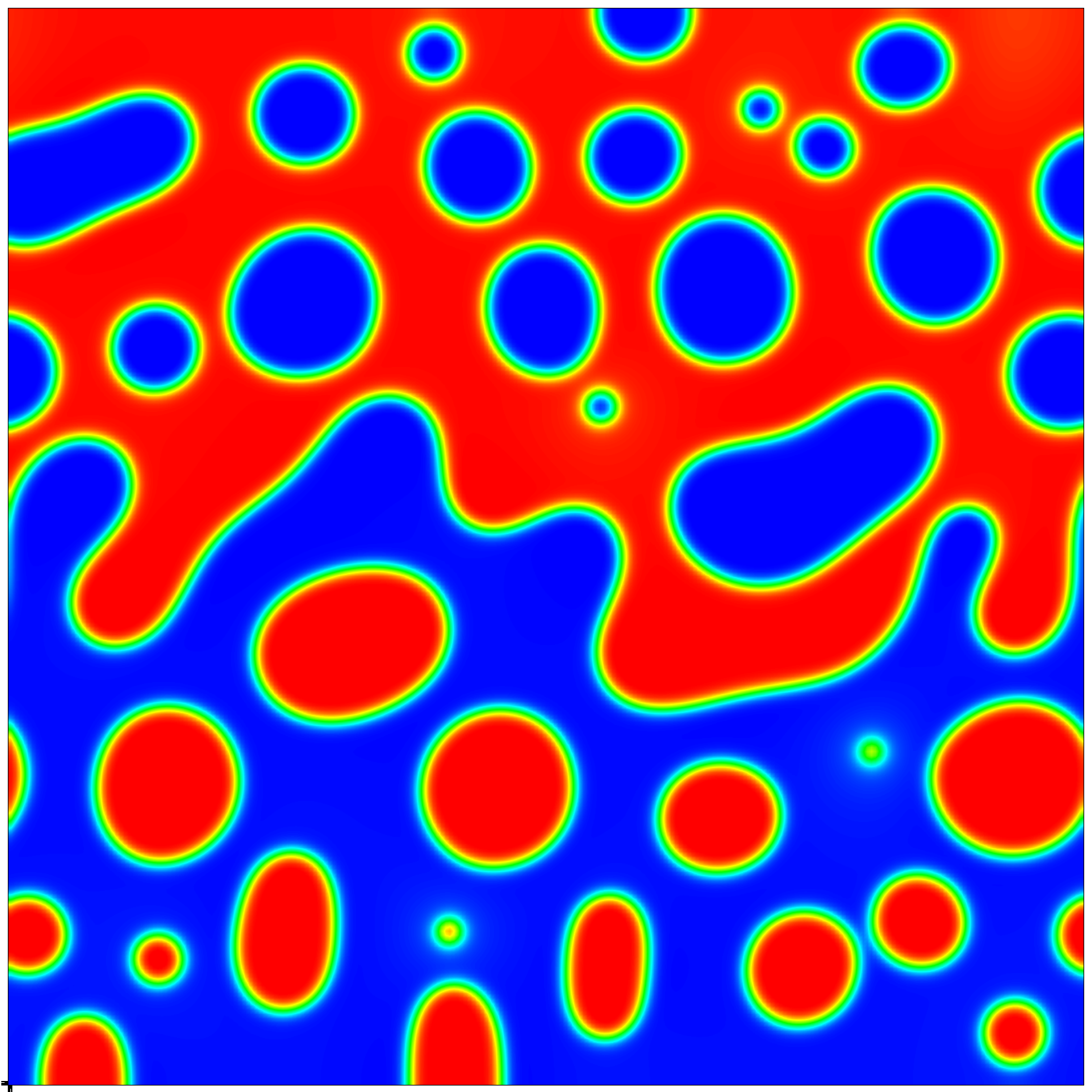}
\includegraphics[width=0.22\textwidth]{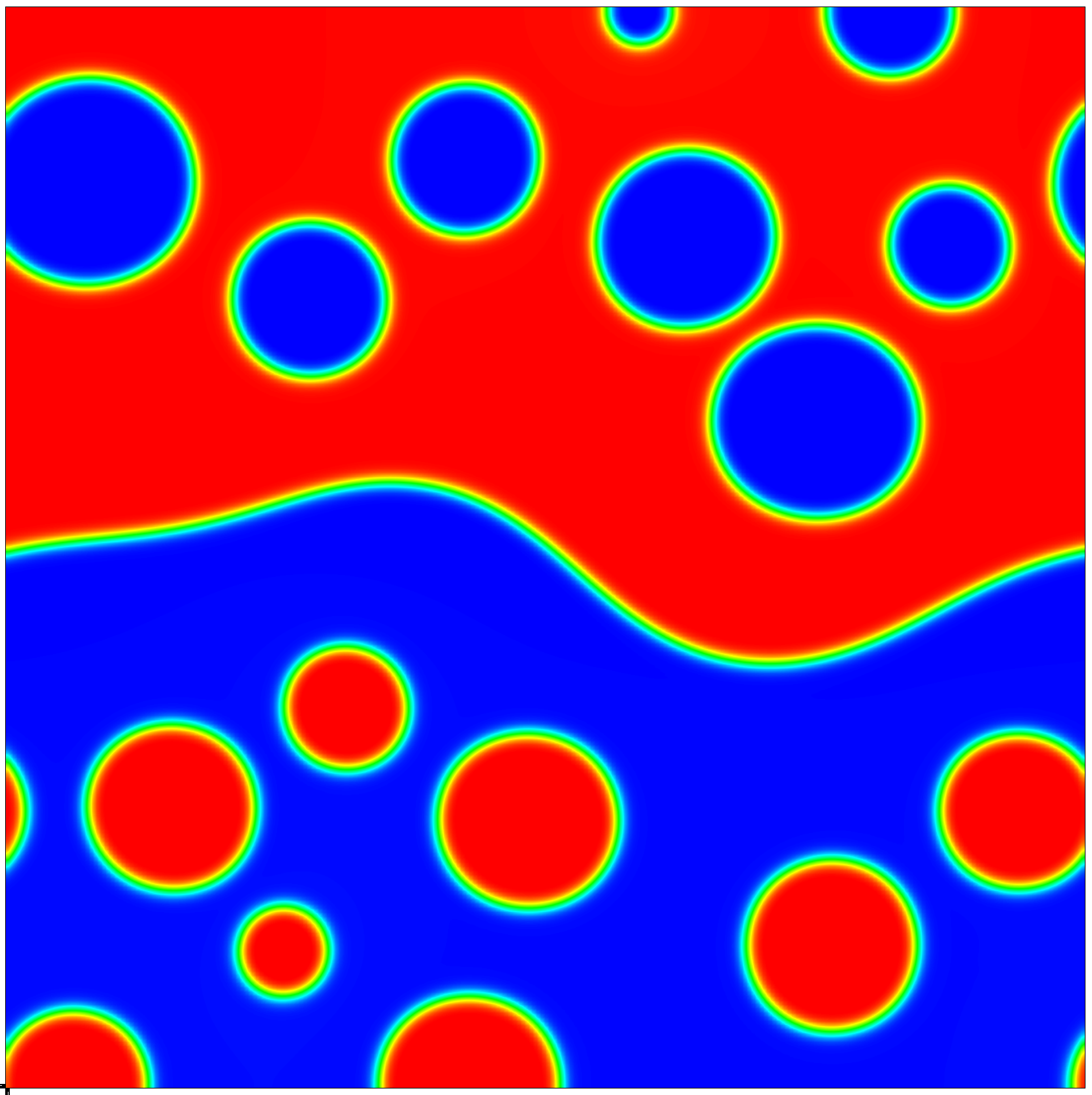}
}

\caption{Coarsening dynamics driven by the Cahn-Hilliard-Navier-Stokes equations. This figures shows that spinodal decomposition and nucleation happen simultaneously in different locations due to the difference of volume fraction ratios.}
\label{fig:CHNS-Coarsening}
\end{figure}

In addition, we further visualize the energy evolution $\cE(t)$ and the numerical solution for $s(t)$. These are summarized in Figure \ref{fig:CHNS-E-S}, highlighting the accuracy and energy stable property of Scheme \ref{scheme:CHNS-steps-fullly} on solving the Cahn-Hilliard-Navier-Stokes equations.

\begin{figure}
\center
\subfigure[caption]{\includegraphics[width=0.45\textwidth]{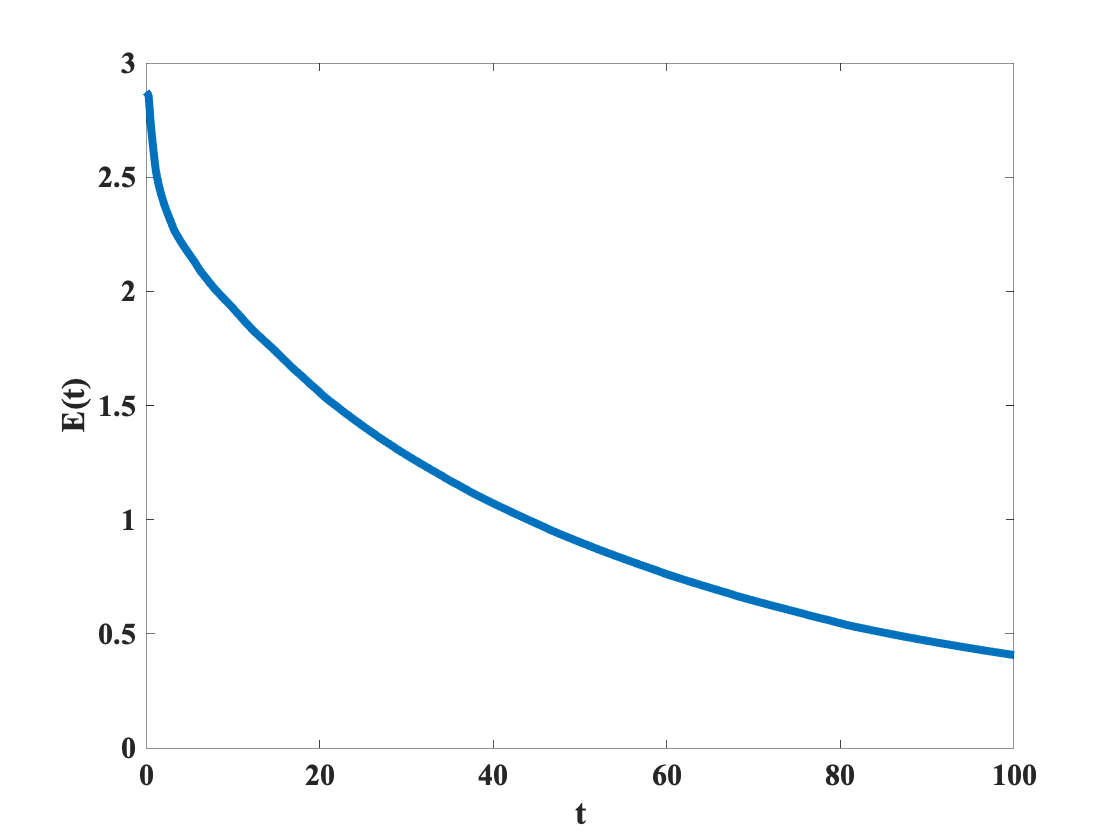}}
\subfigure[caption]{\includegraphics[width=0.45\textwidth]{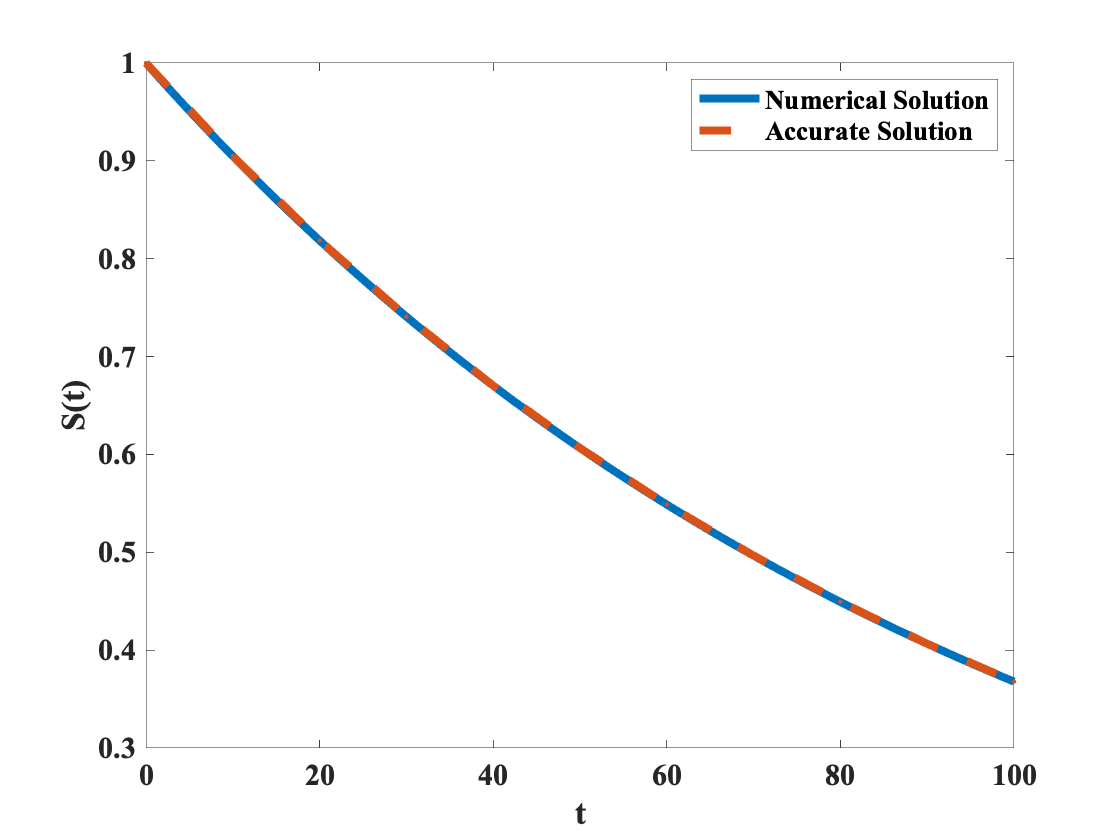}}
\caption{This figure shows the time evolution of energy $\cE(t)$ and auxiliary variable $s(t)$ for the simulations in Figure \ref{fig:CHNS-Coarsening}.  In (a), the numerical calculated energy is decreasing in time, which agrees well with the energy-stable theoretical results. In (b), the numerical solution of $S(t)$ accurately approximates its original definition $e^{-\frac{t}{T}}$.}
\label{fig:CHNS-E-S}
\end{figure}

In the second example, we use Scheme \ref{scheme:CHNS-steps-fullly} to investigate the Ostwald ripening dynamics. Here we choose an initial condition that contains several drops, but with different radii. A similar problem has been used as a benchmark problem for the Cahn-Hilliard equation in \cite{GuoBenchmark}. Here we choose the domain $\Omega=[0, 1]^2$, and parameters $\rho=1$, $\eta =1$, $\varepsilon=0.01$, $T=100$.
In the implementation,  we use $256^2$ meshes and the time step $\delta =0.005$. The numerical results are summarized in Figure \ref{fig:CHNS-SevenBalls}, where the Ostwald ripening dynamics are observed. 

\begin{figure}
\center

\subfigure[profiles of $\phi$ at $t=0,5,10, 15$]{
\includegraphics[width=0.22\textwidth]{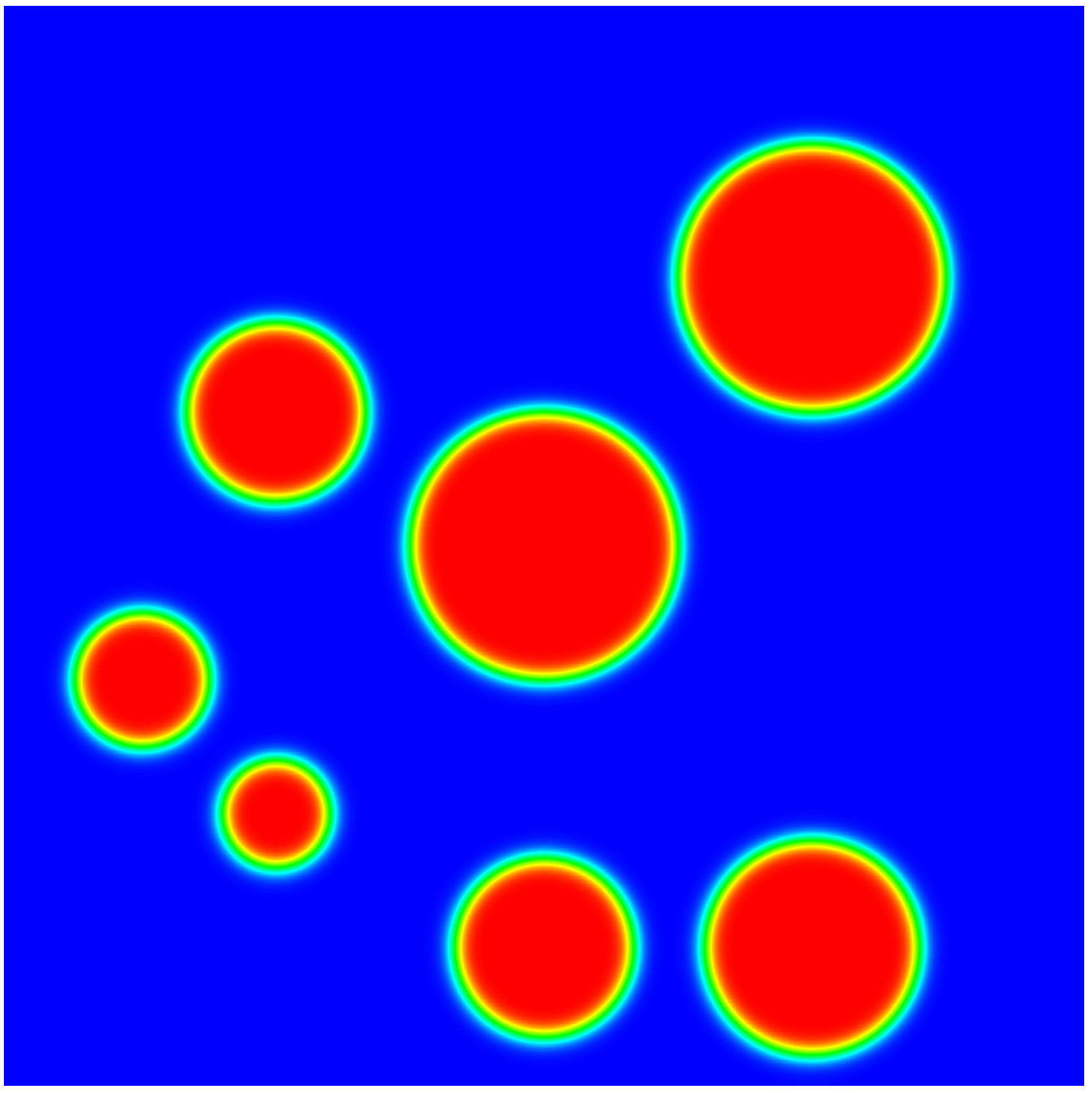}
\includegraphics[width=0.22\textwidth]{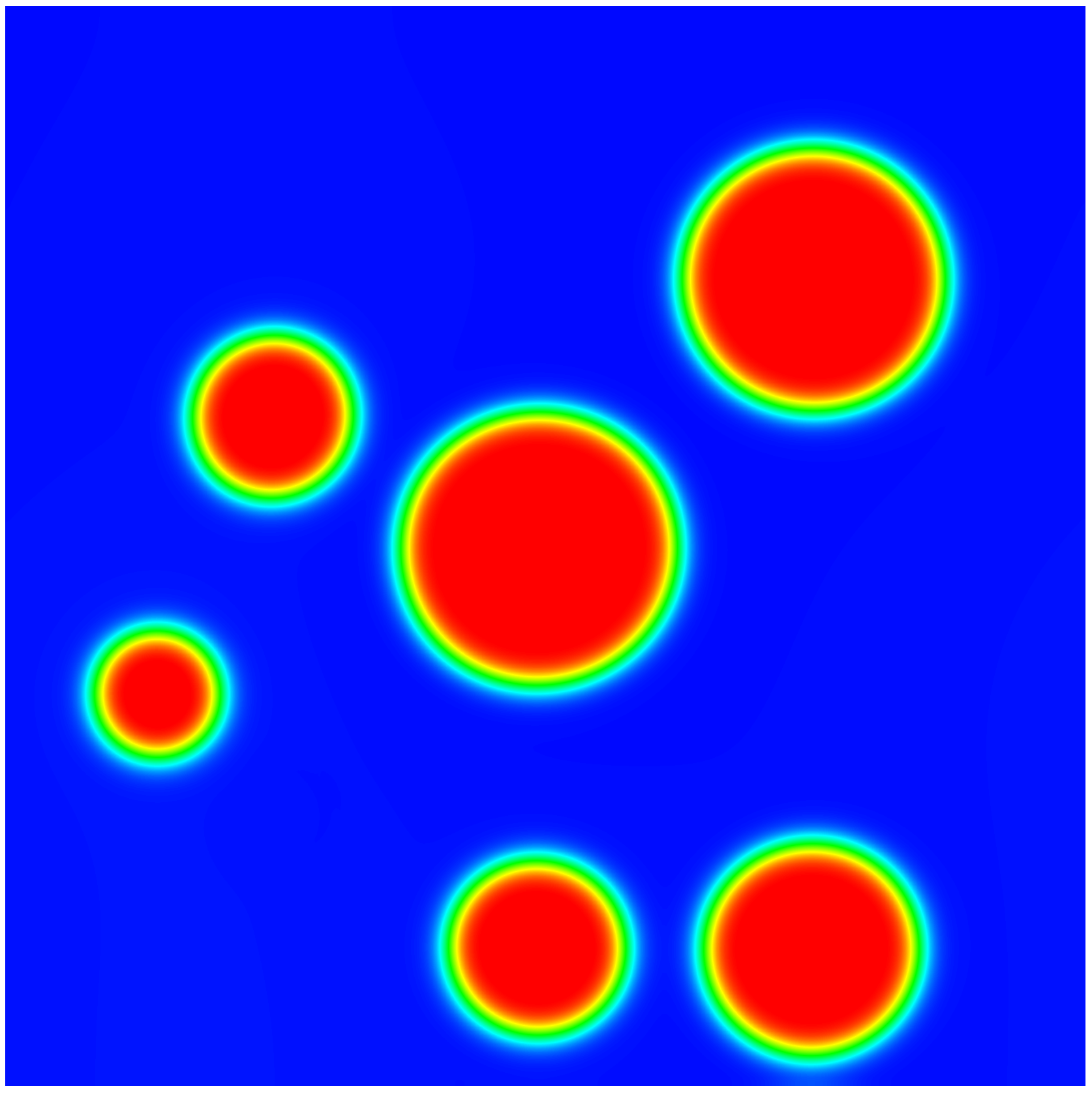}
\includegraphics[width=0.22\textwidth]{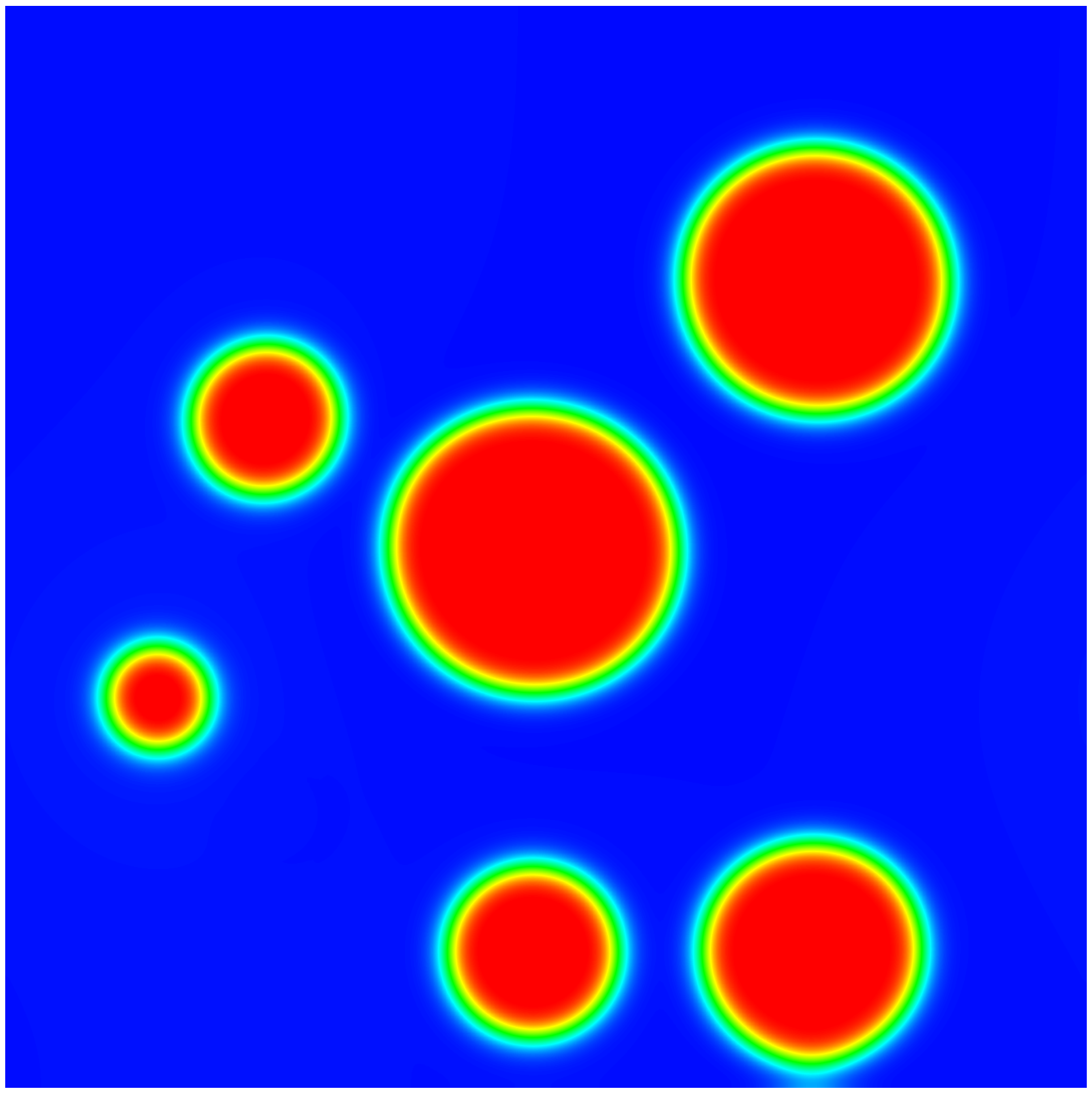}
\includegraphics[width=0.22\textwidth]{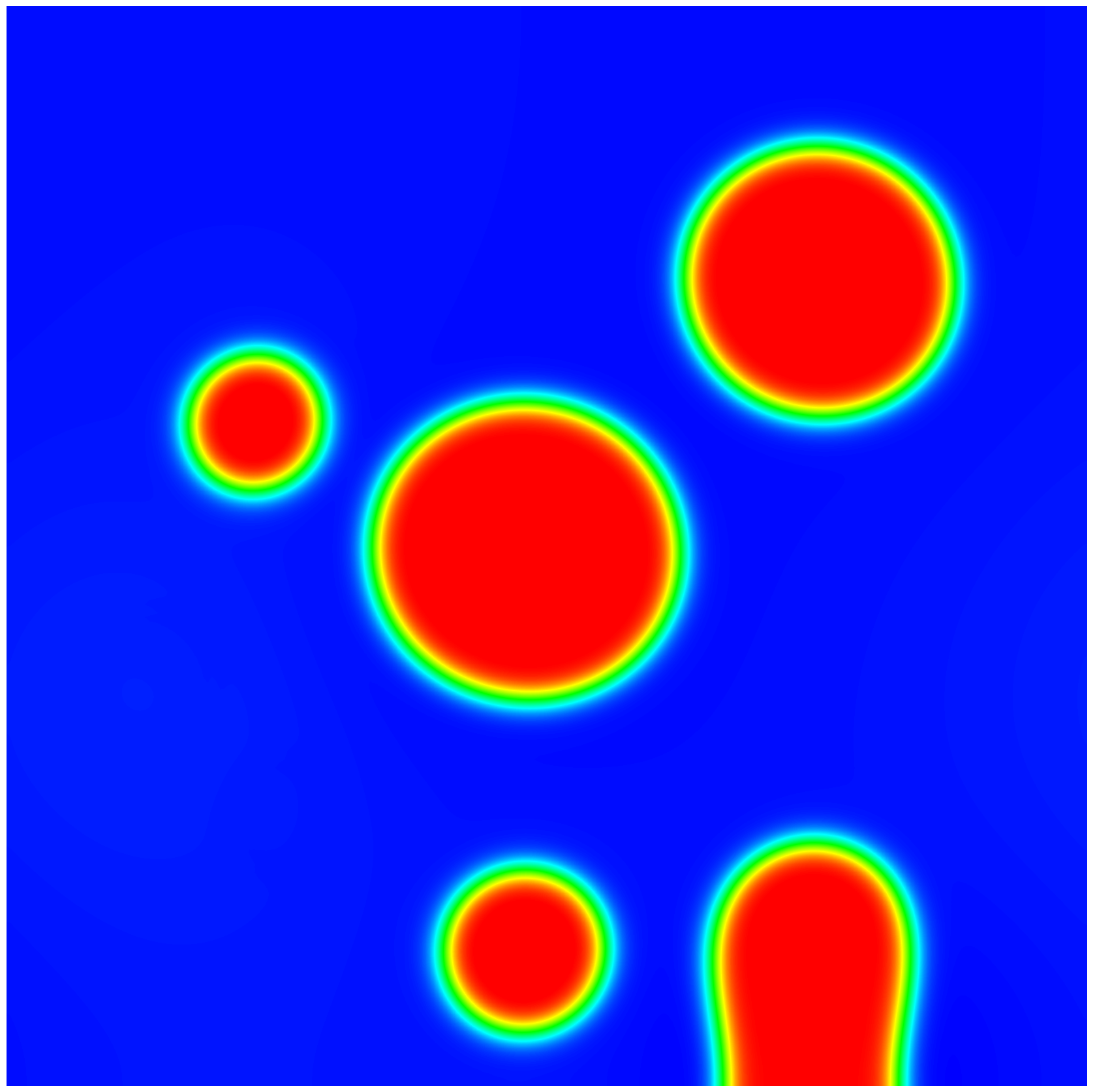}
}

\subfigure[profiles of $\phi$ at $t=20,30,50, 100$]{
\includegraphics[width=0.22\textwidth]{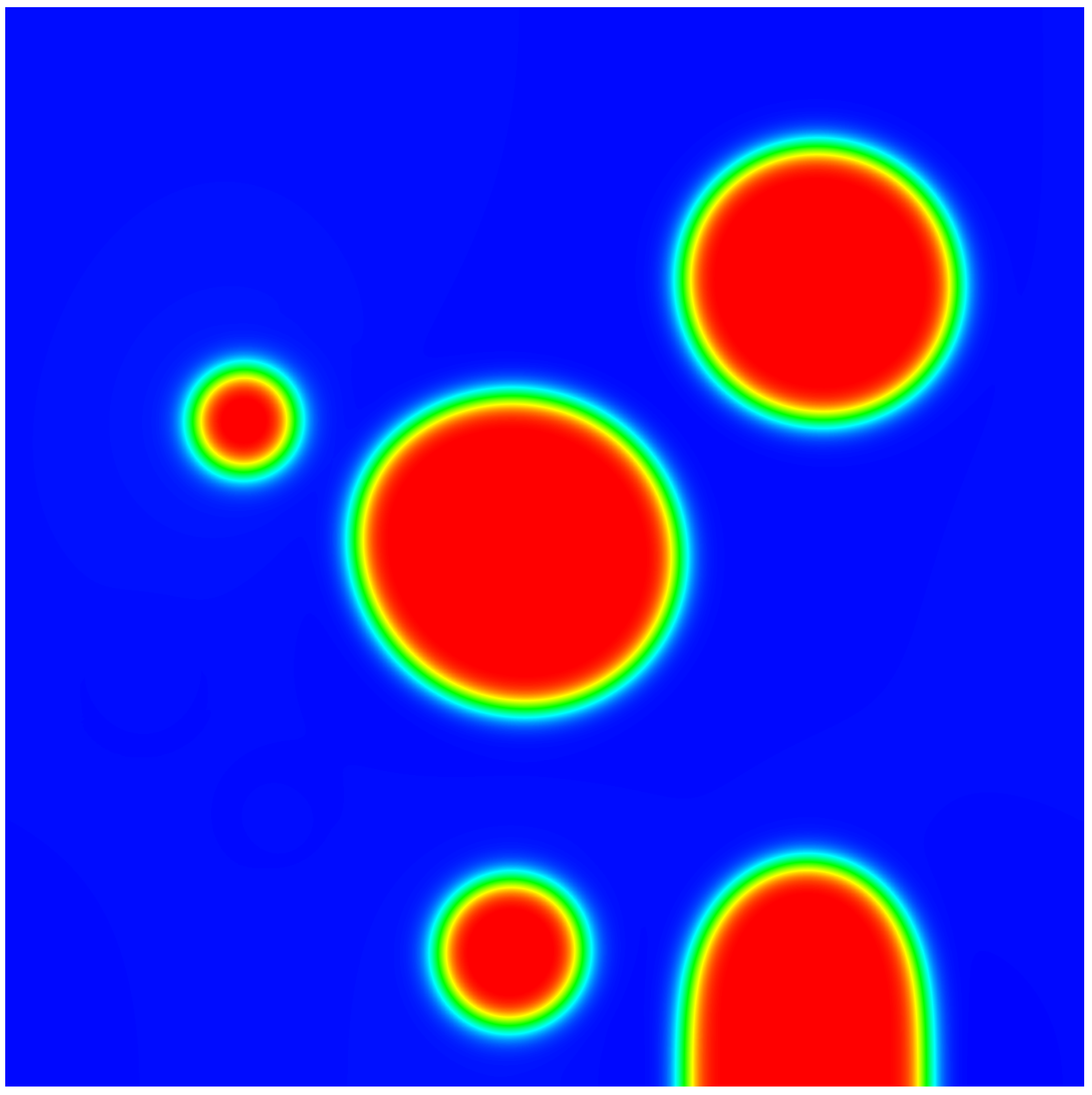}
\includegraphics[width=0.22\textwidth]{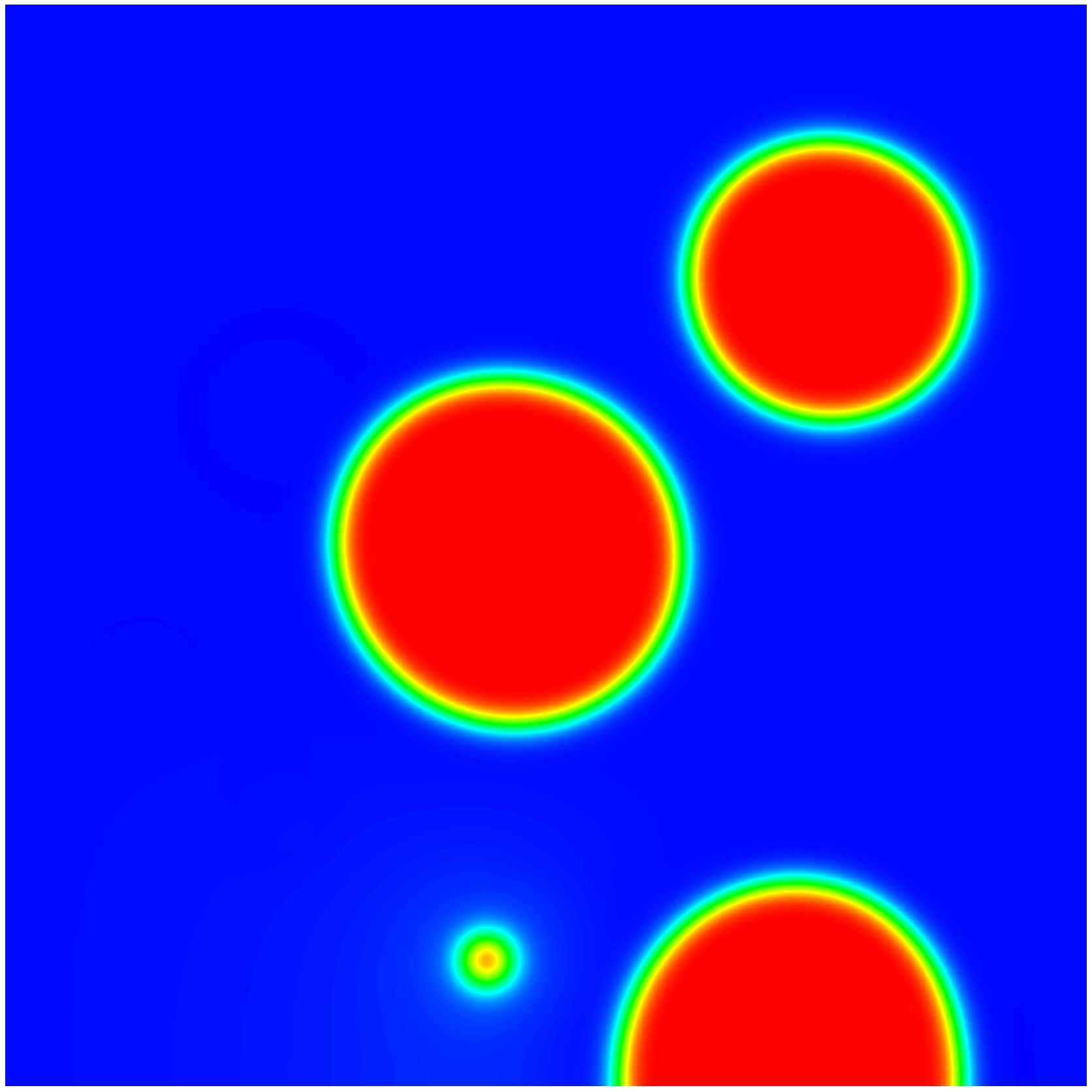}
\includegraphics[width=0.22\textwidth]{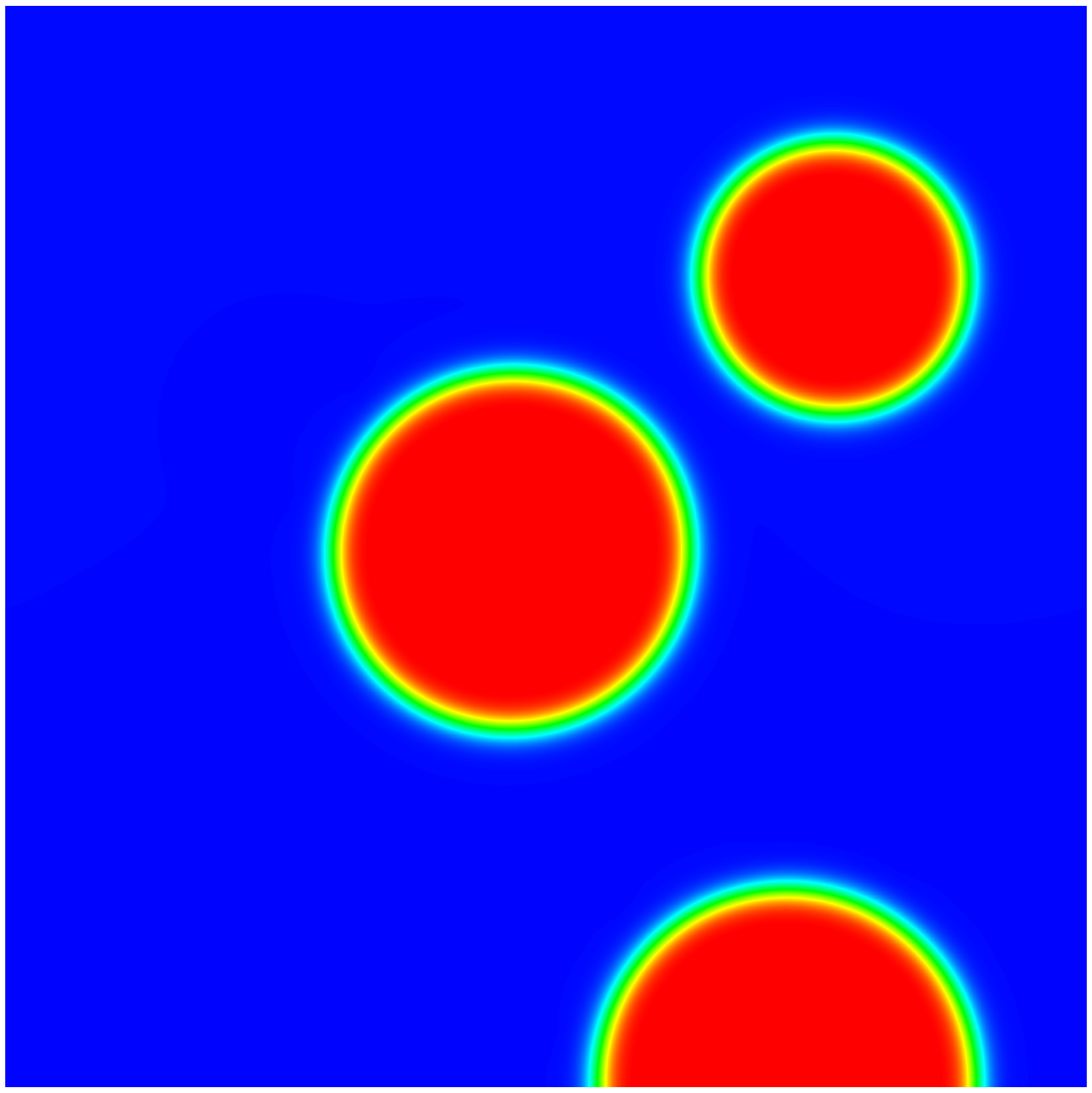}
\includegraphics[width=0.22\textwidth]{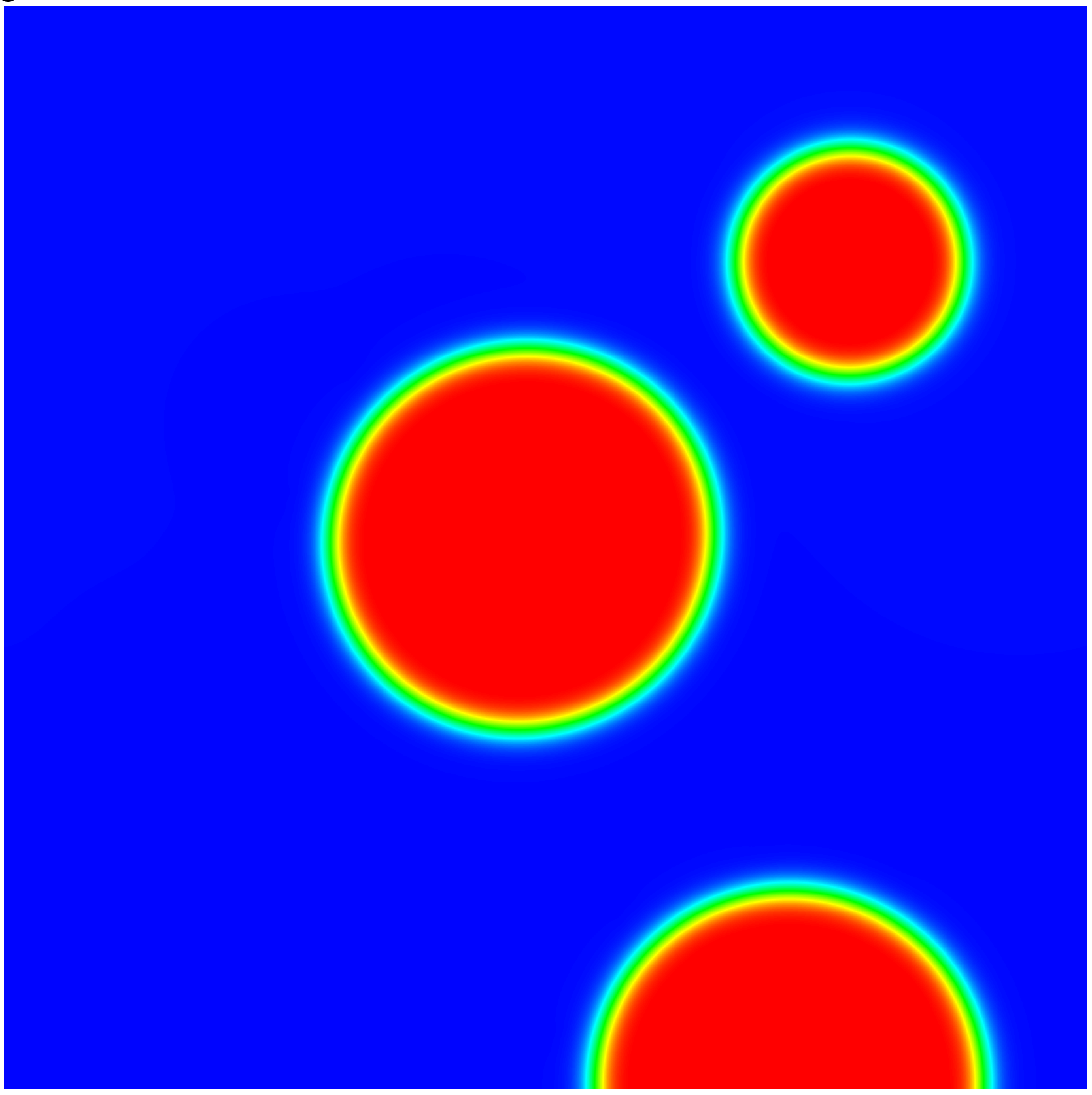}
}
\caption{Oswald ripening that is driven by the Cahn-Hilliard-Navier-Stokes system. In this figure, the profiles of $\phi$ at various times are visualized. Here red represents $1$, and blue represents $-1$.}
\label{fig:CHNS-SevenBalls}
\end{figure}

\subsection{Numerical examples for the Ericksen-Leslie liquid crystal models}
In this sub-section, we use the fully decoupled scheme \ref{scheme:LC-Ericksen-Lesile-fully} or \ref{scheme:LC-Ericksen-Lesile-steps-fully} to investigate the Ericksen-Leslie hydrodynamic model for the nematic liquid crystal fluid flow. First of all, we verify the second-order time convergence of Scheme \ref{scheme:LC-Ericksen-Lesile-fully}. We consider the domain $\Omega=[0, L]^2$ with $L=1$, the initial condition condition
$$
\bp(x, y, t=0) = (0.01 \cos(2\pi y) \cos(2\pi x), 0.01 \cos(2\pi y) \cos(2\pi z)).
$$
and we  choose the parameters $T = 0.5$, $\rho = 1$, $\eta = 10^2$, $\varepsilon^2 = 0.1$, $a = 1.2$, $K = 10^{-2}$ and $\gamma_0 = 0$.. We fix the uniform meshes as $128^2$, and use various time steps $\delta t = 10^{-2} \frac{1}{2^k}$, $k=0,1,\cdots$. Following the same procedure as the previous subsection, we calculate the errors as the difference between the numerical solution at the current time step and the numerical solution with the adjacent finner time step. Both the $l_2$ norm and $l_\infty$ norm for the numerical errors are summarized in Figure \ref{fig:LCP-mesh-refinement}. It can be easily observed that Scheme \ref{scheme:LC-Ericksen-Lesile-steps-fully} provides second-order accuracy in time.

\begin{figure}
\center
\subfigure[$l_2$ error]{\includegraphics[width=0.475\textwidth]{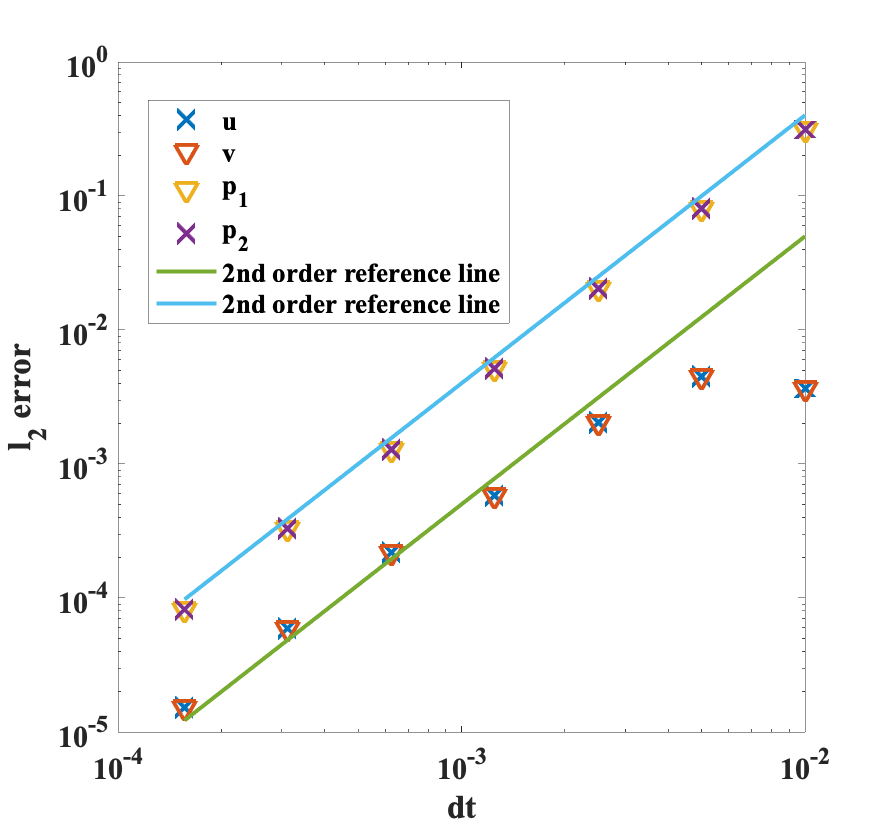}}
\subfigure[$l_\infty$ error]{\includegraphics[width=0.475\textwidth]{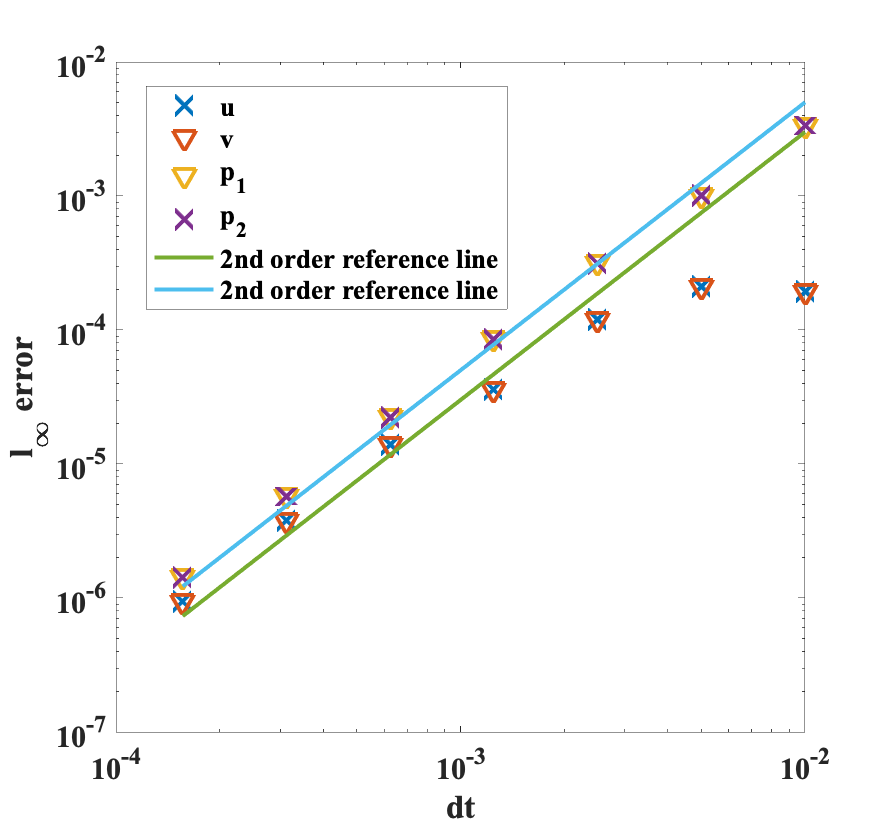}}
\caption{Time step mesh-refinement for Scheme \ref{scheme:LC-Ericksen-Lesile-steps-fully}. Here we denote $\bu=(u,v)$ and $\bp=(\bp_1, \bp_2)$. This figure shows that Scheme \ref{scheme:CHNS-steps-fullly} is second-order accurate in both the $l_2$ norm and $l_\infty$ norm.}
\label{fig:LCP-mesh-refinement}
\end{figure}

With the Scheme \ref{scheme:LC-Ericksen-Lesile-steps-fully}, we further conduct some benchmark simulations. In this example, we consider a rectangular domain $\Omega=[0 , 2]\times [0, 4]$, and set up several defects at the starting time, as shown in Figure \ref{fig:LCP-bp}(a). The parameters are chosen as $T=100$,  $\rho = 1$,  $\eta = 10^2$, $\varepsilon^2 = 0.1$, $a = 1.2$, $K = 10^{-2}$,  and $\gamma_0 = 0$.
We choose the Neumann boundary condition for the liquid crystal $\bp$. We use $128 \times 256$ meshes and $\delta t= 10^{-3}$  for the simulation. It is known that the defects are unstable in this case and they will annihilate by cancellation or relaxing out from the boundary. 

The results for the evolution dynamics of $\bp$ are summarized Figure \ref{fig:LCP-bp}. We do observe that the point $+2$ ($-2$) defect separates into two of $+1$ ($-1$) defects, as shown in Figure \ref{fig:LCP-bp}(b). Then a point $+1$ defect and a point $-1$ defect annihilate, as shown in Figure \ref{fig:LCP-bp}(d)-(e). Eventually the other defects relax out of the domain, since there is no anchoring at the boundaries, as shown in Figure \ref{fig:LCP-bp}(h)-(i).

\begin{figure}
\subfigure[$t=0$]{\includegraphics[width=0.33\textwidth]{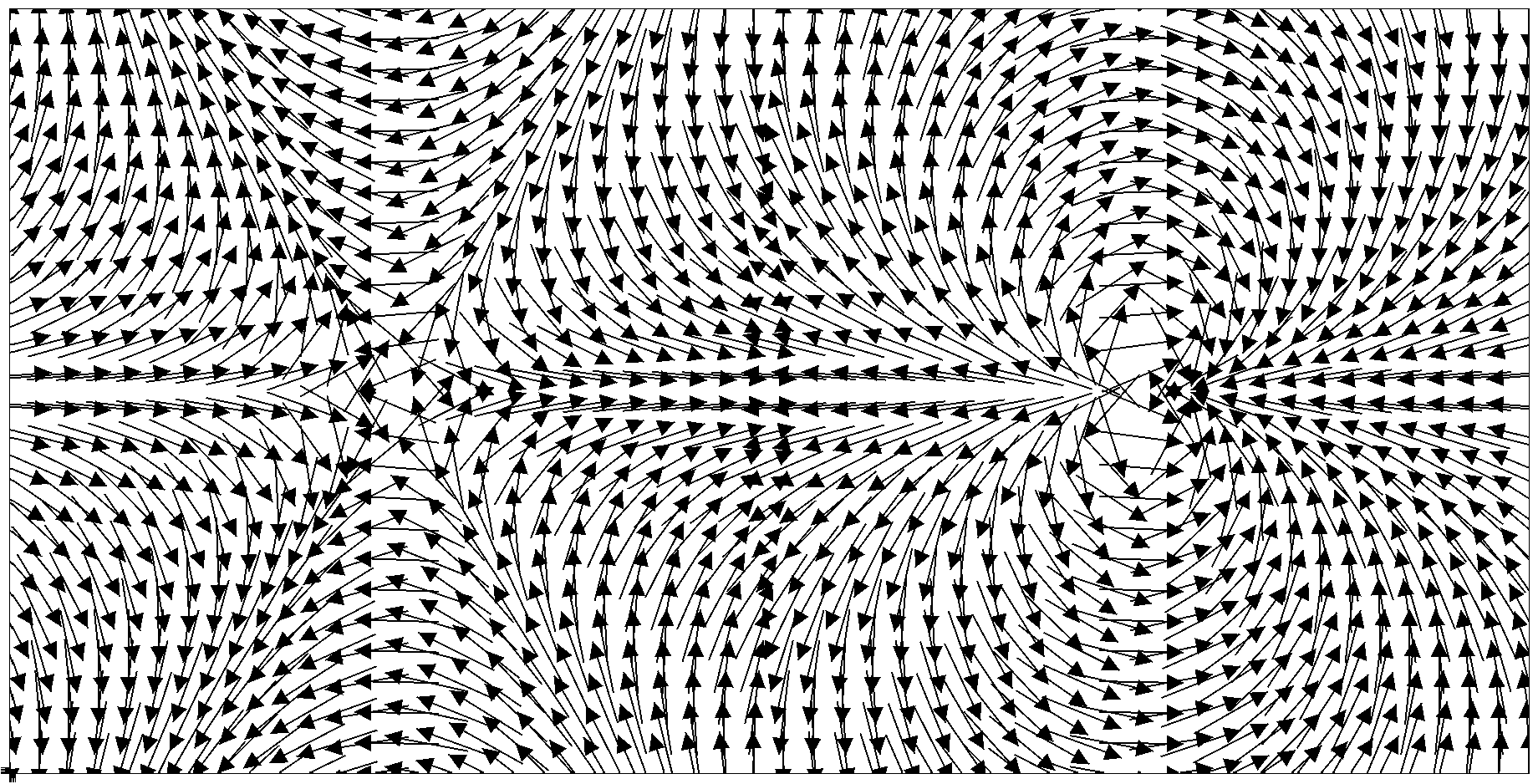}}
\subfigure[$t=2$]{\includegraphics[width=0.33\textwidth]{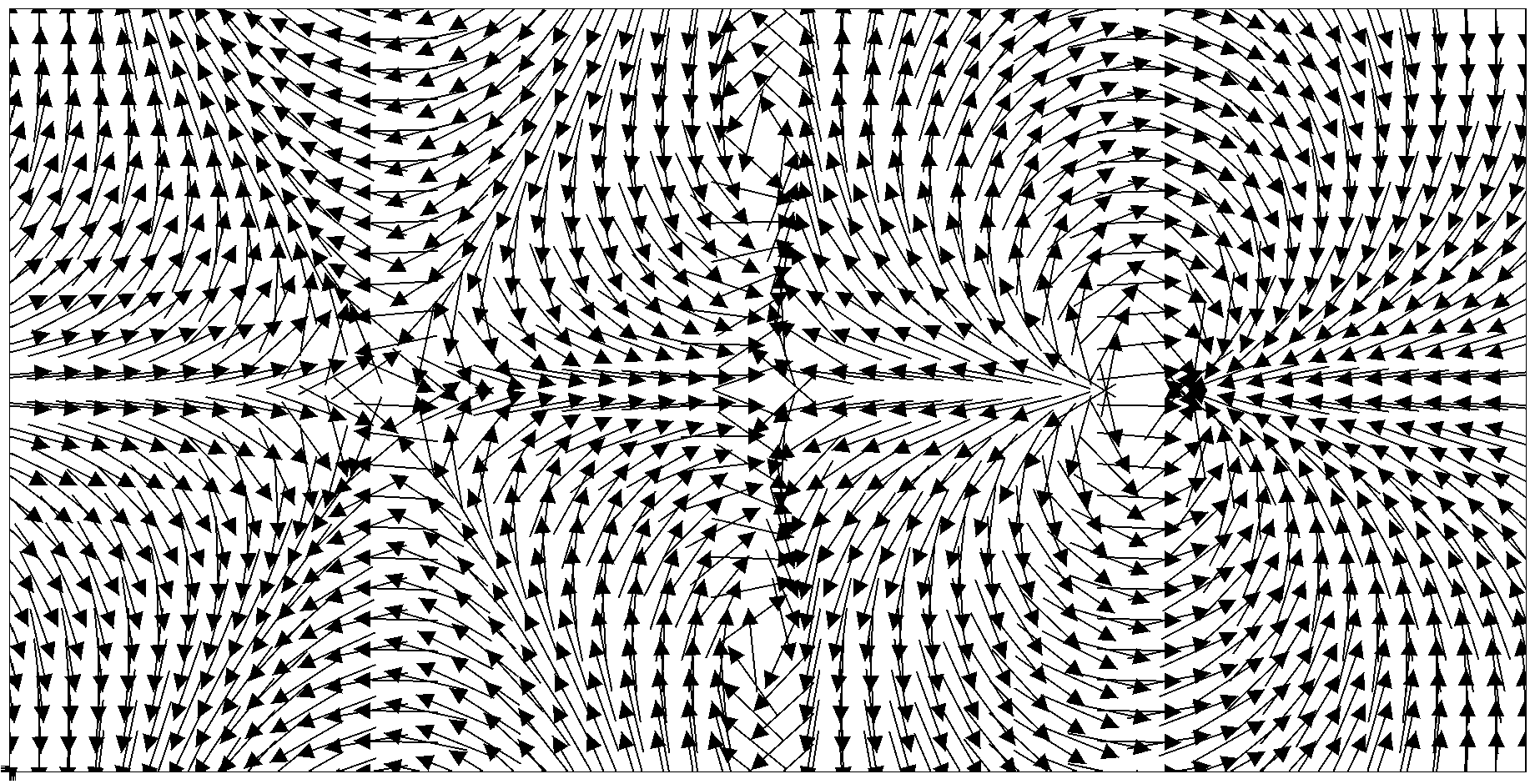}}
\subfigure[$t=5$]{\includegraphics[width=0.33\textwidth]{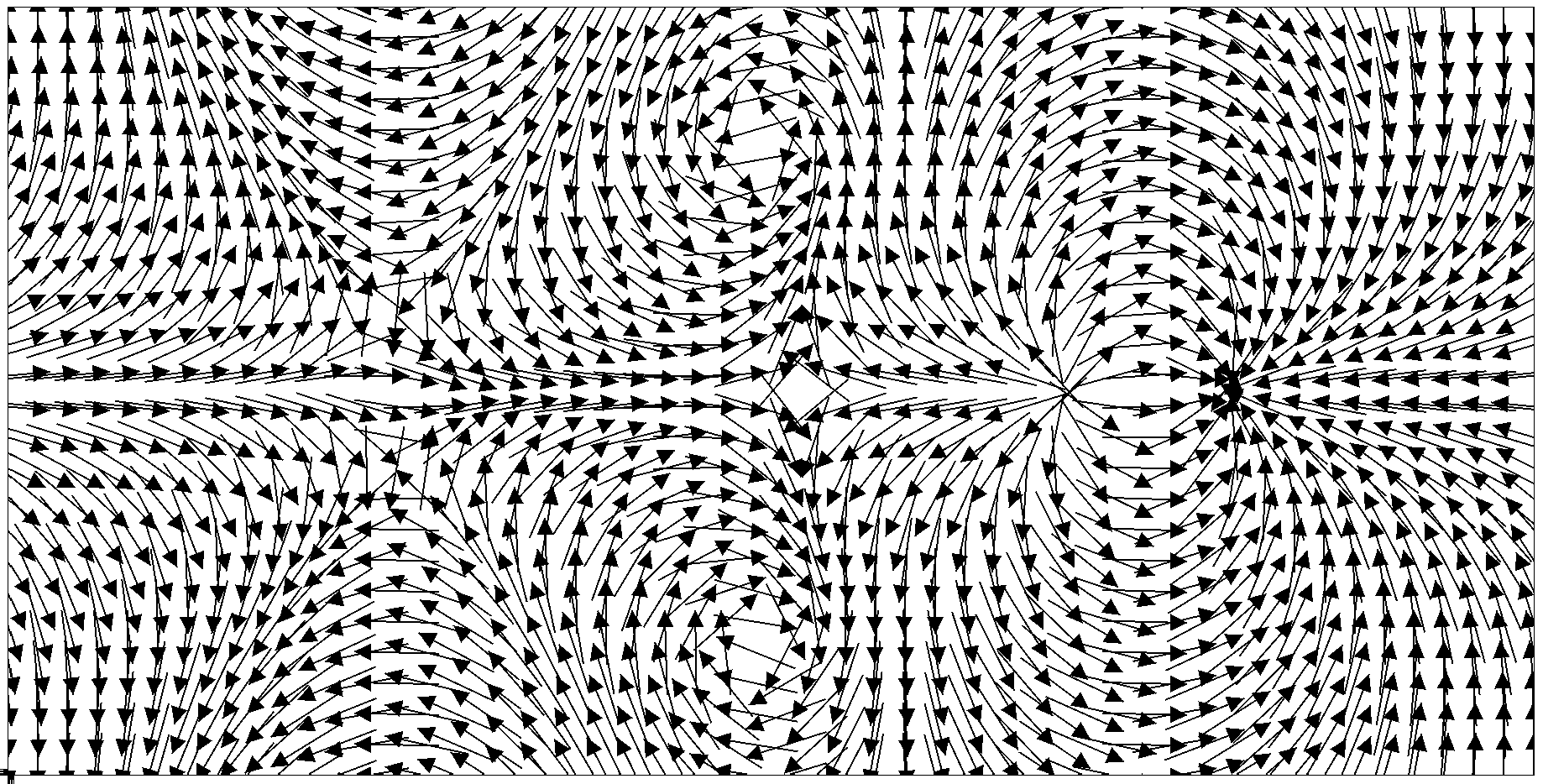}}

\subfigure[$t=17.5$]{\includegraphics[width=0.33\textwidth]{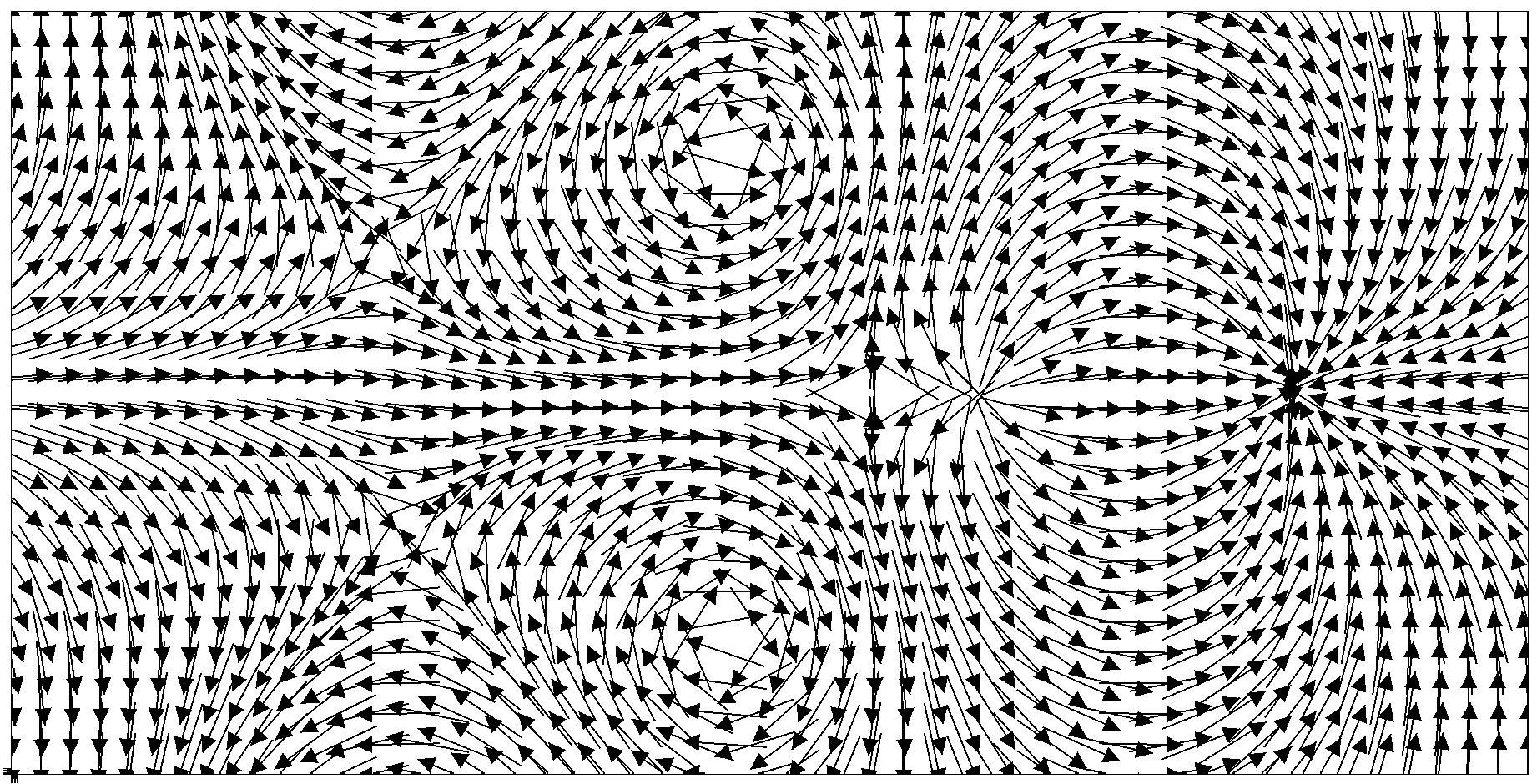}}
\subfigure[$t=20$]{\includegraphics[width=0.33\textwidth]{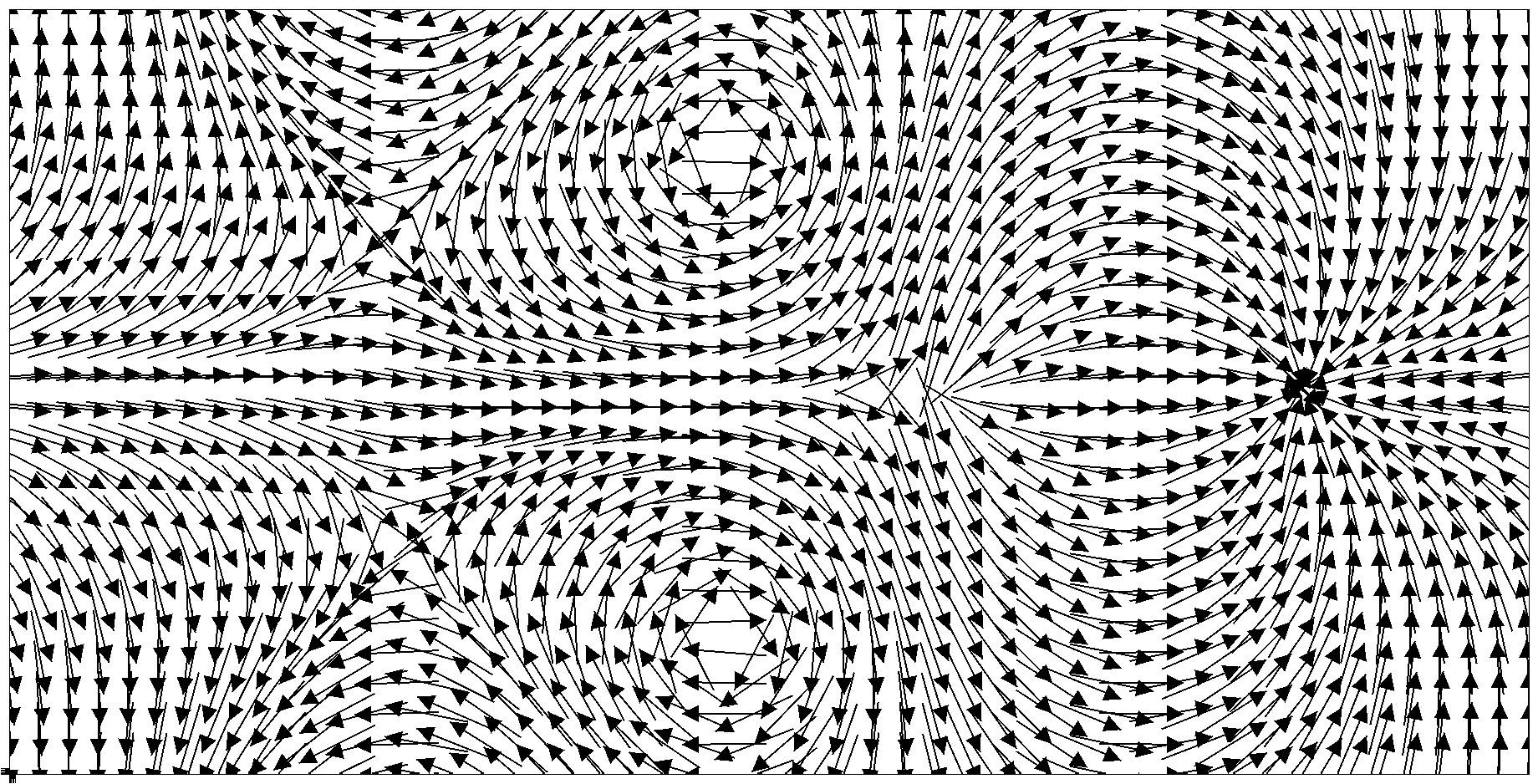}}
\subfigure[$t=25$]{\includegraphics[width=0.33\textwidth]{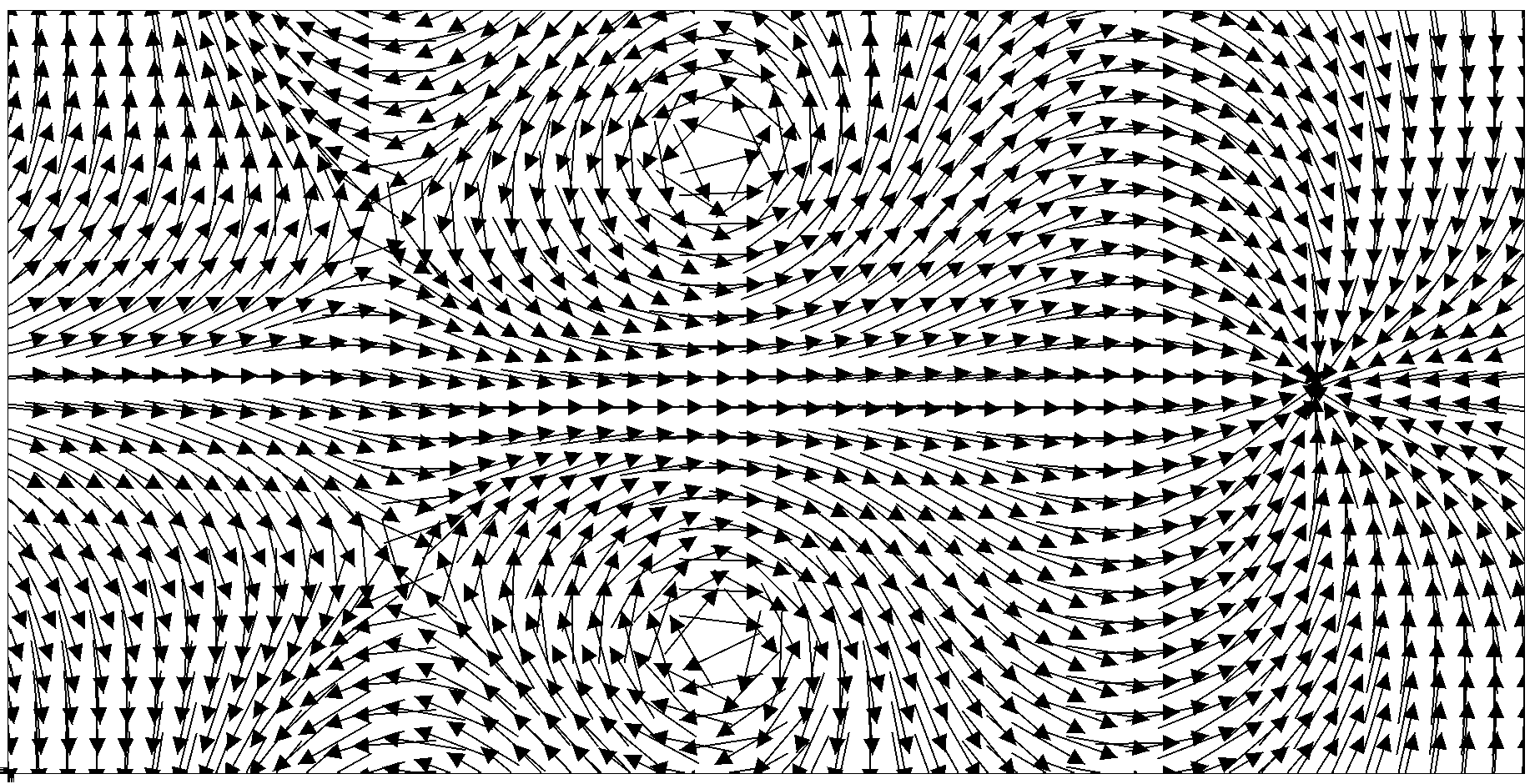}}

\subfigure[$t=50$]{\includegraphics[width=0.33\textwidth]{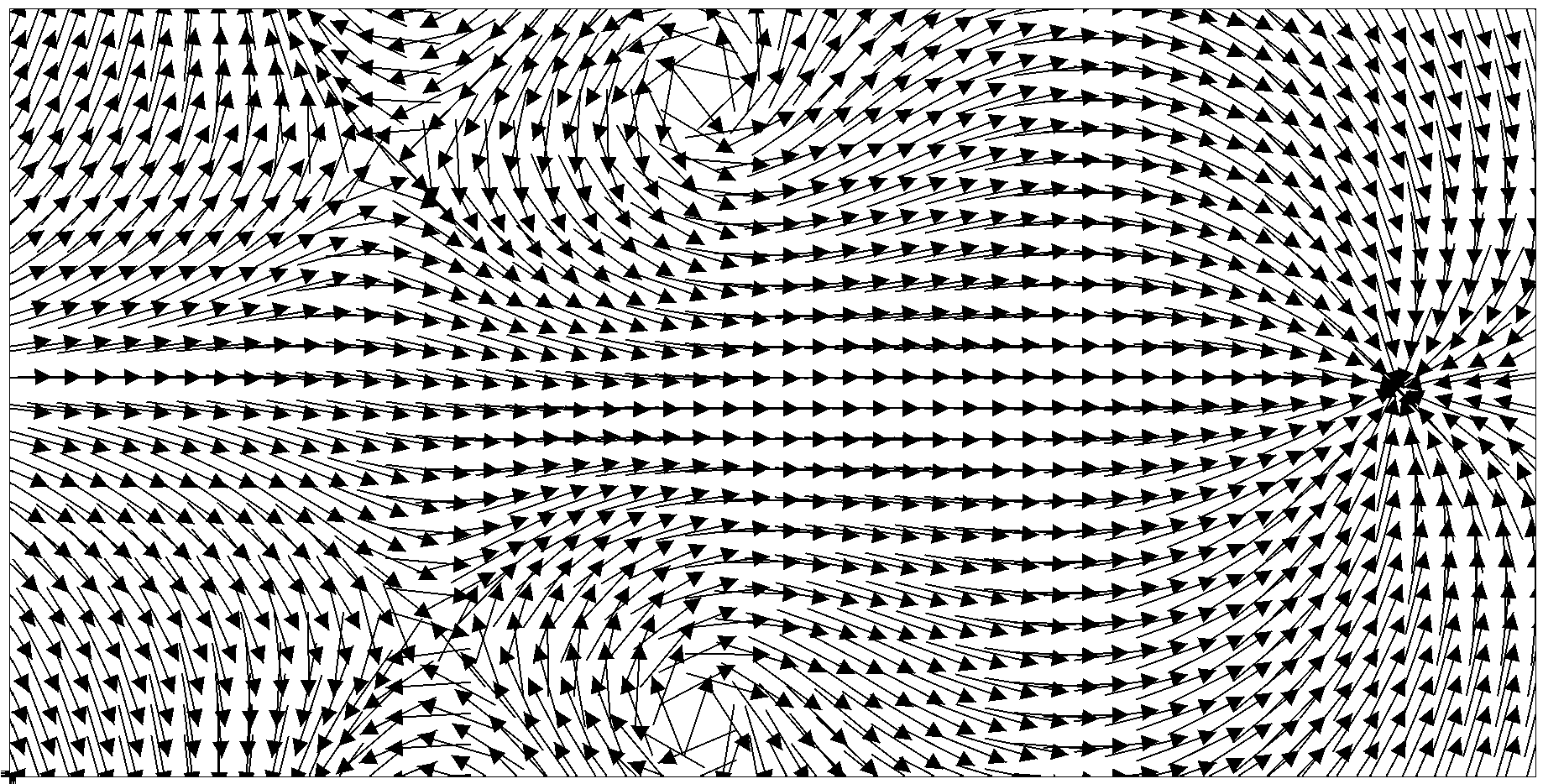}}
\subfigure[$t=60$]{\includegraphics[width=0.33\textwidth]{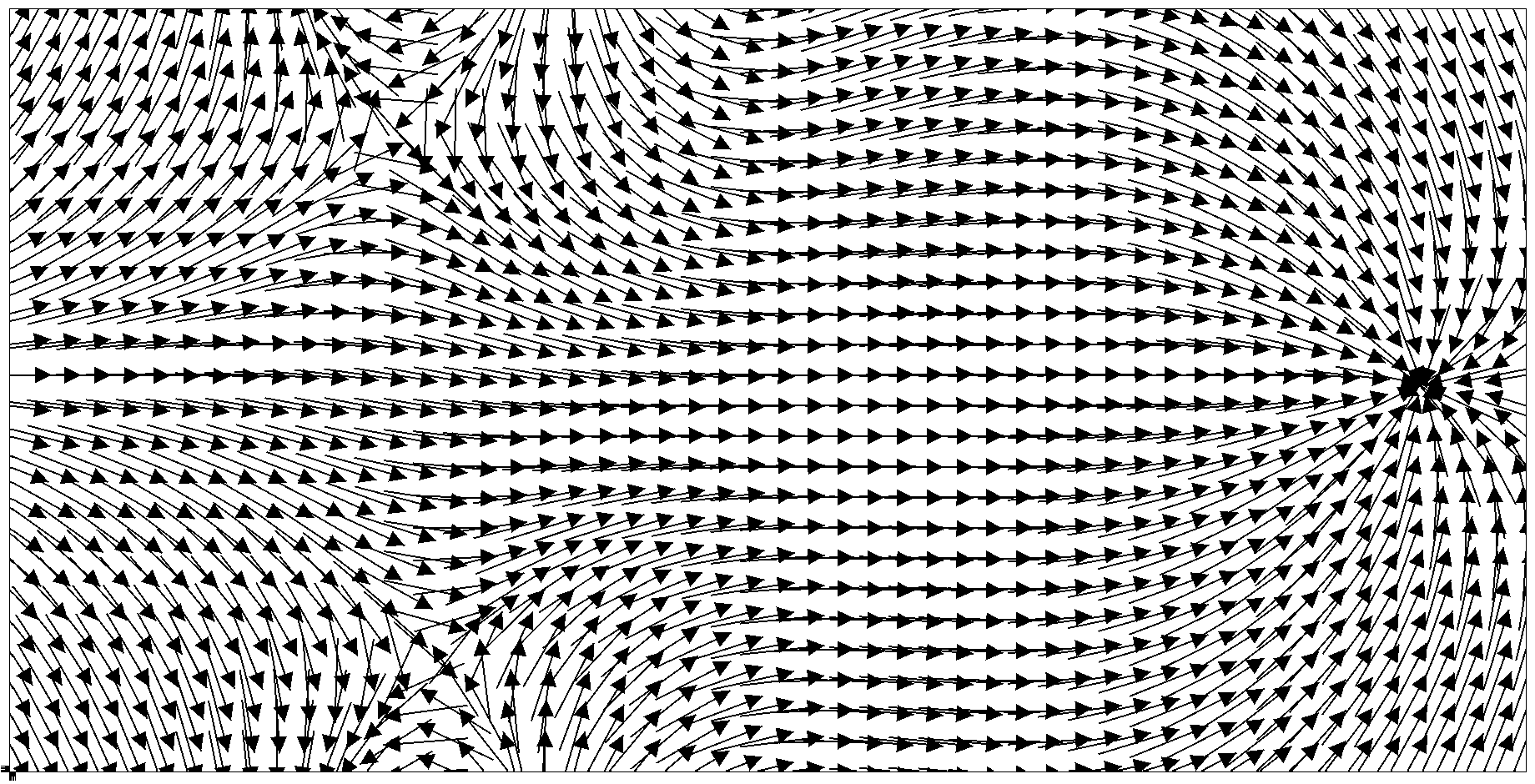}}
\subfigure[$t=75$]{\includegraphics[width=0.33\textwidth]{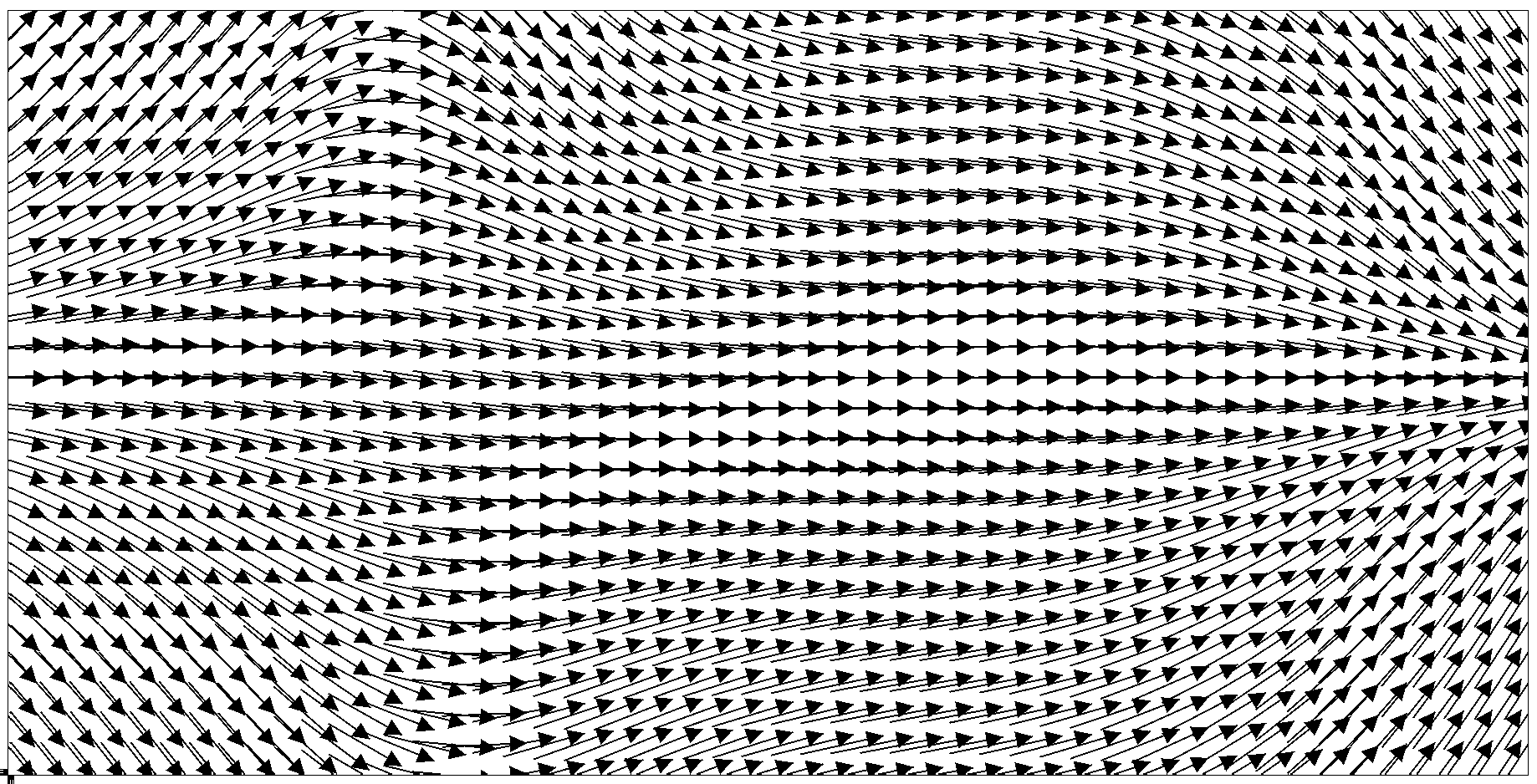}}

\caption{Liquid crystal defect dynamics that are driven by the Ericksen-Leslie hydrodynamic model without any anchoring at the boundaries. In this figure, the evolution dynamics of $\bp$ at different times are shown. The annihilations of defects are observed.}
\label{fig:LCP-bp}
\end{figure}

In addition, to further examine the defect dynamics, we also visualize the length of $\bp$ that are summarized in Figure \ref{fig:LCP-Defect}. It agrees with the qualitative observations for $\bp$ in Figure \ref{fig:LCP-bp}.

\begin{figure}
\subfigure[profiles of the  $\bp$ length at $t=0.5,2,5$]{
\includegraphics[width=0.33\textwidth]{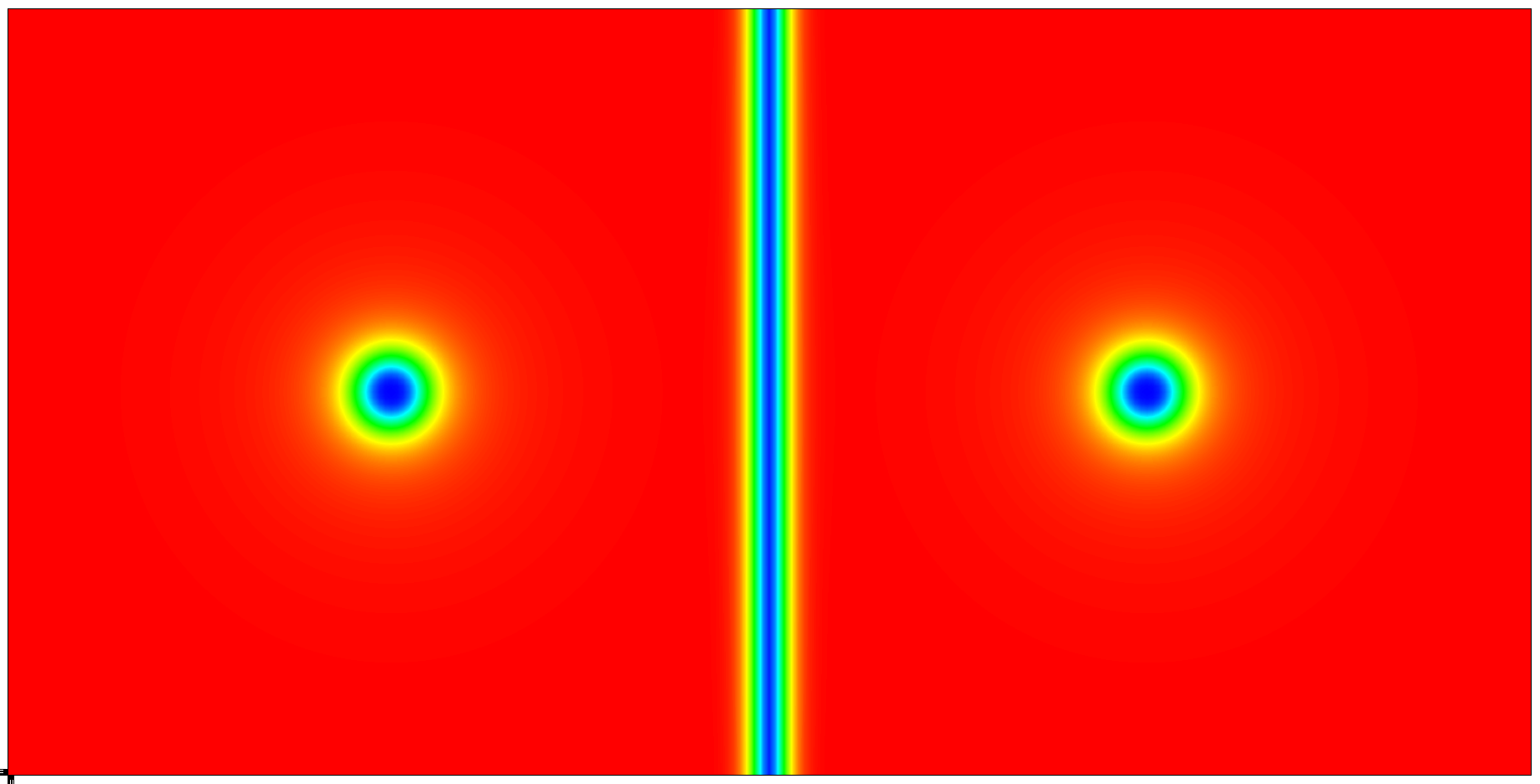}
\includegraphics[width=0.33\textwidth]{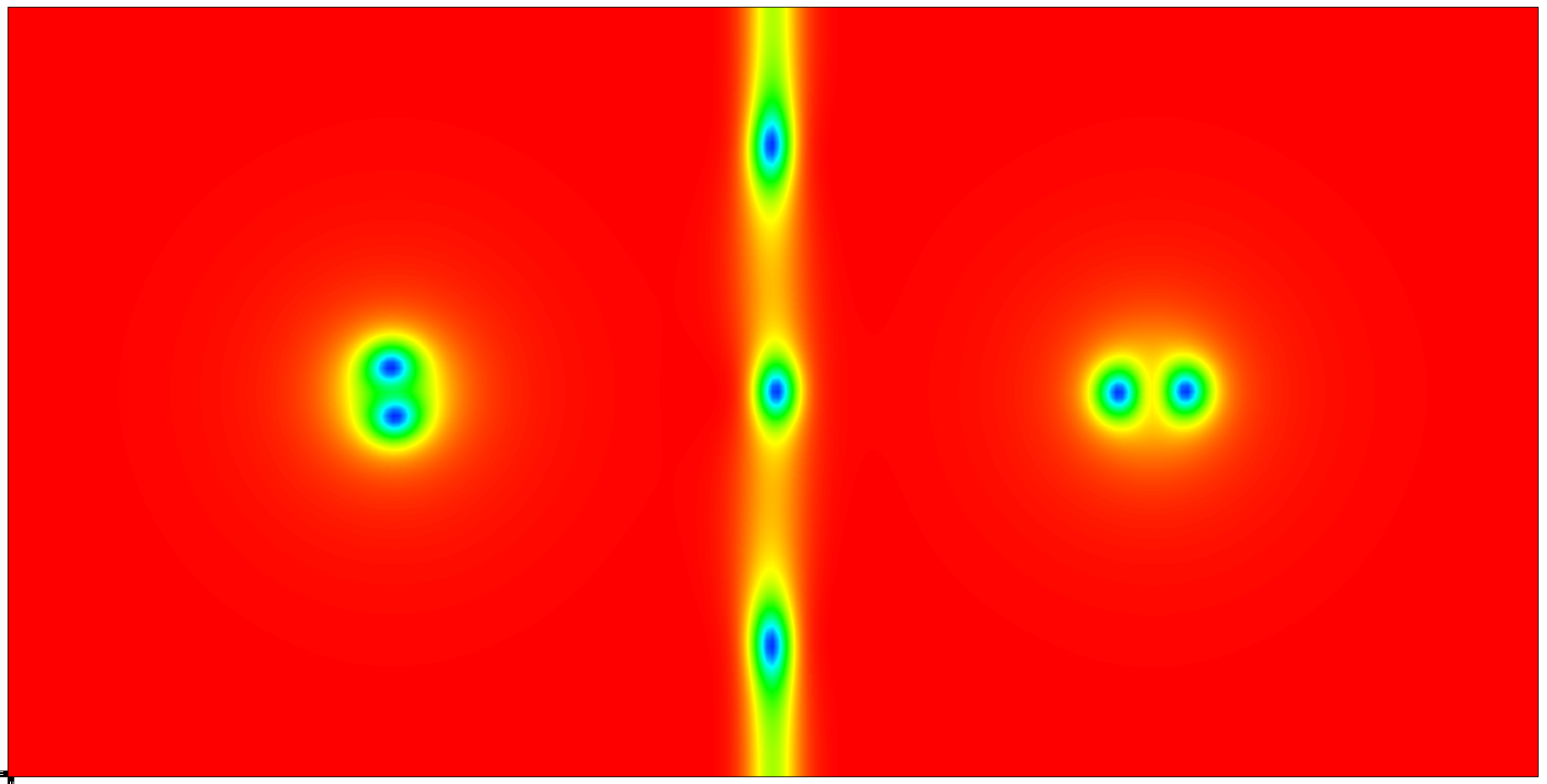}
\includegraphics[width=0.33\textwidth]{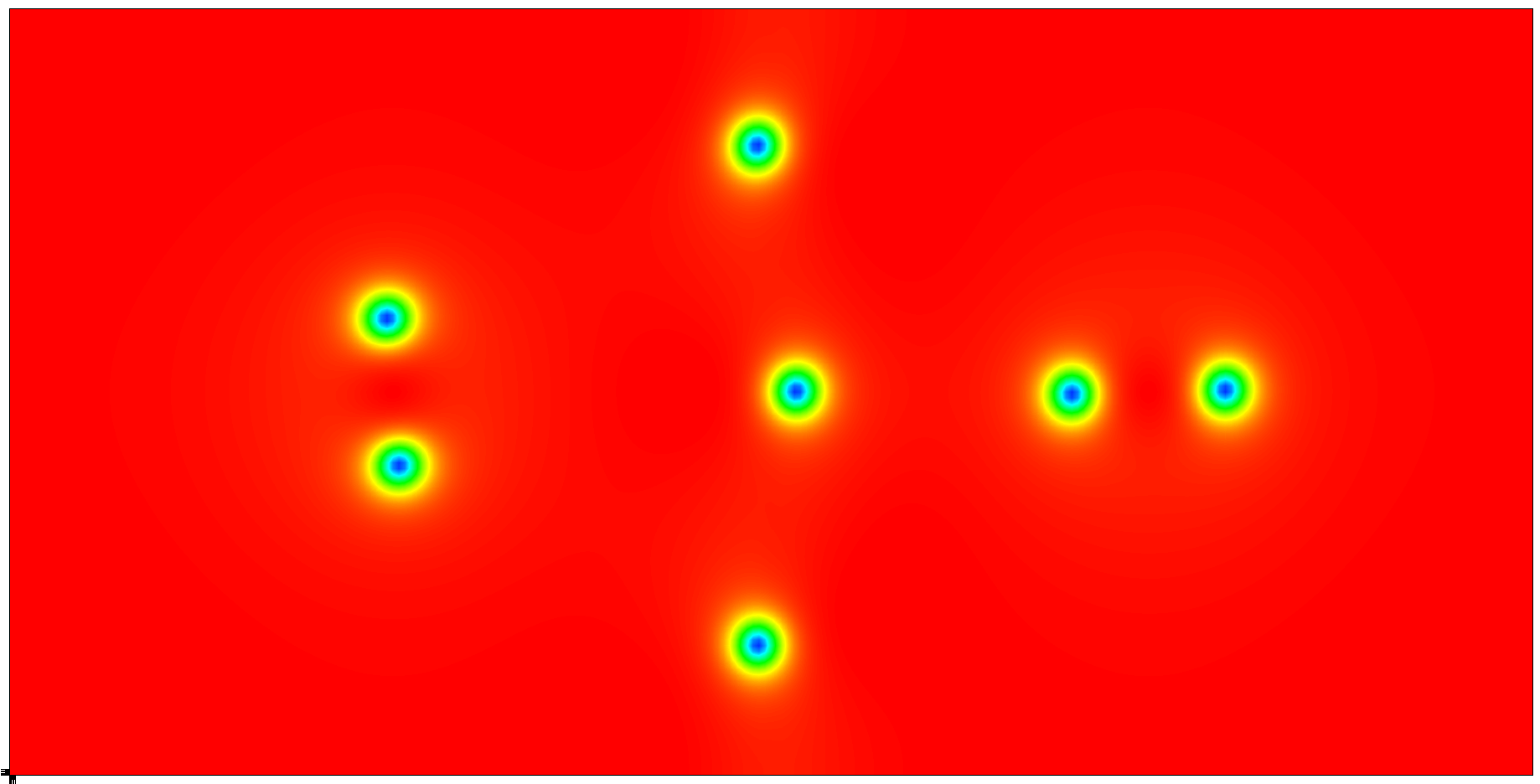}
}

\subfigure[profiles of the $\bp$ length at $t=17.5, 25$]{
\includegraphics[width=0.33\textwidth]{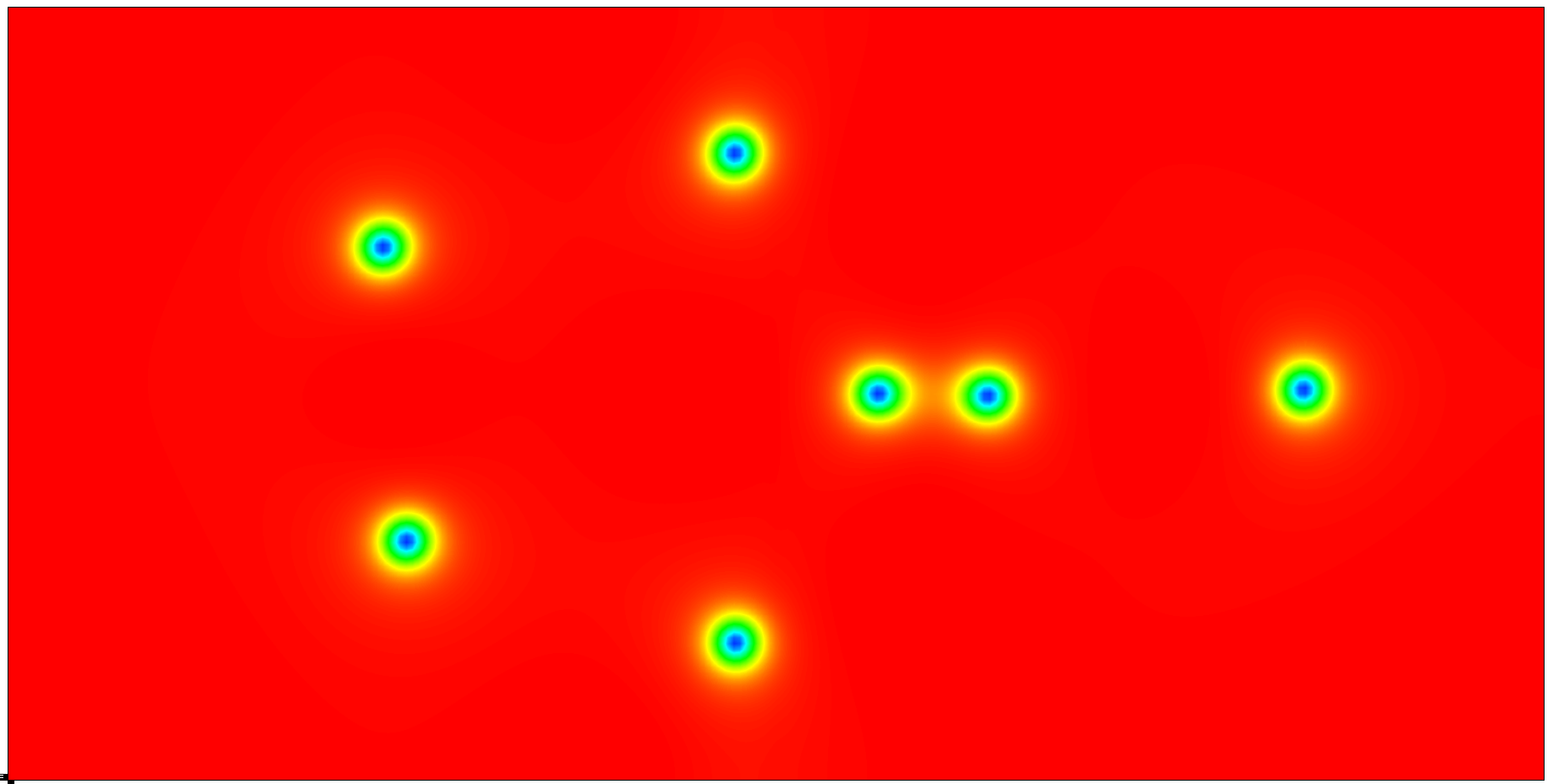}
\includegraphics[width=0.33\textwidth]{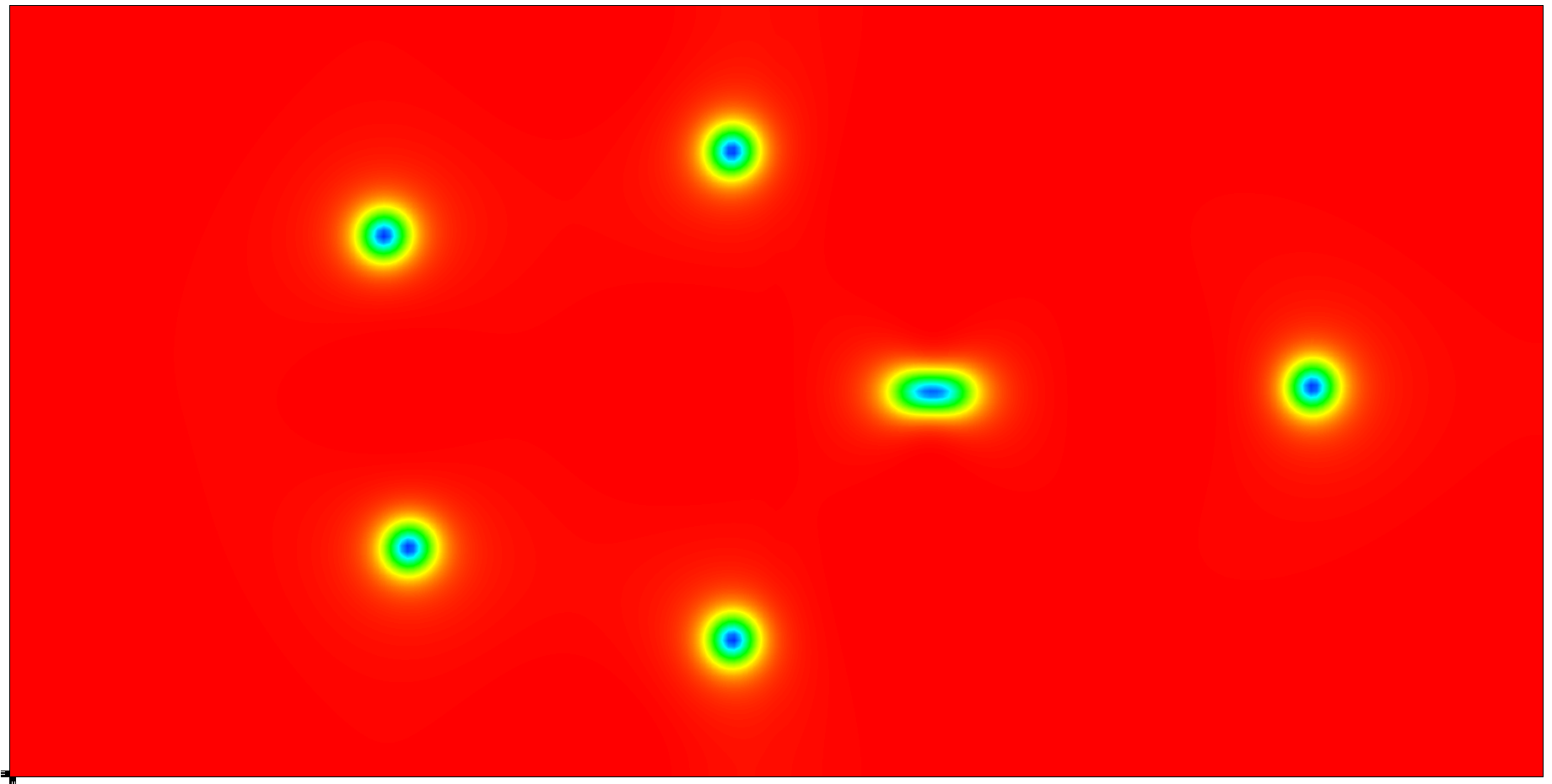}
\includegraphics[width=0.33\textwidth]{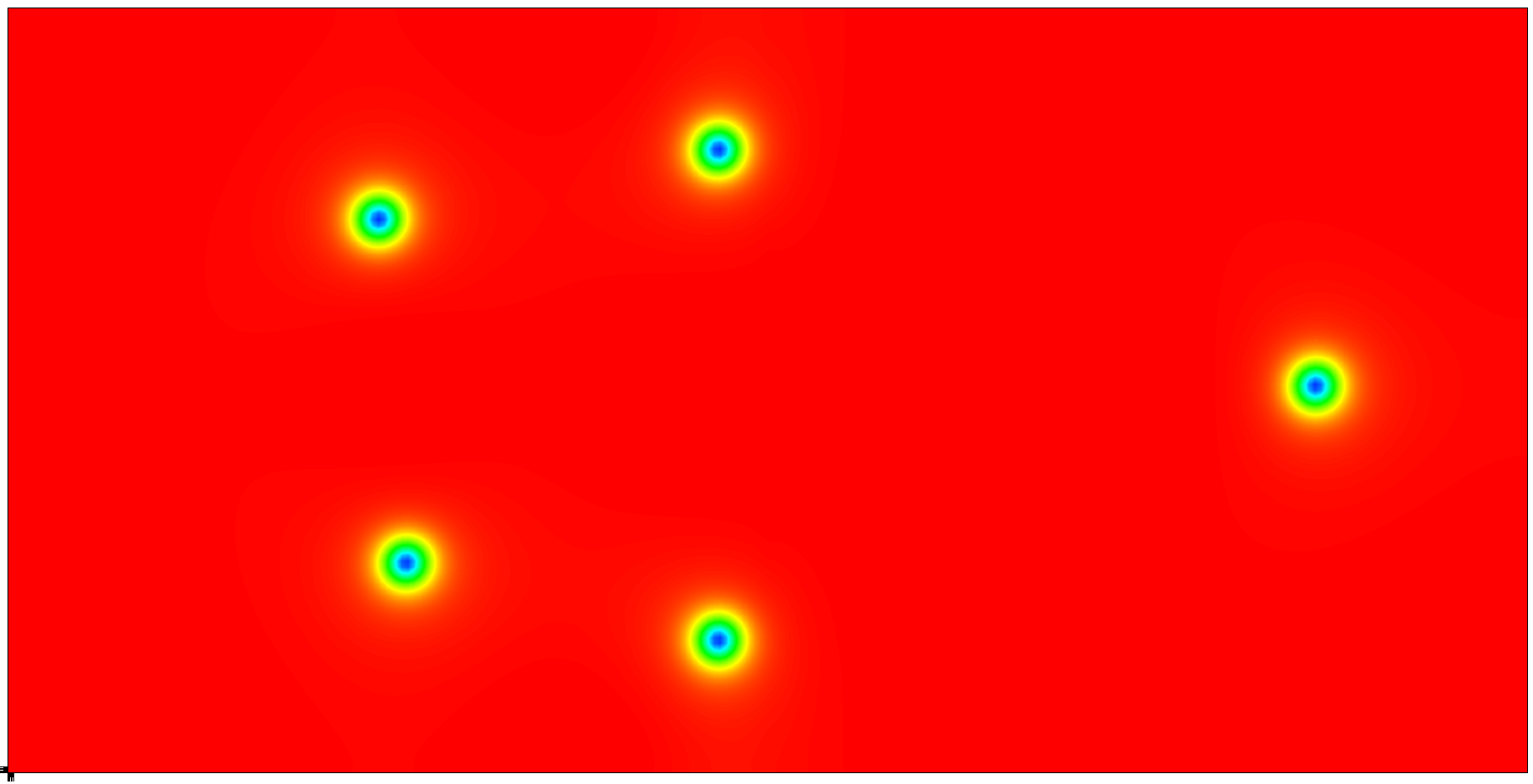}
}

\subfigure[profiles of the  $\bp$ length at $t=50,60,75$]{
\includegraphics[width=0.33\textwidth]{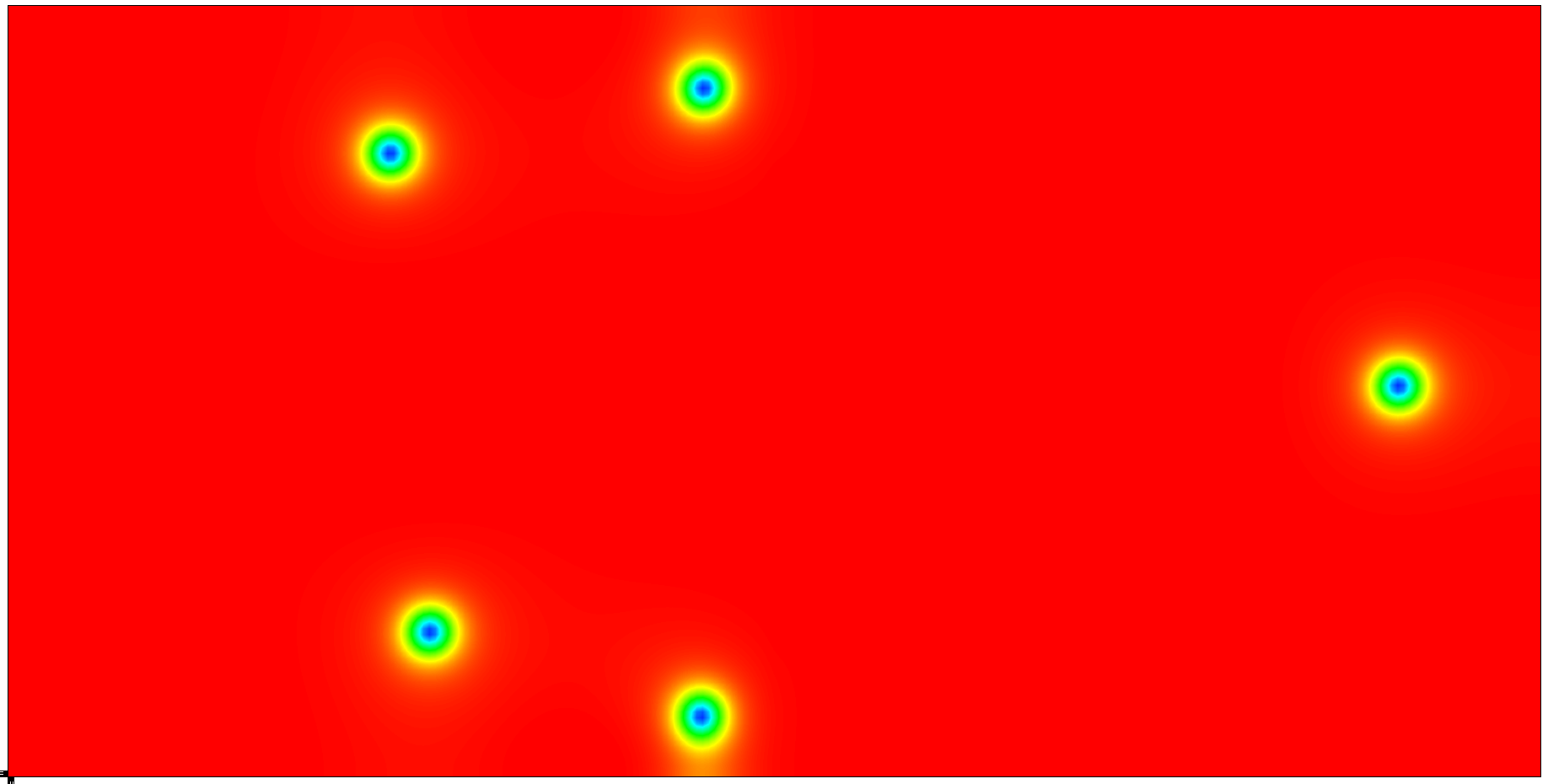}
\includegraphics[width=0.33\textwidth]{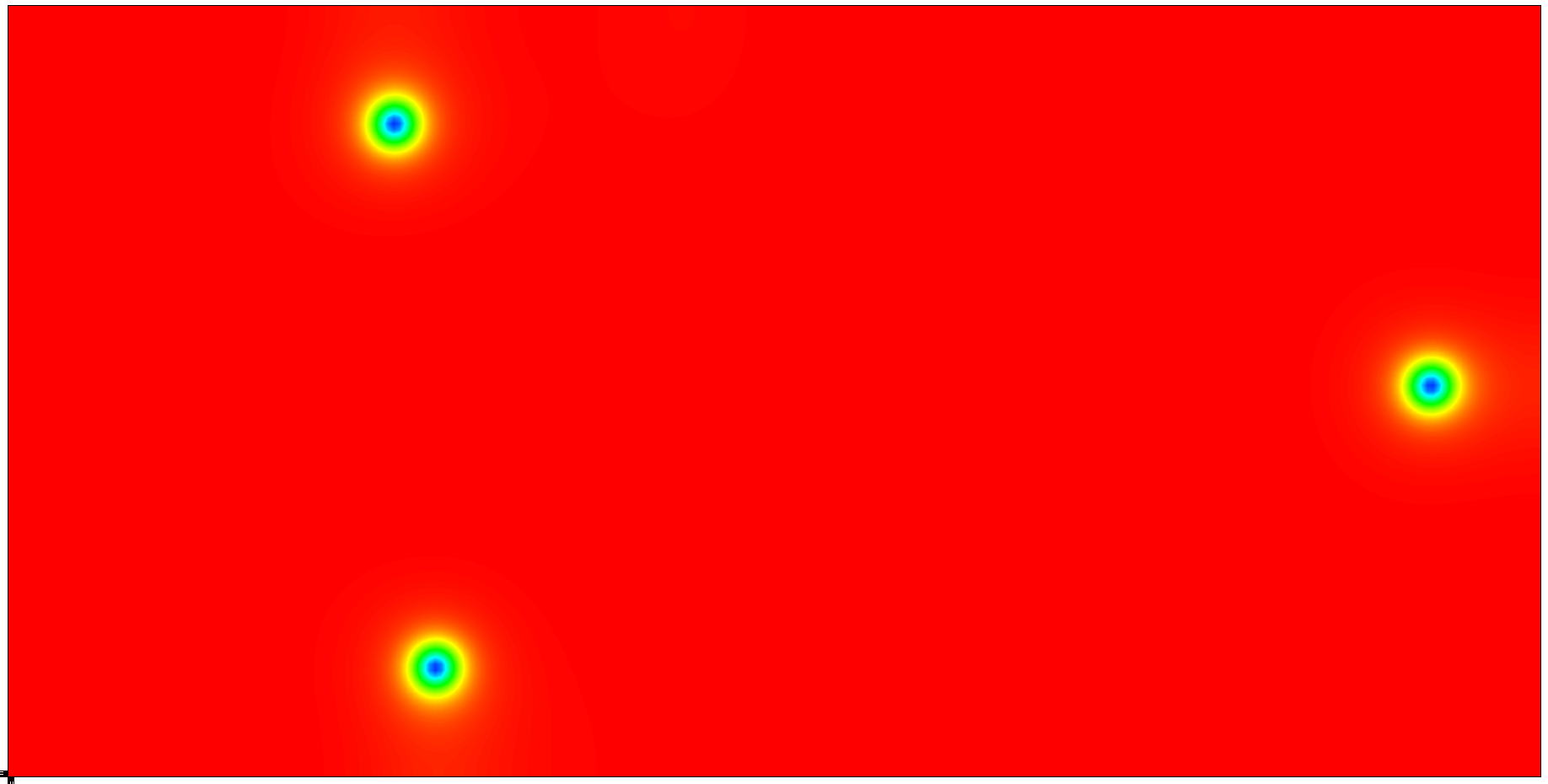}
\includegraphics[width=0.33\textwidth]{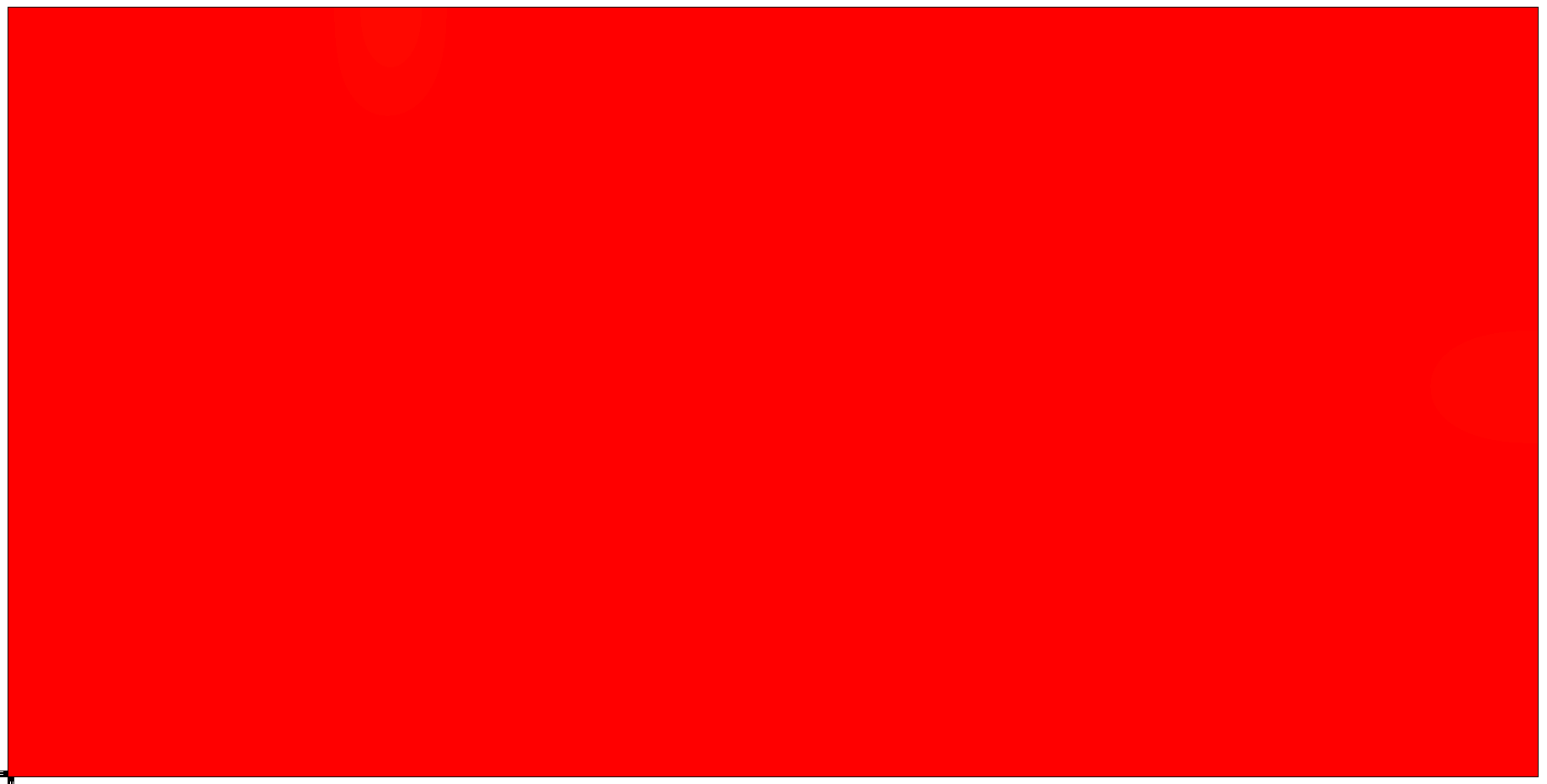}
}

\caption{Defect dynamics of the nematic liquid crystals driven by the Ericksen-Leslie hydrodynamic model. In this figure, the norm of $\bp$ is visualized for the simulation in Figure \ref{fig:LCP-bp}. Here the color red represents $1$, and the color blue represents $0$.}
\label{fig:LCP-Defect}
\end{figure}

To further verify the accuracy and energy stability of the proposed scheme, we also summarize the energy evolution dynamics in Figure \ref{fig:LCP-E-S}(a), and the time evolution of the auxiliary variable $s(t)$ in Figure \ref{fig:LCP-E-S}(b).  It has been observed that the energy is decreasing in time, and the numerical results of $s(t)$ are accurate.

\begin{figure}
\center
\subfigure[Energy evolution with time]{\includegraphics[width=0.45\textwidth]{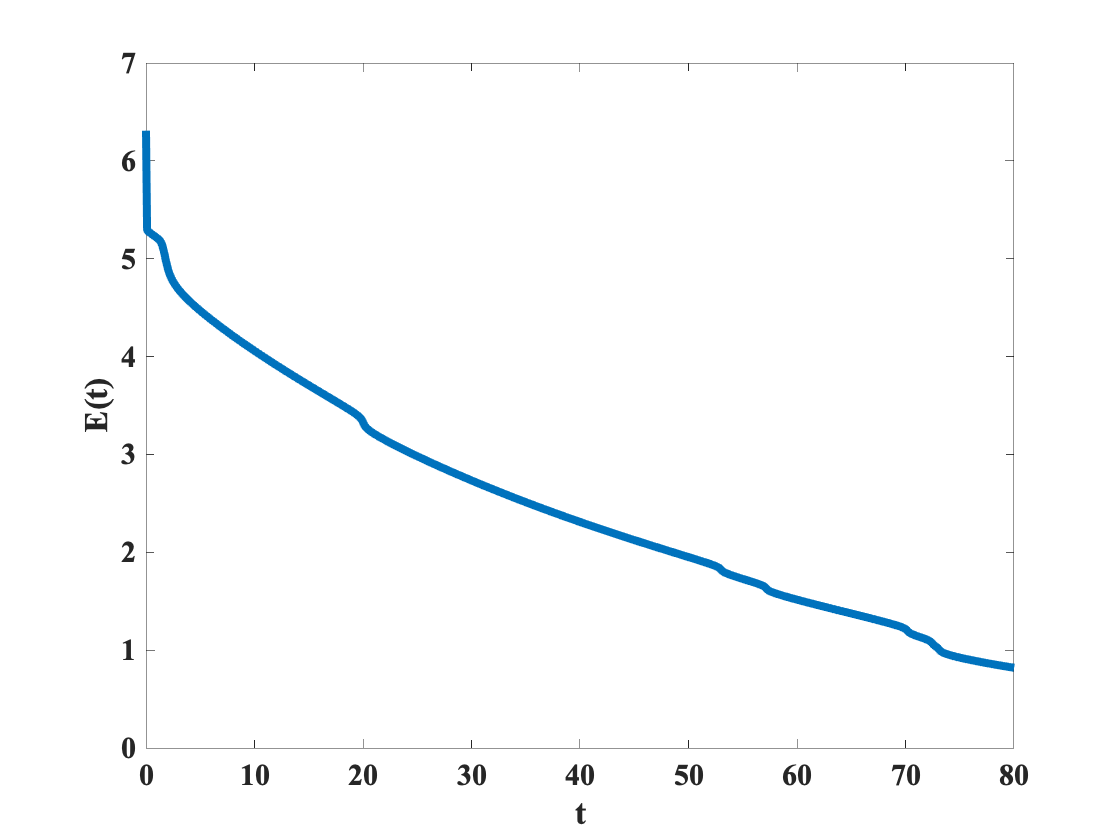}}
\subfigure[Auxiliary variable evolution with time] {\includegraphics[width=0.45\textwidth]{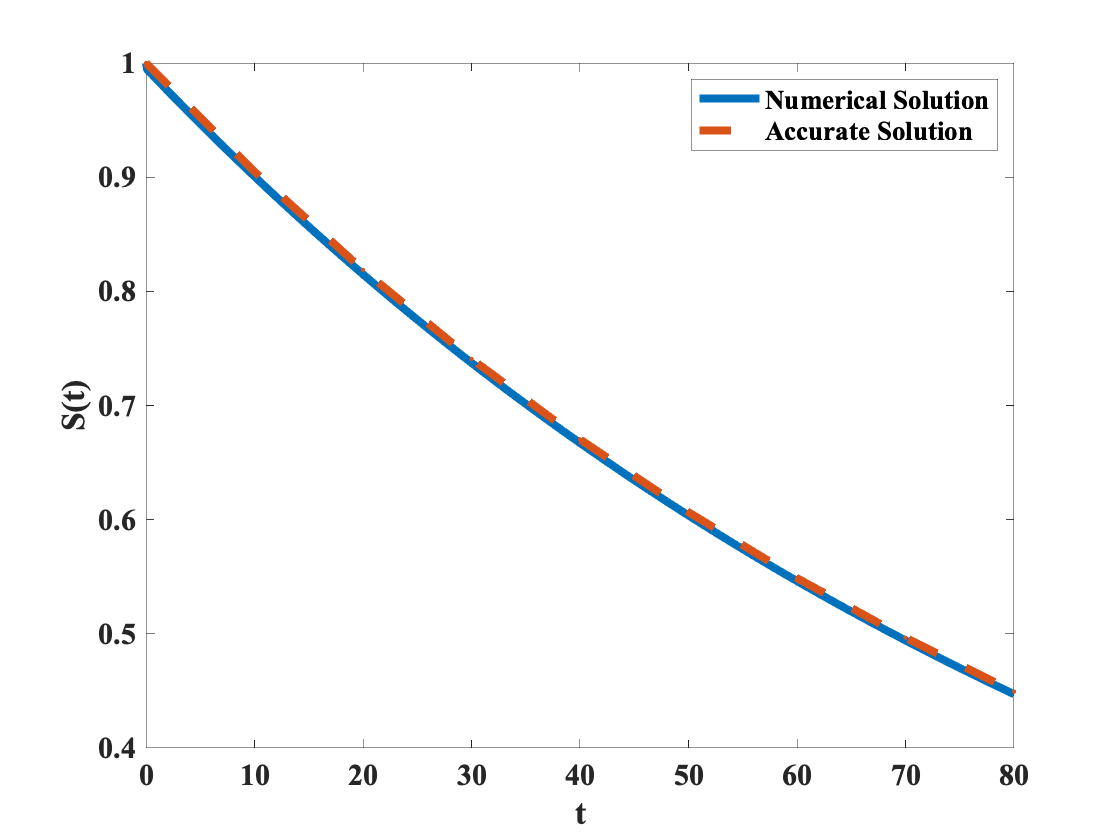}}
\caption{This figure shows the energy $\cE(t)$ and auxiliary variable $s(t)$ for the simulations in Figure \ref{fig:LCP-Defect}.  In (a), the numerical calculated energy is decreasing in time, which agrees well with the theoretical results for energy stability. In (b), the numerical solution of $s(t)$ accurately approximates its original definition $e^{-\frac{t}{T}}$.}
\label{fig:LCP-E-S}
\end{figure}

\section{Conclusion}
In this paper, we have developed a general numerical framework for designing linear, energy stable, and decoupled numerical algorithms for thermodynamically consistent models that can be cast in the generic Onsager form in \eqref{eq:evolution-general}. This framework's central idea is based on equivalent reformulation to unravel the intrinsic physical structures in the model, saying the reversible and irreversible dynamics. This guides us on algorithm design. Specifically, we utilize the energy quadratization (EQ) method to reformulate the Onsager form into the quadratized Qnsager form (which we named the Onsager-Q form). Then, we use the reversible-irreversible-dynamics (RID) method to decouple the reversible and irreversible dynamics. With the reformulated but equivalent form, we are able to introduce a sequence of semi-implicit schemes that have many advantageous properties. This numerical framework is then applied to several widely used incompressible hydrodynamic models. Specific numerical schemes for particular models are elaborated, followed by numerical examples. The second-order accuracy in time is verified through time mesh refinement as well.  Further applications of the general numerical framework on other types of thermodynamically consistent models will be discussed in our later research work.

\section*{Acknowledgments}
Jia Zhao would like to thank Prof. Qi Wang from the University of South Carolina for inspiring discussions on the generalized Onsager principles. 
Jia Zhao would like to acknowledge the support from National Science Foundation with grant NSF-DMS-1816783.
Jia Zhao would also like to acknowledge NVIDIA Corporation for the donation of a Quadro P6000 GPU for conducting some of the numerical simulations in this paper.

\bibliographystyle{unsrt}

\end{document}